\documentclass[twoside,a4paper]{amsart}
\usepackage[bookmarksnumbered,plainpages,driverfallback=dvipdfm]{hyperref}
\usepackage[english]{babel}
\usepackage[T1]{fontenc}
\usepackage{xcolor}
\usepackage{graphicx}
\usepackage{float}
\usepackage{amsfonts,amsmath,amsthm,amssymb,latexsym}
\usepackage{vmargin}
\usepackage{amscd}
\usepackage[all]{xy}
\usepackage[final]{showlabels} 
\usepackage{enumerate}
\usepackage{version}
\usepackage[normalem]{ulem}

\newtheorem{thm}{Theorem}[section]
\newtheorem{cor}[thm]{Corollary}
\newtheorem{lem}[thm]{Lemma}
\newtheorem{prop}[thm]{Proposition}
\theoremstyle{definition}
\newtheorem{defn}[thm]{Definition}
\newtheorem{es}[thm]{Example}

\newtheorem{rmk}[thm]{Remark}
\newtheorem{claim}[thm]{Claim}

\newcommand{\ev}{\mathrm{ev}}
\newcommand{\id}{\mathrm{Id}}

\newcommand{\co}{\mathrm{co}}
\newcommand{\op}{\mathrm{op}}
\newcommand{\Nat}{\mathrm{Nat}}
\newcommand{\Hom}{\mathrm{Hom}}
\newcommand{\End}{\mathrm{End}}
\newcommand{\cc}{\mathcal{C}}
\newcommand{\dd}{\mathcal{D}}
\newcommand{\e}{\mathcal{E}}
\newcommand{\f}{\mathcal{F}}
\newcommand{\G}{\mathfrak{G}}

\newcommand{\m}{\mathcal{M}}

\newcommand{\p}{\mathcal{P}}

\newcommand{\NN}{\mathbb{N}}
\newcommand{\QQ}{\mathbb{Q}}

\newcommand{\ZZ}{\mathbb{Z}}
\newcommand{\Set}{{\sf Set}}
\newcommand{\Comp}{{\sf Comp}}
\newcommand{\Top}{{\sf Top}}
\newcommand{\Dom}{{\sf Dom}}
\newcommand{\Field}{{\sf Field}}
\def\Alg{{\sf Alg}}
\def\Coalg{{\sf Coalg}}

\newenvironment{invisible}{{\noindent\sc \colorbox{yellow}{Invisible:}\;}\color{gray}}{\medskip}
\excludeversion{invisible}

\begin{document}
\title{Semiseparable functors}
\author{Alessandro Ardizzoni}
\address{%
\parbox[b]{\linewidth}{University of Turin, Department of Mathematics ``G. Peano'', via
Carlo Alberto 10, I-10123 Torino, Italy}}
\email{alessandro.ardizzoni@unito.it}
\urladdr{\url{www.sites.google.com/site/aleardizzonihome}}

\author{Lucrezia Bottegoni}
\address{%
\parbox[b]{\linewidth}{University of Turin, Department of Mathematics ``G. Peano'', via
Carlo Alberto 10, I-10123 Torino, Italy}}
\email{lucrezia.bottegoni@unito.it}

\thanks{This paper was written while the authors were members of the
"National Group for Algebraic and Geometric Structures and their
Applications" (GNSAGA-INdAM). They were partially supported by MIUR
within the National Research Project PRIN 2017. The authors would like to express their gratitude to Fosco Loregian, Claudia Menini, Paolo Saracco  and Joost Vercruysse for meaningful comments on a preliminary version of this paper.}

\begin{abstract}
 In this paper we introduce and investigate the notion of semiseparable functor. One of its first features is that it allows a novel description of separable and naturally full functors in terms of faithful and full functors, respectively. To any semiseparable functor we attach an invariant, given by an idempotent natural transformation, which controls when the functor is separable and yields a characterization of separable functors in terms of (dual) Maschke and conservative functors. We prove that any semiseparable functor admits a canonical factorization as a naturally full functor followed by a separable functor. Here the main tool is the construction of the coidentifier category attached to the associated idempotent natural transformation.
  Then we move our attention to the semiseparability of functors that have an adjoint. First we
  obtain a Rafael-type Theorem. Next we characterize the semiseparability of adjoint functors in terms of the (co)separability of the associated (co)monads and the natural fullness of the corresponding (co)comparison functor. We also focus on functors that are part of an adjoint triple. In particular, we describe bireflections as semiseparable (co)reflections, or equivalently, as either Frobenius or naturally full (co)reflections. As an application of our results, we study the semiseparability of functors traditionally attached to ring homomorphisms, coalgebra maps, corings and bimodules, introducing the notions of semicosplit coring and semiseparability relative to a bimodule which extend those of cosplit coring and Sugano's separability relative to a bimodule, respectively.
\end{abstract}
\maketitle
\tableofcontents

\section*{Introduction}
The notion of separable ring extension occurs in Algebra, Number Theory and Algebraic Geometry. In \cite{NVV89} C. N\v{a}st\v{a}sescu et al. reinterpreted this notion at a categorical level by introducing the so-called separable functors.
Explicitly, a functor $F:\mathcal{C}\to\mathcal{D}$ is said to be \emph{separable} if the associated natural transformation $\mathcal{F}^F_{X,Y} : \mathrm{Hom}_{\mathcal{C}}(X,Y)\rightarrow \mathrm{Hom}_{\mathcal{D}}(FX, FY)$, mapping $f$ to $Ff$, has a left inverse, i.e. there is a natural transformation $\mathcal{P}^F_{X,Y} : \mathrm{Hom}_{\dd}(FX, FY)\rightarrow \mathrm{Hom}_{\cc}(X,Y)$ such that $\mathcal{P}^F_{X,Y}\circ\mathcal{F}^F_{X,Y} = \mathrm{Id}_{\mathrm{Hom}_{\cc}(X,Y)}$ for all $X$ and $Y$ in $\cc$. 
A right version of this property yields to \emph{naturally full} functors, as defined in \cite{ACMM06}.
In this paper, we introduce the notion of \emph{semiseparable} functor, by requiring $\mathcal{F}^F_{X,Y}$ to be a regular natural transformation - an analogue of von Neumann regular element - i.e., by requiring $\mathcal{F}^F_{X,Y}$ to admit a natural transformation $\mathcal{P}^F_{X,Y}$ as above such that $\mathcal{F}^F_{X,Y}\circ\mathcal{P}^F_{X,Y}\circ\mathcal{F}^F_{X,Y} =\mathcal{F}^F_{X,Y}$.
Semiseparability allows to treat separability and natural fullness in a unified way and from a new perspective that reveals further features of them. For instance, it is well-known that a separable functor is faithful and that a naturally full functor is full: In Proposition \ref{prop:sep}, we see how the reverse implications hold by adding the assumption of semiseparability. 
In Proposition \ref{prop:idempotent}, to any semiseparable functor $F:\cc\to\dd$ we attach, in a unique way, a suitable idempotent natural transformation $e:\mathrm{Id}_{\cc}\rightarrow \mathrm{Id}_{\cc}$, which is trivial only in case $F$ is separable. Interestingly, as a  particular case, in Corollary \ref{cor:natfidp} we obtain that every naturally full functor admits such an idempotent natural transformation that is not trivial unless the functor is also separable whence fully faithful. By using the idempotent natural transformation $e$, in Corollary \ref{cor:conserv},  we prove that a functor is separable if and only if it is semiseparable and either Maschke, dual Maschke or conservative.
To such an $e$  we can attach a suitable quotient category $\cc_e$ of $\cc$, the so-called coidentifier \cite{FOPTST99}. This is the main ingredient to prove the notable property, stated in Theorem \ref{thm:coidentifier}, that any semiseparable functor $F:\cc\to \dd$ factors as the naturally full quotient functor $H:\cc\to \cc_e$ followed by a unique separable functor $F_e:\cc_e\to\dd$. As a consequence, in Corollary \ref{cor:fattoriz}, a functor is shown to be semiseparable if and only if it factors as a naturally full functor followed by a separable functor.


Next we investigate semiseparable functors which have a right (resp. left) adjoint. In this setting a celebrated result for separable functors is the so-called Rafael Theorem \cite{Raf90}, which provides a characterization of separability in terms of splitting properties of the (co)unit. It is natural to wonder whether such a result is also available for semiseparable functors, and in fact in Theorem \ref{thm:rafael} we prove that a functor which has a right (resp. left) adjoint is semiseparable if and only if the (co)unit of the adjunction is regular as a natural transformation.
Then we study semiseparability in the context of Eilenberg-Moore categories. Our main result here is Theorem \ref{thm:ssepMonad} stating that, given an adjunction $F\dashv G:\dd\to\cc$, the right adjoint $G$ is semiseparable if and only if the monad $GF$ is separable and the comparison functor $K_{GF}:\dd\to \cc_{GF}$ is naturally full, where  $\cc_{GF}$ is the Eilenberg-Moore category of modules over $GF$. A similar result for $F$ is given in Theorem \ref{thm:ssep-comonad}.
As a consequence, we recover similar characterizations for separable, naturally full and fully faithful functors.

Then, we focus on functors that have both a left and a right adjoint. It is well-known that in an adjoint triple $F\dashv G\dashv H$, the functor $F$ is fully faithful if and only if so is $H$. 
Proposition \ref{prop:adj-triples} shows that a similar behaviour holds for semiseparable, separable and naturally full functors. As far as we know, this result is new even at the separable and naturally full cases. Next, as a consequence of Rafael-type Theorem, in Proposition \ref{prop:frob-ssep} we obtain necessary and sufficient conditions for the semiseparability of a Frobenius functor.

We explore semiseparability in connection with functors admitting a fully faithful (left) right adjoint, which are known as  \emph{(co)reflections}, cf. \cite{Ber07}, and with functors admitting a fully faithful left and right adjoint equal and satisfying a coherence condition relating the unit and counit of the two adjunctions, which are called \emph{bireflections}, cf. \cite{FOPTST99}. 
Our main result in this direction is Theorem \ref{thm:frobenius} where we prove that a (co)reflection is semiseparable if and only if it is naturally full if and only if it is Frobenius if and only if it is a bireflection. In Proposition \ref{prop:Hcorefl} we see that, given a category $\cc$ and an idempotent natural transformation $e:\id_{\cc}\rightarrow \id_{\cc}$, the quotient functor $H:\cc\to \cc_e$ is a bireflection if and only if $e$ splits (e.g. $\cc$ is idempotent complete). As a consequence, in Corollary \ref{cor:fact-birefl} we show that a factorization of a semiseparable functor as a bireflection followed by a separable functor is available if and only if the associated idempotent natural transformation $e$ splits, and that such a factorization amounts to the canonical one given by the coidentifier category.

Finally, the results we obtained so far are applied to several functors traditionally connected to the study of separability.
The first functors we look at are the restriction of scalars functor $\varphi_*:S\text{-}\mathrm{Mod}\rightarrow R\text{-}\mathrm{Mod}$, whose semiseparability falls back to its separability, the extension of scalars functor $\varphi^*= S\otimes_{R}(-):R\text{-}\mathrm{Mod}\rightarrow S\text{-}\mathrm{Mod}$ and the coinduction functor $\varphi^!= {}_R\Hom(S,-):R\text{-Mod}\rightarrow S\text{-Mod}$ associated to a ring morphism $\varphi:R\to S$. These functors form an adjoint triple $\varphi^*\dashv\varphi_*\dashv\varphi^!$. In Proposition \ref{prop:inducfunc}, we characterize the semiseparability of $\varphi^*$, equivalent to that of $\varphi^!$, in terms of the regularity of $\varphi$ as a morphism of $R$-bimodules. Explicitly, the extension of scalars functor $\varphi^*$ is semiseparable if and only if there exists an $R$-bimodule map $E:S\to R$ such that $\varphi\circ E\circ\varphi =\varphi$, i.e., such that $\varphi E(1_S)=1_S$.
In a similar fashion, we investigate the semiseparability of the corestriction of coscalars functor $\psi_*:\m^C\to\m^D$ and of the coinduction functor $\psi^*:=(-)\square_D C:\m^D\to\m^C$ attached to a coalgebra morphism $\psi:C\to D$, obtaining in Proposition \ref{prop:coinduc-coalg} that $\psi^* $ is semiseparable if and only if $\psi$ is a regular morphism of $D$-bicomodules if and only if there is a $D$-bicomodule morphism $\chi:D\to C$ such that $\varepsilon_C \circ\chi\circ\psi=\varepsilon_C$.

The subsequent functor we investigate is the induction functor $G:=(-)\otimes_R\cc:\mathrm{Mod}\text{-}R\to \m^\cc$ attached to an $R$-coring $\cc$. Whereas an $R$-coring $\cc$ is sometimes called cosplit in the literature whenever $G$ is a separable functor, we say that $\cc$ is \emph{semicosplit} if $G$ is semiseparable. In Theorem \ref{thm:inducoring} we prove that $\cc$ is semicosplit if and only if the coring counit $\varepsilon_{\cc}:\cc\to R$ is regular as a morphism of $R$-bimodules if and only if there is an invariant element $z\in\cc^R= \{c\in\cc \mid rc=cr,\text{ } \forall r\in R\}$ such that $\varepsilon_{\cc}(z)\varepsilon_{\cc}(c)=\varepsilon_{\cc}(c)$ (i.e., such that $\varepsilon_{\cc}(z)c=c$), for every $c\in \cc$.

Next we consider the coinduction functor $\sigma_*=\Hom_S(M,-): \mathrm{Mod}\text{-}S\to \mathrm{Mod}\text{-}R$ associated to an $(R,S)$-bimodule $M$, together with its left adjoint $\sigma^*:=(-)\otimes_R M: \text{Mod-}R\to \text{Mod-}S$. As we will prove in Theorem \ref{thm:bimod}, the semiseparability of this functor, which results to be equivalent to the fact that the evaluation map $\ev_M:M^*\otimes_R M\to S$ is regular as a morphism of $S$-bimodules and $M\otimes_S\ev_M$ is surjective, can be also completely described in terms of a property of $M$ that led us to introduce the \emph{$M$-semiseparability over $R$} for the ring $S$, an extension of $M$-separability investigated by Sugano in \cite{Su71}. In Corollary \ref{cor:sep-bimod} we provide the following characterization: $S$ is $M$-separable over $R$ if and only if $S$ is $M$-semiseparable over $R$ and $M$ is a generator in $\mathrm{Mod}\text{-}S$. A different characterization of $M$-semiseparability of $S$ over $R$ is obtained in Proposition \ref{prop:ssVSs}. This allows us to exhibit in Example \ref{es:ssVSs} an instance where $S$ is $M$-semiseparable but not $M$-separable over $R$. 
As in the separable case, if we add the assumption that $M$ is finitely generated and projective as a right $S$-module, then the (co)monad associated to the adjunction $(\sigma^*,\sigma_*)$ can be described in a easier way. This allows to achieve further characterizations of the semiseparability of $\sigma_*$ and $\sigma^*$ in Proposition \ref{prop:right-equiv-bimod} and Proposition \ref{prop:left-equiv-bimod}, respectively. Moreover, in Proposition \ref{prop:phi*}, Corollary \ref{cor:inducoring}, and Proposition \ref{prop:condcMfact}, an explicit factorization, as a bireflection followed by a separable functor, is provided for the above functors $\varphi^*$, $G$, and $\sigma_*$, respectively, when they are semiseparable.
Finally, in Theorem \ref{thm:rightHopf} we study the semiseparability of the coinvariant functor $\left( -\right)^{\co B}:\mathfrak{M}_{B}^{B}\rightarrow \mathfrak{M}$ attached to a bialgebra $B$ over a field $\Bbbk$, where $\mathfrak{M}_{B}^{B}$ and $\mathfrak{M}$ denote the category of right Hopf modules over $B$ and the category of $\Bbbk$-vector spaces, respectively. Explicitly, $\left( -\right)^{\co B}$ is semiseparable if and only if $B$ is a right Hopf algebra with anti-multiplicative and anti-comultiplicative right antipode.\medskip

The organization of the paper is the following. In Section \ref{sect:notion-semisep} we introduce the notion of semiseparable functor and we investigate its interaction with separable and naturally full functors, relative separable functors and composition. We show that any semiseparable functor admits the associated idempotent natural transformation and that it factors as a naturally full functor followed by a separable one. We also see how the existence of a suitable type of generator within its source category implies that a functor is semiseparable if and only if it is separable.
Section \ref{sect:semisep-adjunct} collects results on semiseparable functors that have an adjoint. Explicitly, we obtain a Rafael-type theorem for semiseparable functors, we study the behaviour of semiseparable adjoint functors in terms of (co)monads and the associated (co)comparison functor, we investigate functors that have both a (possibly equal or fully faithful) left and right adjoint.
In Section \ref{sect:applications} we test the notion of semiseparability on relevant functors attached to ring homomorphisms, coalgebra maps, corings, bimodules and Hopf modules. This section closes with a discussion on (co)reflections that provides an overview of the notions considered in this work highlighting their mutual interaction.


\subsection{Preliminaries and notations}
Given an object $X$ in a category $\cc$, the identity morphism on $X$ will be denoted either by $\id_X$ or $X$ for short. For categories $\cc$ and $\dd$, a functor $F:\cc\to \dd$ just means a covariant functor. By $\id_{\cc}$ we denote the identity functor on $\cc$. For any functor $F:\cc\to \dd$, we denote $\id_{F}:F\to F$ (or just $F$, if no confusion may arise) the natural transformation defined by $(\id_{F})_X:=\id_{FX}$.

Let $\cc$ be a category. Denote by $\cc^{\op }$ the opposite category of $\cc$. An object $X$ and a morphism $f:X\rightarrow Y$ in $%
\cc$ will be denoted by $X^{\op }$ and $f^{\op }:Y^{%
\op }\rightarrow X^{\op }$ respectively when regarded as an object and a
morphism in $\cc^{\op }$. Given a functor $F:\cc\to \dd$, one defines its opposite functor $F^\op :\cc^\op \to \dd^\op $ by setting $F^\op X^\op =(FX)^\op $ and $F^\op f^\op =(Ff)^\op $. 

A morphism (natural transformation) $f$ is called \emph{regular}\footnote{This terminology is derived from von Neumann regularity of rings. See also \cite[Subsection 2.1]{Wisb13}.} provided there is a morphism (resp. natural transformation) $g$ with $f\circ g\circ f=f$. Note that the asymmetry of this definition is apparent as $g$ could be replaced by $g'=g\circ f\circ g$ in such a way that $f\circ g'\circ f=f$ and $g'\circ f\circ g'=g'$.

By a ring we mean a unital associative ring.

\section{The notion of semiseparable functor}\label{sect:notion-semisep}
In this section we introduce and investigate the notion of semiseparable functor. Subsection \ref{sub:semis} presents its definition and characterizes the known notions of separable and naturally full functors in terms of it. In Subsection \ref{sub:assosidmp} we attach an invariant to any semiseparable functor, that we call the associated idempotent natural transformation, which controls the separability of the functor and allows a characterization of separable functors in terms of (dual) Maschke and conservative functors. In Subsection \ref{sub:relsep} the connection between semiseparable and relative separable functors is explored.
Subsection \ref{subsect:comp} concerns the behaviour of semiseparable functors with respect to composition. Subsection \ref{sub:coidentifier} shows how semiseparable functors admit a canonical factorization as a naturally full functor followed by a separable one. Here the main tool is the construction of the coidentifier category attached to the associated idempotent natural transformation. In Subsection \ref{generators} we investigate under which conditions the existence of a suitable type of generator within its source category implies that a functor is semiseparable if and only if it is separable.

\subsection{Semiseparable functors}\label{sub:semis}
Let $F: \cc \rightarrow \dd$ be a functor and consider the associated natural transformation
\begin{equation*}\label{nat_transf}
\f^F : \Hom_{\cc}(-,-)\rightarrow \Hom_{\dd}(F-, F-),
\end{equation*}
defined by setting $\f^F_{C,C'}(f)= F(f)$, for any $f:C\rightarrow C'$ in $\cc$. We recall that $F$ is said to be \emph{separable} \cite{NVV89} if there is a natural transformation
$\p^F : \Hom_{\dd}(F-, F-)\rightarrow \Hom_{\cc}(-,-)$
such that $\p^F\circ\f^F = \id _{\Hom_{\cc}(-,-)}$.
Similarly, a functor $F: \cc \rightarrow \dd$ is called \emph{naturally full} \cite{ACMM06} if there exists a natural transformation $\p^F : \Hom_{\dd}(F-, F-)\rightarrow \Hom_{\cc}(-,-)$ such that $\f^F\circ\p^F = \id _{\Hom_{\dd}(F-,F-)}$.

Clearly, a functor is fully faithful if and only if it is both separable and naturally full.

\begin{defn}\label{semisep}
We say that a functor $F: \cc \rightarrow \dd$ is \emph{\textbf{semiseparable}} if the natural transformation $\f^F$ is regular, i.e. if
there exists a natural transformation $\p^F : \Hom_{\dd}(F-, F-)\rightarrow \Hom_{\cc}(-,-)$ such that
$\f^F\circ\p^F\circ\f^F = \f^F$.
\end{defn}

%

\begin{rmk}\label{rmk:opposemi}
  Since $\f^F_{X,Y}=\f^{F^\op}_{Y^\op,X^\op}$ it is clear that a functor $F:\cc\to\dd$ is semiseparable (resp. separable, naturally full, full, faithful, fully faithful) if and only if so is $F^\op:\cc^\op\to\dd^\op$.
\end{rmk}

It is well-known that a separable functor is faithful and that a naturally full functor is full. Let us see how adding the notion of semiseparable functor to the picture allows us to turn these implications into equivalences.

\begin{prop}\label{prop:sep} Let $F: \cc \rightarrow \dd$ be a functor. Then,
\begin{itemize}
\item[(i)]$F$ is separable if and only if $F$ is semiseparable and faithful;
\item[(ii)]$F$ is naturally full if and only if $F$ is semiseparable and full.
\end{itemize}
\end{prop}
\proof We only prove (i), the proof of (ii) being similar. Assume that $F$ is separable. From $\p^F\circ\f^F = \id $ it follows that $\f^F\circ \p^F\circ\f^F = \f^F\circ\id=\f^F$, i.e. $F$ is semiseparable, and that for all $C,C'\in\cc$, the map $\f^F _{C,C'}$ is injective, i.e. $F$ is faithful. Conversely, if $F$ is semiseparable, we have that there exists a natural transformation $\p^F $ such that $\f^F\circ\p^F\circ\f^F = \f^F$, hence, if $F$ is faithful, $\p^F\circ\f^F = \id $, as $\f^F$ is injective on components.
\begin{invisible}
  If $F$ is naturally full, then, for all $C,C'\in\cc$, the map $\f^F _{C,C'}: \Hom_{\cc}(C,C')\rightarrow \Hom_{\dd}(FC,FC')$ is surjective, since it has a right inverse, and it follows that $F$ is a full functor. $F$ is also semiseparable, as $(\f^F\circ\p^F )\circ\f^F = \id _{\Hom_{\dd}(F-,F-)}\circ\f^F =\f^F$.\par  Conversely, if $F$ is semiseparable, there is a natural transformation $\p^F : \Hom_{\dd}(F-, F-)\rightarrow \Hom_{\cc}(-,-)$ such that $\f^F\circ\p^F\circ\f^F = \f^F$. So, if $F$ is full, then $\f^F\circ\p^F = \id _{\Hom_{\dd}(F-,F-)}$, as $\f^F$ is surjective.
\end{invisible}
\endproof

In view of Proposition \ref{prop:sep}, both separable and naturally full functors are instances of semiseparable functors.
Next aim is to endow any semiseparable functor with an invariant that will play a central role in our treatment.

\subsection{The associated idempotent}\label{sub:assosidmp}
Here we attach, in a canonical way, a suitable idempotent natural transformation to any semiseparable functor.

\begin{prop}\label{prop:idempotent}
 Let $F:\cc\rightarrow \dd$ be a semiseparable functor. Then
there is a unique idempotent natural transformation $e:\id_{\cc%
}\rightarrow \id_{\cc}$ such that $Fe=\id_{F}$ with the following universal property: if $f,g:A\to B$ are morphisms, then $Ff=Fg$ if and only if $e_B\circ f=e_B\circ g$.
 \end{prop}

 \begin{proof}
Since $F$ is semiseparable, there is a natural transformation $\mathcal{P}^F$
such that $\mathcal{F}^F\circ \mathcal{P}^F\circ \mathcal{F}^F=\mathcal{F}^F$. Set $%
e_{X}:=\mathcal{P}^F_{X,X}\left( \id_{FX}\right) $. Note that $Fe_{X}=F%
\mathcal{P}^F_{X,X}\left( \id_{FX}\right) =\mathcal{F}^F_{X,X}\mathcal{P}^F%
_{X,X}\mathcal{F}^F_{X,X}\left( \id_{X}\right) =\mathcal{F}^F%
_{X,X}\left( \id_{X}\right) =\id_{FX}$. Thus $e_{X}\circ
e_{X}=\mathcal{P}^F_{X,X}\left( \id_{FX}\right) \circ e_{X}=\mathcal{P}^F%
_{X,X}\left( \id_{FX}\circ Fe_{X}\right) =\mathcal{P}^F_{X,X}\left(
\id_{FX}\right) =e_{X}$ and hence $e_{X}$ is idempotent. Moreover,
for every morphism $f:X\rightarrow Y$ we have $f\circ e_{X}=f\circ \mathcal{P%
}^F_{X,X}\left( \id_{FX}\right) =\mathcal{P}^F_{X,Y}\left( Ff\circ
\id_{FX}\right) =\mathcal{P}^F_{X,Y}\left( \id_{FY}\circ
Ff\right) =\mathcal{P}^F_{Y,Y}\left( \id_{FY}\right) \circ
f=e_{Y}\circ f$ so that $f\circ e_{X}=e_{Y}\circ f$, i.e. $e=\left(
e_{X}\right) _{X\in \cc}:\id_{\cc}\rightarrow
\id_{\cc}$ is an idempotent natural transformation such that
$Fe=\id_{F}$. Now, consider morphisms $f,g:A\to B$. If $Ff=Fg$, then $\mathcal{P}^F%
_{A,B}\left( Ff\right) =\mathcal{P}^F_{A,B}\left( Fg\right) $ i.e. $\mathcal{P}^F%
_{B,B}\left( \id_{FB}\right) \circ f=\mathcal{P}^F_{B,B}\left( \mathrm{%
Id}_{FB}\right) \circ g$, i.e. $e_{B}\circ f=e_{B}\circ g$. Conversely, from $e_{B}\circ f=e_{B}\circ g$ we get $Fe_{B}\circ Ff=Fe_{B}\circ Fg$ and hence $Ff=Fg$ as $Fe=\id_{F}$. Finally, let $e':\id_{\cc%
}\rightarrow \id_{\cc}$ be an idempotent natural transformation such that, if $f,g:A\to B$ are morphisms, then $Ff=Fg$ if and only if $e'_B\circ f=e'_B\circ g$. From $e'_X\circ e'_X=e'_X\circ \id_X$ we get $Fe'_X=F\id_X$ (whence $Fe'=\id_F$). From the universal property of $e$ we get $e_X\circ e'_X=e_X\circ \id_X$ i.e. $e_X\circ e'_X=e_X$. If we interchange the roles of $e$ and $e'$, in a similar way we get $e'_X\circ e_X=e'_X$. By naturality we have $e_X\circ e'_X=e'_X\circ e_X$ whence $e_X=e'_X$, i.e. $e=e'$.
 \end{proof}

\begin{defn}
  The idempotent natural transformation $e:\id_\cc\to\id_\cc$ we have attached to a semiseparable functor  $F:\cc\to\dd$ in Proposition \ref{prop:idempotent} 
  will be called the \textbf{associated idempotent natural transformation}. Thus $e$ is defined on components by $e_{X}:=\mathcal{P}^F_{X,X}\left( \id_{FX}\right) $ where $\mathcal{P}^F$ is any natural transformation such that $\mathcal{F}^F\circ \mathcal{P}^F\circ \mathcal{F}^F=\mathcal{F}^F$.
\end{defn}

Since naturally full functors are in particular semiseparable, we get the following result which was unknown before to the best of our knowledge.
\begin{cor}\label{cor:natfidp}
Any naturally full functor admits the associated idempotent natural transformation.
\end{cor}

We now see how the idempotent natural transformation associated to a semiseparable functor controls its separability.

\begin{cor}\label{cor:esep}
Let $F:\cc\to\dd$ be a semiseparable functor and let $e:\id_\cc\to\id_\cc$ be the associated idempotent natural transformation. Then, $F$ is separable if and only if $e=\id$.
\end{cor}

\begin{proof}
By construction $e_X=\mathcal{P}^F_{X,X}(\id_{FX})$ where $\mathcal{P}^F$ is a natural transformation such that $\mathcal{F}^F\circ\mathcal{P}^F\circ \mathcal{F}^F=\mathcal{F}^F.$ If $F$ is separable then $\mathcal{P}^F\circ \mathcal{F}^F=\id$ and hence $e_X=\mathcal{P}^F_{X,X}(\id_{FX})=\mathcal{P}^F_{X,X}\mathcal{F}^F_{X,X}(\id_{X})=\id_{X}$. Conversely, suppose $e=\id$. Then, for every $f:X\to Y$, we have $\mathcal{P}^F_{X,Y}(Ff)=\mathcal{P}^F_{X,Y}(Ff\circ \id_{FX})=f\circ \mathcal{P}^F_{X,X}(\id_{FX})=f\circ e_X=f$ so that $\mathcal{P}^F\circ \mathcal{F}^F=\id$ and $F$ is separable.
\end{proof}

The existence of the associated idempotent natural transformation leads us to a further characterization of separable functors in terms of Maschke, dual Maschke and conservative functors. Recall that a functor $F:\cc\to\dd$ is called a \emph{Maschke functor} if it reflects split-monomorphisms,
 i.e. for every morphism $i$ in $\cc$ such that $Fi$ is split-mono, then $i$ is split-mono\footnote{This is equivalent to \cite[Remark 6]{CMZ02}, where $F$ is called a Maschke functor if every object in $\cc$ is relative injective. Recall that an object $M$ is called relative injective if, for every morphism $i:C\to C'$
such that $Fi$ is split-mono, then the map $\Hom_\cc(i,M):\Hom_\cc(C',M)\to \Hom_\cc(C,M),f\mapsto f\circ i$, is surjective.}. Similarly, $F$ is a \emph{dual Maschke functor} if it reflects split-epimorphisms. A functor is called \emph{conservative} if it reflects isomorphisms. 

\begin{rmk}\label{rmk:Maschke}By \cite[Proposition 1.2]{NVV89} a separable functor is both Maschke and dual Maschke.
Moreover a functor which is both Maschke and dual Maschke is conservative.
\end{rmk}

\begin{cor}\label{cor:conserv}
The following assertions are equivalent for a functor $F:\cc\to\dd$.
\begin{enumerate}[$(1)$]
  \item $F$ is separable;
  \item $F$ is semiseparable and Maschke;
  \item $F$ is semiseparable and dual Maschke;
  \item $F$ is semiseparable and conservative.
\end{enumerate}
\end{cor}
\begin{proof}
$(1)\Rightarrow (2), (3), (4)$. By Proposition \ref{prop:sep} (i), a separable functor is semiseparable. Moreover, by Remark \ref{rmk:Maschke}, a separable functor is both Maschke and dual Maschke whence conservative.

$(2), (3), (4)\Rightarrow (1)$. Since $F$ is semiseparable, we can consider its associated idempotent natural transformation $e$ such that $Fe_X=\id_{FX}$, for every object $X$ in $\cc$. Thus $Fe_X$ is split-mono, split-epi and iso. Depending on whether $F$ is either Maschke, dual Maschke or conservative, we get that $e_X$ is either split-mono, split-epi or iso. Since $e_X$ is idempotent, we get $e_X=\id_X$ so that $F$ is separable by Corollary \ref{cor:esep}.
\end{proof}

\subsection{Relative separability}\label{sub:relsep}
Let $F:\cc\to\dd$ and $H:\cc\to\e$ be functors. We recall from \cite[Definition 4, page 97]{CMZ02} that $F$ is called $H$-\emph{separable} if there exists a natural transformation
\begin{equation*}
\p^{F,H} : \Hom_{\dd}(F-, F-)\rightarrow \Hom_{\e}(H-,H-)
\end{equation*}
such that $\p^{F,H}\circ\f^F = \f^H$. In particular, a $\id _{\cc}$-separable functor coincides with a separable functor.
The following result represents a connection between semiseparable functors and $H$-separable ones, and it will be used in Lemma \ref{lem_A} to study what happens if $G\circ F$ is semiseparable and $G$ is faithful.

\begin{prop}\label{prop:H-sep}
Let $H:\cc\to\e$ be a semiseparable functor with associated idempotent natural transformation $e$ and let $F:\cc\to\dd$ be a $H$-separable functor. If $Fe=\id_F$ (e.g. $\p^{F,H} $ is injective on components), then $F$ is semiseparable.
\end{prop}
\begin{proof}
By definition $\p^{F,H}\circ\f^F = \f^H$. Since $H$ is semiseparable, there exists a natural transformation $\p^H : \Hom_{\e}(H-, H-)\rightarrow \Hom_{\cc}(-,-)$ such that $\f^H\circ\p^H\circ\f^H =\f^H$. Set $\p^F:=\p^H\circ\p^{F,H}$, for every $X,Y$ in $\cc$. Then $\p^F\circ\f^F=\p^H\circ\p^{F,H}\circ\f^F=\p^H\circ\f^H$. Thus, for every $f:FX\to FY$, we have $\p^F_{X,Y}\f^F_{X,Y}(f)=\p^H_{X,Y}\f^H_{X,Y}(f)=\p^H_{X,Y}(Hf)=f\circ\p^H_{X,X}(\id_{HX})=f\circ e_X$ and hence $\f^F_{X,Y}\p^F_{X,Y}\f^F_{X,Y}(f)=F(f\circ e_X)=Ff\circ Fe_X=Ff=\f^F_{X,Y}(f)$ so that $\f^F_{X,Y}\p^F_{X,Y}\f^F_{X,Y}=\f^F_{X,Y}$ i.e. $F$ is semiseparable.
If $\p^{F,H} $ is injective on components, then from $\p^{F,H}_{X,X}(Fe_X)=\p^{F,H}_{X,X}\f^F_{X,X}(e_X)= \f^H_{X,X} (e_X)=He_X=H\id_X=\f^H_{X,X}(\id_X)=\p^{F,H}_{X,X}(F\id_X)$ we infer $Fe_X=\id_{FX}$.
\end{proof}

\begin{cor}\label{cor:isomorphic}
Let $H:\cc\to\e$ be a semiseparable functor with associated idempotent natural transformation $e$ and assume $H$ is a retract of a functor $F:\cc\to\dd$. If $Fe=\id_F$ then $F$ is semiseparable.
As a consequence, semiseparable functors are closed under isomorphisms.
\end{cor}

\begin{proof}
Since $H$ is a retract of $F$, there are natural transformations $\varphi : F\rightarrow H$ and $\psi : H\rightarrow F$ such that $\varphi\circ \psi=\id_H$. Define $\p^{F,H} $ by  setting $\p^{F,H}_{X,Y}(g):=\varphi_{Y}\circ g\circ\psi_{X}$, for every $g: FX\rightarrow FY$, and note that $\p^{F,H}\circ\f^F =\f^H$ so that $F:\cc\to\dd$ is $H$-separable.
\begin{invisible}
Indeed, since $\varphi$ is natural, for any morphism $f:X\to Y$ in $\cc$, we have
$\p^{F,H}_{X,Y}(F(f))= (\varphi_{Y}\circ F(f))\circ\psi_{X} = H(f)\circ\varphi_{X}\circ\psi_{X}=H(f).$
 \end{invisible}
Thus, by Proposition \ref{prop:H-sep}, if $Fe=\id_F$, the functor $F$ is semiseparable.
Let us prove the last part of the statement. Let $\varphi : F\rightarrow H$ be an isomorphism of the functors $F, H:\cc\rightarrow\dd$, where $H$ is semiseparable with associated idempotent natural transformation $e$. Clearly $H$ is a retract of $F$ via $\psi:=\varphi^{-1}$. Thus $F$ is semiseparable, as $Fe=Fe\circ \psi\circ\varphi=\psi\circ He\circ \varphi=\psi\circ \id_H\circ \varphi=\id_F$.
\end{proof}

\subsection{Behaviour with respect to composition}\label{subsect:comp} It is known that if $F:\cc\to\dd$ and $G:\dd\to\e$ are separable functors so is their composition $G\circ F$ and, the other way around, if the composition $G\circ F$ is separable so is $F$, see \cite[Lemma 1.1]{NVV89}. A similar result with some difference, holds for naturally full functors, see \cite[Proposition 2.3]{ACMM06}. Here we study the behaviour of semiseparable functors with respect to composition.
The first difference, with respect to the separable and naturally full cases, is that semiseparable functors are not closed under composition as we will see later in Example \ref{es:sscompos}. However the closeness is available in some cases, as the following result shows.

\begin{lem}\label{lem:comp}
Let $F: \cc \rightarrow \dd$ and $G:\dd\rightarrow\e$ be functors and consider the composite $G\circ F:\cc\rightarrow \e$.
\begin{itemize}
\item[(i)] If $F$ is semiseparable and $G$ is separable, then $G\circ F$ is semiseparable.
\item[(ii)] If $F$ is naturally full and $G$ is semiseparable, then $G\circ F$ is semiseparable.
\end{itemize}
\end{lem}
\proof
If $F$ is semiseparable with respect to $\p^F$ and $G$ is separable with respect to $\p^G$, then for every $X,Y$ in $\cc$ we have
\begin{gather*}
\f^{GF}_{X,Y}\p^F_{X,Y}\p^G_{FX,FY}\f^{GF}_{X,Y}=\f^{GF}_{X,Y}\p^F_{X,Y}\p^G_{FX,FY}\f^G_{FX,FY}\f^F_{X,Y}\\ 
=\f^{GF}_{X,Y}\p^F_{X,Y}\f^F_{X,Y}= \f^G_{FX,FY}\f^F_{X,Y}\p^F_{X,Y}\f^F_{X,Y} =\f^G_{FX,FY}\f^F_{X,Y}=\f^{GF}_{X,Y},
\end{gather*}
hence $G\circ F$ is semiseparable through $\p^{GF}_{X,Y}:=\p^F_{X,Y}\p^G_{FX,FY}$.
The proof of (ii) is similar.
\begin{invisible}
If $F$ is naturally full with respect to $\p^F$ and $G$ is semiseparable with respect to $\p^G$, then for every $X,Y$ in $\cc$ we have
\begin{gather*}
\f^{GF}_{X,Y}\p^F_{X,Y}\p^G_{FX,FY}\f^{GF}_{X,Y}=\f^G_{FX,FY}\f^F_{X,Y}\p^F_{X,Y}\p^G_{FX,FY}\f^{GF}_{X,Y}\\ 
=\f^G_{FX,FY}\p^G_{FX,FY}\f^{GF}_{X,Y}=\f^G_{FX,FY}\p^G_{FX,FY}\f^G_{FX,FY}\f^F_{X,Y} =\f^G_{FX,FY}\f^F_{X,Y}=\f^{GF}_{X,Y}\,
\end{gather*}
hence $G\circ F$ is semiseparable with respect to $\p^{GF}_{X,Y}:=\p^F_{X,Y}\p^G_{FX,FY}$.
\end{invisible}
\endproof


We now provide a variant of the property that if $G\circ F$ is separable so is $F$.

\begin{lem}\label{lem_A}
Let $F:\cc\to\dd$ and $G:\dd\to\e$ be functors. If $G\circ F$ is semiseparable and $G$ is faithful, then $F$ is semiseparable.
\end{lem}
\begin{proof}It follows from Proposition \ref{prop:H-sep}, by setting for every $X,Y$ in $\cc$, $\p^{F,GF}_{X,Y}:=\f^G_{FX,FY}$, which is injective, as $G$ is faithful.
\end{proof}

Afterwards, in Proposition \ref{prop:HGsscoref} we will see how, under stronger assumptions on $F$, the functor $G$ comes out to be semiseparable whenever $G\circ F$ is.

In Theorem \ref{thm:coidentifier}, we will give a criterion to factorize any semiseparable functor as the composition of a naturally full functor followed by a separable functor. The main ingredient will be the coidentifier category which is the object of the following subsection.

\subsection{The coidentifier}\label{sub:coidentifier} Given a category $\cc$ and an idempotent natural transformation $e:%
\id_{\cc}\rightarrow \id_{\cc}$, consider the coidentifier $\cc_{e}$ defined as in \cite[Example 17]{FOPTST99}. This is the quotient category $\cc/\sim$ of $\cc$ where $\sim$ is the congruence relation on the hom-sets defined, for all $f,g:A\to B$, by setting $f\sim g$ if and only if $e_{B}\circ f=e_{B}\circ g$. Thus, $\mathrm{Ob}\left( \cc%
_{e}\right) =\mathrm{Ob}\left( \cc\right) $ and $\Hom_{%
\cc_{e}}\left( A,B\right) =\Hom_{\cc}\left(
A,B\right) /\hspace{-2pt}\sim $. We denote by $\overline{f}$ the class of $f\in \Hom_{\cc%
}\left( A,B\right) $ in $\Hom_{\cc_{e}}\left( A,B\right) $.
We have the quotient functor $H:\cc\rightarrow \cc_{e}$
acting as the identity on objects and as the canonical projection on
morphisms. Note that $H$ is naturally full with respect to $\mathcal{P}^H_{A,B}:\Hom_{%
\cc_{e}}\left( A,B\right)\to \Hom_{\cc}\left(
A,B\right)$ defined by  $\mathcal{P}^H_{A,B}(\overline{f})=e_B\circ f$ and that the idempotent natural transformation associated to $H$ is exactly $e$.

\begin{lem}\label{lem:coidentifier}
 Let $\cc$ be a category, let $e:\id_{\cc}\rightarrow \id_{\cc}$ be an idempotent natural transformation and let $H:\cc\rightarrow \cc_{e}$ be the quotient functor.

  \begin{enumerate}
    \item[$(1)$] A functor $F:\cc\to \dd$  satisfies $Fe=\id_{F}$ if and only if there is a functor $F_{e}:\cc_{e}\rightarrow \dd$ (necessarily unique) such that $F=F_{e}\circ H$. Given $F,F':\cc\to \dd$  such that $Fe=\id_{F}$ and $F'e=\id_{F'}$, and a natural transformation $\beta:F\to F'$, there is a unique natural transformation $\beta_e:F_e\to F'_e$ such that $\beta=\beta_e H$.
    \item[$(2)$] The functor $H:\cc\rightarrow \cc_{e}$ is orthogonal to any faithful functor $S:\dd\to \mathcal{E}$ i.e., given  functors $F$ and $G$ such that $S\circ F=G\circ H$, then there is a unique functor $F_e:\cc_e\to \dd$ such that $F_e\circ H=F$ and $S\circ F_e=G$.
  \end{enumerate}

 \begin{gather*}
  \xymatrix{\cc\ar[r]^H\ar[d]_{F}&\cc_e \ar@{.>}[dl]^{F_e}\\ \dd
  }\qquad
  \xymatrix{\cc\ar[r]^H\ar[d]_{F}&\cc_e \ar@{.>}[dl]|{F_e}\ar[d]^{G}\\\dd\ar[r]^S&\mathcal{E}
  }
  \end{gather*}
\end{lem}

\begin{proof} First note that $He=\id_H$ as $e$ is the idempotent natural transformation associated to $H$.

(1). This property is the universal property of the coidentifier that can be deduced from the dual version of \cite[Definition 14(1)]{FOPTST99}. We just point out that the functor $F_{e}:\cc_{e}\rightarrow
\dd$ acts as $F$ on objects and maps the class $\overline{f}$ into $Ff$ and that, for every object $X$ in $\cc$, we have $(\beta_e)_X=\beta_X$.
\begin{invisible}
 Assume $Fe=\id_{F}$. Given $f,g\in \Hom_{\cc}\left( A,B\right) $, we have that  $\overline{f}=\overline{g}$ implies $e_{B}\circ
f=e_{B}\circ g$ and hence $Fe_{B}\circ Ff=Fe_{B}\circ Fg$ i.e. $Ff=Fg$ as $%
Fe_{B}=\id_{FB}$. Thus we can define the functor $F_{e}:\cc_{e}\rightarrow
\dd$ which acts as $F$ on objects and maps the class $\overline{f}$
into $Ff$. A direct computation shows that $F_{e}\circ H=F.$  Conversely, given a functor $G:\cc_e\to \dd$ such that $G\circ H=F$, we have $Fe=GHe=G\id_{H}=\id_{GH}=\id_F$ so that $Fe=\id_{F}$ and we can consider $F_e$. Moreover $GX=GHX=FX=F_eX$ and $G\overline{f}=GHf=Ff=F_e\overline{f}$ so that $G=F_e$. Given $\overline{f}:X\to Y$, the naturality of $\beta_e$ defined as above, is $(\beta_e)_Y\circ F_e\overline{f}=F_e'\overline{f}\circ (\beta_e)_X$ i.e. $\beta_Y\circ F{f}=F'{f}\circ \beta_X$ which is true by naturality of $\beta$.
\end{invisible}

(2). We compute $SFe_X=GHe_X=G\id_{HX}=\id_{GHX}=\id_{SFX}=S\id_{FX}$ so that, since $S$ is faithful, we get that $Fe_X=\id_{FX}$ and hence $Fe=\id_F$. Thus, by (1) there is a unique functor $F_{e}:\cc_{e}\rightarrow
\dd$, such that $F_e\circ H=F$, which acts as $F$ on objects and maps the class $\overline{f}$
into $Ff$. Moreover $S F_eX=SF_eHX=SFX=GHX=GX$ and $S F_e\overline{f}=S F_eHf=SFf=GHf=G\overline{f}$ so that $S\circ F_e= G.$
\end{proof}

 Another way to prove that $H:\cc\rightarrow \cc_{e}$ is orthogonal to any faithful functor $S:\dd\to \mathcal{E}$ is to observe it is eso (essentially surjective on objects, i.e. for every $D\in\cc_e$ there is $C\in\cc$ such that $D\cong H(C)$) and full and that there is an (eso and full, faithful) factorization system, see e.g. \cite[Example 7.9]{DV03}.
 In the following result, we show that any semiseparable functor admits a special type of (eso and full, faithful) factorization. In fact $F=F_{e}\circ H$, where $H$ is eso and (naturally) full while $F_{e}$ is separable whence faithful.

\begin{thm}\label{thm:coidentifier}
Let $F:\cc\rightarrow \dd$ be a semiseparable functor and let $e:\id_\cc\to\id_\cc$ be the associated idempotent natural transformation. Then, there is a unique functor $F_{e}:\cc%
_{e}\rightarrow \dd$ (necessarily separable) such that $F=F_{e}\circ H$ where $H:\cc\rightarrow \cc_{e}$ is the quotient functor. Furthermore, if $F$ also factors as $S\circ N$ where $S:\mathcal{E}\rightarrow \dd$ is a separable functor and $N:\cc\rightarrow \mathcal{E}$ is a naturally full functor, then there is a unique functor $N_e:\cc_e\to \mathcal{E}$ (necessarily fully faithful) such that $N_e\circ H=N$ and $S\circ N_e=F_e$, and $e$ is also the idempotent natural transformation associated to $N$.
\begin{gather*}
  \xymatrix{\cc\ar[r]^H\ar[d]_{N}&\cc_e \ar@{.>}[dl]|{N_e}\ar[d]^{F_e}\\\mathcal{E}\ar[r]^S&\dd
  }
  \end{gather*}
\end{thm}

\begin{proof}
By Lemma \ref{lem:coidentifier}, there is a unique functor $F_{e}:\cc_{e}\rightarrow \dd$ such that $F=F_{e}\circ H$ where $H:\cc\rightarrow \cc_{e}$ is the quotient functor. If $F_{e}%
\overline{f}=F_{e}\overline{g},$ then $Ff=Fg$ so that, by Proposition \ref{prop:idempotent}, we get $e_{B}\circ f=e_{B}\circ g$ which means $%
\overline{f}=\overline{g}$. Thus $F_{e}$ is faithful.  Moreover $\mathcal{F}^F%
_{X,Y}\circ \mathcal{P}^F_{X,Y}\circ \mathcal{F}^F_{X,Y}=\mathcal{F}^F_{X,Y}$
rewrites as $\mathcal{F}^{F_e} _{X,Y}\circ \mathcal{F}^H%
_{X,Y}\circ \mathcal{P}^F_{X,Y}\circ \mathcal{F}^{F_e}_{X,Y}\circ
\mathcal{F}^H_{X,Y}=\mathcal{F}^{F_e}_{X,Y}\circ \mathcal{F}^H%
_{X,Y}.$ Since $\mathcal{F}^{F_e}_{X,Y}$ is injective and $%
\mathcal{F}^H_{X,Y}$ is surjective, we get $\mathcal{F}^H_{X,Y}\circ \mathcal{P}^F%
_{X,Y}\circ \mathcal{F}^{F_e}_{X,Y}=\id$ which implies
that $F_{e}$ is separable (and also that $H$ is naturally full,
fact that we already know).

Concerning the last sentence, since $S$ is separable, then it is faithful. By Lemma \ref{lem:coidentifier} $H$ is  orthogonal to $S$ so that there is  a unique functor $N_e:\cc_e\to \mathcal{E}$ such that $N_e\circ H=N$ and $S\circ N_e=F_e$. Since $N_e\circ H=N$ and $N$ is full, we deduce that $N_e$ is full (this is not true in general, but here $H$ acts as the identity on objects) and since $S\circ N_e= F_e$ and $F_e$ is faithful, we deduce that $N_e$ is faithful. Thus $N_e$ is fully faithful.\par
It remains to prove that $F$ and $N$ share the same associated idempotent natural transformation\footnote{thus the functor $N_e:\cc_e\to \mathcal{E}$ such that $N_e\circ H=N$ and $S\circ N_e=F_e$, is exactly the separable functor achieved from the first part of this theorem applied to the naturally full functor $N$.} $e: \id_\cc\to\id_\cc$. Indeed, by Corollary \ref{cor:natfidp}, $N$ has an associated idempotent natural transformation $e': \id_\cc\to\id_\cc$ and by definition we have $e'_X:=\p^{N}_{X,X}(\id_{NX})$, for any $X\in\cc$. Since $F=S\circ N$, by the proof of Lemma \ref{lem:comp} (i), we can choose $\p^F_{X,Y}:=\p^{N}_{X,Y}\circ \p^{S}_{NX,NY}$, so that $e_X=\p^F_{X,X}(\id_{FX})=\p^{N}_{X,X}(\p^{S}_{NX,NX}(\id_{SNX}))=\p^{N}_{X,X}(\id_{NX})=e'_X$ whence $e=e'$.
\end{proof}

We are now ready to prove the desired characterization of semiseparable functors in terms of separable and  naturally full functors.

\begin{cor}\label{cor:fattoriz}
  A functor is semiseparable if and only if it factors as $S\circ N$ where $S$ is a separable functor and $N$ is a naturally full functor.
\end{cor}
\begin{proof}
If a functor is semiseparable, it factors as a naturally full functor followed by a separable one by Theorem \ref{thm:coidentifier}. Conversely, by Lemma \ref{lem:comp} (i), the composition $S\circ N$, of a separable functor $S$ by a naturally full (whence semiseparable) functor $N$, is semiseparable.
\end{proof}

\begin{invisible}\begin{rmk}(Cf. \cite[Lemma 2.7]{CJKP97})
If a semiseparable functor $F:\cc\to\dd$ is orthogonal to the induced functor $F_e:\cc_e\to\dd$ of Theorem \ref{thm:coidentifier}, then $F_e$ is an isomorphism, which means that $\cc_e$ and $\dd$ are isomorphic. 
\end{rmk}\end{invisible}

\subsection{Generators}\label{generators}
We want to investigate how the existence of a suitable type of generator within a category $\cc$ could imply that a functor $F:\cc\to \dd$ is semiseparable if and only if it is separable.\medskip

Recall, from \cite[Definition 7]{He71}, that a morphism $k:X\to Y$ in a category $\cc$ is called \emph{constant} provided that for each object $Z$ in $\cc$ and for each pair of morphisms $g,h:Z\to X$, it follows $k\circ g=k\circ h$. A category $\cc$ is said to be \emph{constant-generated} provided that, for any pair of morphisms $f,g:X\to Y$ in $\cc$ such that $f\neq g$, then there exist an object $\G$ and a constant morphism $k:\G\to X$ such that $f\circ k\neq g\circ k$. We point out that the definition of constant-generated category we are giving here differs from the original one of \cite[Definition 8]{He71} in the fact that we do not require that $\Hom_\cc(X,Y)\neq\emptyset$, condition which is superfluous for our purposes.

\begin{prop}\label{prop:constgen}
  If $\cc$ is a constant-generated category, then $\mathrm{Nat}(\id_\cc,\id_\cc)=\{\id\}$. As a consequence, a functor $F:\cc\to \dd$ is semiseparable if and only if it is separable.
\end{prop}

\begin{proof} Let  $e\in \mathrm{Nat}(\id_\cc,\id_\cc)$ and suppose that $e_X\neq \id_X$, for some object $X$ in $\cc$. Since $\cc$ is constant-generated, there are an object $\G$ and a constant morphism $k:\G\to X$ such that $e_X\circ k\neq \id_X\circ k$. By naturality of $e$ and since $k$ is constant, we have $e_X\circ k=k\circ e_{\G}=k\circ \id_{\G}=\id_X\circ k$, a contradiction. Therefore $e_X=\id_X$ and hence $e=\id$. We conclude by Corollary \ref{cor:esep}.
\end{proof}

Recall that an object $\G$ in a  category $\cc$ is called a \emph{generator} if, for every pair of morphisms $f,g:X\to Y$ in $\cc$ such that $f\neq g$, there is a morphism $p:\G\to X$ such that $f\circ p\neq g\circ p$. If the domain of a functor $F$ is a category with a generator, instead of a constant-generated category, it is not obvious that $F$ is semiseparable if and only if it is separable. However we are able to retrieve the same conclusion by adding suitable assumptions.

A first example in this direction is given by taking a \emph{well-pointed} category, i.e. a category that has a generator which is at the same time a terminal object.
\begin{cor}\label{cor:well-pointed}
  If $\cc$ is a well-pointed category, then it is constant-generated. As a consequence, a functor $F:\cc\to \dd$ is semiseparable if and only if it is separable.
\end{cor}

\begin{proof}
Let $\G$ be a generator which is a terminal object. Given morphisms $f,g:X\to Y$ in $\cc$ such that $f\neq g$, since $\G$ is a generator there is a morphism $k:\G\to X$ such that $f\circ k \neq g\circ k$. On the other hand, since $\G$ is a terminal object, then $k$ is constant.
\begin{invisible}
 In fact, if $a,b:Z\to \G$ are any morphisms, then they are necessarily equal whence $k\circ a=k\circ b$.
\end{invisible}We conclude by Proposition \ref{prop:constgen}.
\end{proof}

\begin{es}
Corollary \ref{cor:well-pointed} applies in case $\cc$ is either the category $\Set$ of sets or the category $\Top$  of topological spaces or the category $\Comp$ of compact Hausdorff spaces which are well-pointed. In fact the singleton $\{*\}$ is both a terminal object and a generator in all of these categories, see \cite[2.3.2.a, 2.1.7g, 4.5.17.a, 4.5.17.f and 4.5.17.g]{Bor94}.
\end{es}

\begin{rmk}
  In view of Proposition \ref{prop:constgen} and Corollary \ref{cor:well-pointed}, we get that in a well-pointed category $\cc$ one has $\mathrm{Nat}(\id_\cc,\id_\cc)=\{\id\}$. This result already appeared in \cite[Corollary 21]{FOPTST99}.
\end{rmk}

Looking for other additional conditions guaranteeing the equivalence between the semiseparability and the separability of a functor, we need the notion of \emph{central idempotent endomorphism} of an object $\G$ in a category $\cc$. By this, we mean a central idempotent in the monoid $(\mathrm{End}(\G),\circ,\id_{\G})$, i.e. a morphism $g:\G\to \G$ such that $g\circ g=g$ and $g\circ f=f\circ g$ for every morphism $f:\G\to \G$.

\begin{prop}\label{prop:generator}
  Let $\cc$ be a category with a generator $\G$ and let $F:\cc\to \dd$ be a functor. Assume there is no central idempotent endomorphism $g\neq \id_{\G}:\G\to \G$ such that $Fg=\id_{F\G}$.

  Then, $F$ is semiseparable if and only if it is separable.
\end{prop}
\begin{proof}
 Consider an idempotent natural transformation $e:\id_\cc\to\id_\cc$ such that $Fe=\id_F$. Then $e_{\G}$ is a central idempotent endomorphism of $\G$ such that $Fe_\G=\id_{F\G}$ and hence $e_\G=\id_\G$ by hypothesis. Let $X$ be an object in $\cc$ and suppose that $e_X\neq \id_X$. Since $\G$ is a generator, there is a morphism $p:\G\to X$ such that $e_X\circ p\neq \id_X\circ p$ but, by naturality of $e$, we have $e_X\circ p=p\circ e_\G=p\circ \id_\G=p$ so that we are led to a contradiction. Therefore $e_X=\id_X$ and hence $e=\id$. We conclude by Corollary \ref{cor:esep}.
\end{proof}

We are now going to apply Proposition \ref{prop:generator} to the category $R$-$\mathrm{Mod}$ of left $R$-modules. First we need the following easy lemma.

\begin{lem}\label{lem:centidpendR}
Let $R$ be a ring. Then $g:R\to R$ is a  central idempotent endomorphism of left $R$-modules if and only if $g=z\id_R$ for a central idempotent $z\in R$, namely $z=g(1)$.
\end{lem}

\begin{proof}
 For every $r\in R$, consider the morphism of left $R$-modules $f_r:R\to R, f_r(x):= xr$. Assume that $g$ is a  central idempotent endomorphism of left $R$-modules. Then $g(r)=rg(1)=rz$. From $g\circ f_r=f_r\circ g$ we get $zr=f_r(z)=f_r(g(1))=g(f_r(1))=g(r)=rz$ so that $z$ is in the center of $R$. Moreover, since $g$ is left $R$-linear and idempotent, we get $zz=g(z)=g(g(1))=g(1)=z$. Conversely, it is clear that $z\id_R:R\to R$ is a  central idempotent endomorphism of left $R$-modules in case $z$ is a central idempotent in $R$.
\end{proof}

\begin{cor}
  Let $R$ be a ring with no non-trivial central idempotent (e.g. $R$ is a domain). A functor $F:{}R\text{-}\mathrm{Mod}\to \dd$ such that $F0\neq\id_{FR}$ is semiseparable if and only if it is separable.
\end{cor}

\begin{proof}
Let $g:R\to R$ be a central idempotent endomorphism of left $R$-modules such that $Fg=\id_{FR}$. By Lemma \ref{lem:centidpendR}, we have that $g=z\id_R$ for a central idempotent $z\in R$. By hypothesis $z$ is trivial i.e. $z=1$ or $z=0$ and hence we get either $g=\id_R$ or $g=0$. Since $Fg=\id_{FR}$, we must have $g=\id_R$. Since $R$ is a generator in $R\text{-}\mathrm{Mod}$, by Proposition \ref{prop:generator}, we conclude.
\end{proof}

\section{Semiseparability and adjunctions}\label{sect:semisep-adjunct}
This section collects results on semiseparable functors which have an adjoint. Explicitly in Subsection \ref{sub:rafael}, we investigate a Rafael-type theorem for semiseparable functors. In Subsection \ref{section:Eil-Kleisli}, we study the behaviour of semiseparable adjoint functors in terms of (co)monads and the associated (co)comparison functor. Subsection \ref{sub:adjtriple} contains results on semiseparability of functors that have both a (possibly equal) left and right adjoint.

\subsection{Rafael-type Theorem}\label{sub:rafael}
Rafael Theorem \cite{Raf90} provides a characterization of separable functors which have an adjoint: explicitly, given an adjunction $F\dashv G:\dd\to\cc$ with unit $\eta$ and counit $\epsilon$, then $F$ is separable if and only if there exists a natural transformation $\nu:GF\to \id _\cc$ such that $\nu\circ\eta =\id_{\id_\cc}$ while $G$ is separable if and only if there exists a natural transformation $\gamma:\id_\dd\to FG$ such that $\epsilon\circ\gamma=\id_{\id_\dd}$.
Next result extends Rafael Theorem to semiseparable functors.

\begin{thm}(Rafael-type Theorem)\label{thm:rafael}
Let $(F:\cc\to\dd, G:\dd\to\cc)$ be an adjoint pair of functors, with unit $\eta:\id\to GF$ and counit $\epsilon:FG\to\id$. Then:
\begin{itemize}
\item[(i)] $F$ is semiseparable if and only if $\eta$ is regular.
\item[(ii)] $G$ is semiseparable if and only if $\epsilon$ is regular.
\end{itemize}
\end{thm}
\proof
(i) Assume that $F$ is semiseparable and let $\p^F$ be a natural transformation such that $\f^F\circ\p^F\circ\f^F =\f^F$. We define $\nu : GF\rightarrow \id _{\cc}$ on components by setting
$\nu_{X}:=\p^F_{GFX,X}(\epsilon_{FX}): GFX\rightarrow X$,
for any object $X$ in $\cc$. The naturality of $\nu_{X}$ in $X$ follows from the one of $\p^F$. Moreover, by naturality of $\p^F$, for any $X,Y$ in $\cc$ and $g:FX\to FY$ we also have
\begin{equation*}
\begin{split}
\nu_Y\circ Gg\circ\eta_X &=\p^F_{GFY,Y}(\epsilon_{FY})\circ Gg\circ\eta_X=\p^F_{X,Y}(\epsilon_{FY}\circ FGg\circ F\eta_X)\\
&=\p^F_{X,Y}(g\circ\epsilon_{FX}\circ F\eta_X)=\p^F_{X,Y}(g\circ \id_{FX})=\p^F_{X,Y}(g).
\end{split}
\end{equation*}
The associated natural transformation is defined by $e_X:= \p^F_{X,X}(\id_{FX})=\nu_X\circ G\id_{FX}\circ\eta_X =\nu_X\circ\eta_X$ so that $e=\nu\circ\eta$. We compute
$\eta\circ\nu\circ\eta=\eta\circ e=GFe\circ\eta=G\id_F\circ\eta=\eta$. Thus, $\eta$ is regular.

Conversely, suppose $\eta$ is regular, i.e. there exists a natural transformation $\nu :GF\rightarrow \id _{\cc}$ such that $\eta\circ\nu\circ\eta =\eta$, and for any $f\in \Hom_{\dd}(FX,FY)$ define
$ \p^F_{X,Y}(f):=\nu_Y\circ Gf\circ\eta_X.$
From the naturality of $\eta$ and $\nu$, for any $h:X\to Y$, $k:FY\to FZ$, and $l:Z\to T$ we have $\p^F_{X,T}(Fl\circ k\circ Fh)=\nu_T\circ G(Fl\circ k\circ Fh)\circ\eta_X=(\nu_T\circ GFl )\circ Gk\circ (GFh\circ\eta_X)=l\circ (\nu_Z\circ Gk\circ \eta_{Y})\circ h=l\circ\p^F_{Y,Z}(k)\circ h$, thus $\p^F : \Hom_{\dd}(F-, F-)\rightarrow \Hom_{\cc}(-,-)$ is a natural transformation. Since $\p^F_{GFX,X}(\epsilon_{FX})=\nu_X\circ G\epsilon_{FX}\circ\eta_{GFX}=\nu_X\circ \id_{GFX}=\nu_X$, the correspondence between $\p^F$ and $\nu$ is bijective.
  Set $e:=\nu\circ\eta.$ Then $Fe=F(\nu\circ\eta)
=\id _{F}\circ F(\nu\circ\eta)=\epsilon {F}\circ F\eta\circ F(\nu\circ\eta)
=\epsilon {F}\circ F(\eta\circ\nu\circ\eta)=\epsilon {F}\circ F\eta
=\id _{F}$ i.e. $Fe=\id_F$. Therefore $F$ is semiseparable by the following computation, that holds for every $ f:X\to Y$
\begin{gather*}
(\f^F _{X,Y}\circ\p^F _{X,Y}\circ\f^F _{X,Y}) (f)= F(\p^F _{X,Y}(F(f)))=F(\nu_Y\circ GF(f)\circ\eta_X)\\
=F(\nu_Y\circ\eta_{Y}\circ f)=F(e_Y)\circ Ff=\id_{FY}\circ \f^F _{X,Y}(f)=\f^F _{X,Y}(f).
\end{gather*}
%
%
(ii) It follows by duality.\qedhere

\endproof

We include here a useful lemma, which characterizes the regularity of unit and counit.
\begin{lem}\label{lem_B}
Let $(F:\cc\to\dd, G:\dd\to\cc)$ be an adjoint pair of functors, with unit $\eta$ and counit $\epsilon$.
\begin{itemize}
\item[(i)] The following equalities are equivalent for a natural transformation $\nu : GF\to \id _{\cc}$:
\begin{enumerate}[$(1)$]
  \item $\eta\circ\nu\circ\eta =\eta$ (i.e. $\eta$ is regular);
  \item $F\nu\circ F\eta = \id _{F}$;
  \item $\nu G\circ \eta G= \id _{G}$.
\end{enumerate}
\item[(ii)] The following equalities are equivalent for a natural transformation $\gamma : \id _{\dd}\to FG$:
\begin{enumerate}[$(1)$]
\item  $\epsilon\circ\gamma\circ\epsilon = \epsilon$  (i.e. $\epsilon$ is regular);
\item $G\epsilon\circ G\gamma = \id _{G}$;
\item $\epsilon F\circ \gamma F = \id _{F}$.
\end{enumerate}
\end{itemize}
\end{lem}
\proof
We just prove $(i)$ as $(ii)$ follows dually.\\
$(1)\Rightarrow (2)$. $F\nu \circ F\eta = \id _{F}\circ F\nu \circ F\eta =\epsilon F\circ F\eta\circ F\nu\circ F\eta =\epsilon F\circ F\eta = \id _{F}$.\\
$(2)\Rightarrow (1)$. By naturality of $\eta$, we have $\eta\circ\nu\circ\eta =\eta\circ (\nu \circ\eta ) = GF(\nu\circ\eta )\circ\eta = G(F\nu\circ F\eta )\circ \eta = \eta$.\\
$(1)\Rightarrow (3)$. $\nu G \circ \eta G = \id _{G}\circ \nu G \circ \eta G =G\epsilon \circ \eta G\circ \nu G \circ \eta G =G\epsilon \circ \eta G = \id _{G}$.\\
$(3)\Rightarrow (1)$. By naturality of $\nu \circ\eta$, we have $\eta\circ\nu\circ\eta =\eta\circ (\nu \circ\eta ) = (\nu\circ\eta )GF\circ\eta = (\nu G\circ \eta G)F\circ \eta = \eta$.
\endproof

\begin{rmk}\label{rmk:rafidp}
Let $(F:\cc\to\dd, G:\dd\to\cc)$ be an adjunction with unit $\eta$ and counit $\epsilon$.

1) Assume that there is a natural transformation $\nu :GF \rightarrow \id_\cc$ such that $\eta\circ\nu\circ\eta = \eta $. By Theorem \ref{thm:rafael}, we know that $F$ is semiseparable so that we can take the associated idempotent natural transformation $e:\id_\cc\to \id_\cc$. We can write it explicitly in terms of $\nu$. Indeed, by the proof of Theorem \ref{thm:rafael}, we can define $\p^F_{X,Y}:\Hom_{\dd}(FX,FY)\to\Hom_\cc(X,Y)$ by setting $\p^F_{X,Y}(f):=\nu_Y\circ Gf\circ\eta_X$ for every morphism $f:FX\to FY$. By definition, $e_X:=\p^F_{X,X}(\id_{FX})=\nu_X\circ\eta_X$ so that $e=\nu\circ\eta$.

2) Dually, if there is a natural transformation $\gamma :\id_\dd \rightarrow FG$ such that $\epsilon\circ\gamma\circ\epsilon = \epsilon $, then $G$ is semiseparable and the associated idempotent natural transformation is $e=\epsilon\circ\gamma:\id_\dd\to \id_\dd$.
\begin{invisible}
By the proof of Theorem \ref{thm:rafael}, we can define $\p^G_{X,Y}:\Hom_{\cc}(GX,GY)\to\Hom_\dd(X,Y)$ by setting $\p^G_{X,Y}(f):=\epsilon_Y\circ Ff\circ\gamma_X$ for every morphism $f:GX\to GY$. By the proof of Proposition \ref{prop:idempotent}, we can define $e:\id_\dd\to \id_\dd$ by setting $e_X:=\p^G_{X,X}(\id_{GX})=\epsilon_X\circ\gamma_X$.
\end{invisible}
    \end{rmk}

\subsection{Eilenberg-Moore category}\label{section:Eil-Kleisli}
In order to study the behaviour of semiseparable adjoint functors in terms of separable (co)monads and associated (co)comparison functor, we remind some basics facts concerning Eilenberg-Moore categories \cite{EM65}.

Given a monad $(\top,m:\top\top\to \top,\eta:\id_\cc\to \top)$ on a category $\cc$ we denote by $\cc_{\top}$ the Eilenberg-Moore category of modules (or algebras) over it. The forgetful functor $U_\top:\cc_\top\to \cc$ has a left adjoint, namely the \emph{free functor}
$$V_\top:\cc\to \cc_\top,\qquad C\mapsto (\top C,m _C),\qquad f\mapsto \top (f).$$
The unit $\id_\cc\to U_\top V_\top=\top$ is exactly $\eta$ while the counit $\beta:V_\top U_\top\to \id_{\cc_\top}$ is completely determined by the equality by $U_\top\beta_{(X,\mu)}=\mu$ for every object $(X,\mu)$ in $\cc_\top$ (see \cite[Proposition 4.1.4]{BorII94}). Dually, given a comonad $(\bot,\Delta:\bot\to \bot\bot,\epsilon:\bot\to \id_\cc)$ on a category $\cc$ we denote by $\cc^{\bot}$ the Eilenberg-Moore category of comodules (or coalgebras) over it. The forgetful functor $U^\bot:\cc^\bot\to \cc$ has a right adjoint, namely the \emph{cofree functor} $$V^\bot:\cc\to \cc^\bot,\qquad C\mapsto (\bot C,\Delta _C),\qquad f\mapsto \bot (f).$$
The unit $\alpha:\id_{\cc^\bot}\to V^\bot U^\bot $ is completely determined by the equality by $U^\bot\alpha_{(X,\rho)}=\rho$ for every object $(X,\rho)$ in $\cc^\bot$ while the counit $U^\bot V^\bot =\bot\to \id_{\cc}$ is exactly $\epsilon$.\medskip

Given an adjunction $F:\cc\to\dd, G:\dd\to\cc)$, with unit $\eta$ and counit $\epsilon$, we can consider the monad $(GF, G\epsilon F, \eta )$ and the comonad $(FG, F\eta G, \epsilon )$.
We have the \emph{comparison functor}
$$K_{GF}:\dd\to \cc_{GF},\qquad D\mapsto (GD,G\epsilon_{D}),\qquad f\mapsto G (f),$$
and the \emph{cocomparison functor}
$$K^{FG}:\cc\to \dd^{FG},\qquad C\mapsto (FC,F\eta_{C}),\qquad f\mapsto F (f),$$
that fit into the diagram
\begin{gather}\label{diag:eil-moore}
\vcenter{%
\xymatrixcolsep{1.8cm}\xymatrixrowsep{1cm}\xymatrix{
\dd^{FG} \ar@{}[r]|-\perp \ar@<1ex>[r]^-{U^{FG}}&\dd\ar@/^1pc/[rd]^{K_{GF}} \ar@<1ex>[l]^-{V^{FG}}
\ar@<1ex>[d]^*-<0.1cm>{^{G}}\\
&\cc \ar@/^1pc/[lu]^{K^{FG}}  \ar@<1ex>[u]^*-<0.1cm>{^{F}}\ar@{}[u]|{\dashv}\ar@{}[r]|-\perp \ar@<1ex>[r]^-{V_{GF}}&\cc _{GF},\ar@<1ex>[l]^-{U_{GF}}
} }
\end{gather}
where $U_{GF}\circ K_{GF}=G$, $K_{GF}\circ F = V_{GF}$, $U^{FG}\circ K^{FG}=F$ and $K^{FG}\circ G = V^{FG}$.

We recall that a monad $(\top,m:\top\top\to \top,\eta:\id_\cc\to \top)$ on a category $\cc$ is said to be \emph{separable} \cite{BruV07} if there exists a natural transformation $\sigma :\top\to\top\top$ such that $m\circ\sigma =\id _{\top}$ and $\top m\circ\sigma \top = \sigma\circ m = m\top\circ \top\sigma$; in particular, a separable monad is a monad satisfying the equivalent conditions of \cite[Proposition 6.3]{BruV07}.\par
Dually, a comonad $(\bot,\Delta:\bot\to \bot\bot,\epsilon:\bot\to \id_\cc)$ on a category $\cc$ is said to be \emph{coseparable} if there exists a natural transformation $\tau :\bot\bot\to\bot$ satisfying $\tau\circ\Delta = \id _{\bot}$ and $\bot\tau\circ\Delta\bot=\Delta \circ \tau =\tau\bot\circ\bot\Delta$.\par
Furthermore, an \emph{idempotent monad} is a monad $(\top,m,\eta)$ on a category $\cc$ whose multiplication $m$ is an isomorphism or, equivalently, such that the forgetful functor $U_\top:\cc_\top\to \cc$ is fully faithful, see \cite[Proposition 4.2.3]{BorII94}. 
Dually, an \emph{idempotent comonad} is a comonad $(\bot,\Delta,\epsilon)$ on a category $\cc$ whose comultiplication $\Delta$ is an isomorphism or, equivalently, such that the forgetful functor $U^\bot:\cc^\bot\to \cc$ is fully faithful, see \cite[Section 6]{ApT69}.

 An adjunction $F\dashv G:\dd\to\cc$ with unit $\eta :\id_{%
\cc}\rightarrow GF$ and counit  $\epsilon :FG\rightarrow \mathrm{%
Id}_{\dd}$ is said to be an \emph{idempotent adjunction}\footnote{As underlined in \cite{ClW11}, the first hint of \emph{idempotent adjunctions} can be found in \cite{Mar66} under the name of \emph{idempotent constructions.}} if the monad $(GF,G\epsilon F,\eta)$ is idempotent, or equivalently if the comonad $(FG, F\eta G, \epsilon )$ is idempotent, see e.g. \cite[Subsection 3.4]{ClW11}. Indeed by \cite[Proposition 2.8]{MS82} this is equivalent to ask that anyone of the natural transformations $\epsilon F,$ $G\epsilon$, $F\eta$ and $\eta G$ is an isomorphism.

\begin{rmk}
An idempotent (co)monad on a category $\cc$ is always (co)separable with splitting given by the inverse of the (co)multiplication. Another way to arrive at the same conclusion is to observe that the forgetful functor $U_\top:\cc_\top\to \cc$ (resp. $U^\bot:\cc^\bot\to\cc$) is both separable and naturally full whenever it is fully faithful.
\end{rmk}

\begin{rmk}\label{rmk:monad}Let $(F:\cc\to\dd, G:\dd\to\cc)$ be an adjunction with unit $\eta :\id_{\cc}\rightarrow GF$ and counit $\epsilon :FG\rightarrow \mathrm{Id}_{\dd}$.
\begin{enumerate}[\quad 1)]
  \item The adjunctions $(F,G)$ and $(V_{GF},U_{GF})$ have the same associated monad $(GF,G\epsilon F,\eta)$, whereas the adjunctions $(F,G)$ and $(U^{FG},V^{FG})$ have the same associated comonad $(FG,F\eta G,\epsilon )$.
  \item By 1), $(F,G)$ is idempotent if and only if $(V_{GF},U_{GF})$ is idempotent, if and only if $(U^{FG},V^{FG})$ is idempotent.
  \item The counit of an adjunction coincides with the counit of the associated comonad. Thus, by 2), the adjunctions $(F,G)$ and $(U^{FG},V^{FG})$ have the same counit. As a consequence, $G$ is semiseparable (resp. separable, naturally full, fully faithful) if and only if so is $V^{FG}$ in view of the corresponding Rafael-type Theorems (i.e. Theorem \ref{thm:rafael}, \cite[Theorem 1.2]{Raf90} and \cite[Theorem 2.6]{ACMM06}) and their combination for fully faithfulness.
  \item Similarly, the adjunctions $(F,G)$ and $(V_{GF},U_{GF})$ have the same unit and hence $F$ is semiseparable (resp. separable, naturally full, fully faithful) if and only if so is $V_{GF}$.
\end{enumerate}
\begin{invisible}
By 2) these two adjunction have the same associated monad whence the same unit. The last sentence follows from Rafael-type Theorem for left adjoint functors and their combination for fully faithfulness.
\end{invisible}
\end{rmk}


Now, let $(F,G,\eta,\epsilon )$ be an adjunction. In \cite[Lemma 3.1]{Chen15}  it is proved that if the right adjoint $G$ is separable then the associated monad $(GF,G\epsilon F,\eta)$ is separable. We show that the semiseparability of $G$ is enough to gain the separability of the associated monad. The proof is similar to the separable case but uses Lemma \ref{lem_B}. We also prove the analogous result involving the left adjoint and the associated comonad.
\begin{lem}\label{lem:monad} Let $(F,G,\eta,\epsilon )$ be an adjunction.
\begin{itemize}
\item[(i)] If $G$ is semiseparable, then the associated monad $(GF,G\epsilon F,\eta)$ is separable.
\item[(ii)] If $F$ is semiseparable, then the associated comonad $(FG, F\eta G, \epsilon)$ is coseparable.
\end{itemize}
\end{lem}
\proof
(i) Assume $G$ is semiseparable. Then, by Theorem \ref{thm:rafael} and Lemma \ref{lem_B} (ii), there is a natural transformation $\gamma : \id _{\dd}\to FG$ such that $G\epsilon\circ G\gamma = \id _{G}$. Set $\sigma :=G\gamma F: G\id _{\dd}F\to GFGF$. It follows that $G\epsilon F\circ\sigma = G\epsilon F\circ G\gamma F=\id _{GF}$. Moreover, from the naturality of $\epsilon$ and $\gamma$, we have $\gamma\circ\epsilon =\epsilon FG\circ FG\gamma$ and $\gamma\circ\epsilon = FG\epsilon\circ\gamma FG$, respectively, hence
$GFG\epsilon F\circ \sigma GF=GFG\epsilon F\circ G\gamma FGF= G\gamma F\circ G\epsilon F= G\epsilon FGF\circ GFG\gamma F = G\epsilon FGF\circ GF\sigma .$ 
Therefore, the monad $(GF,G\epsilon F,\eta)$ is separable.

(ii) The proof is dual by using Lemma \ref{lem_B} (i).\qedhere
\begin{invisible}
Indeed, if $F$ is semiseparable, by Theorem \ref{thm:rafael} and Lemma \ref{lem_B} (i), there exists a natural transformation $\nu :GF\rightarrow \id _{\cc}$ such that $F\nu\circ F\eta = \id _{F}$. Set $\tau :=F\nu G: FGFG\to F\id _{\cc}G$, so we have $\tau\circ F\eta G =\id_{FG}$. From the naturalness of $\eta$ e $\nu$, it follows that
\begin{equation*}
FG\tau\circ F\eta GFG = F\eta G\circ\tau = \tau FG\circ FGF\eta G ,
\end{equation*}
hence $(FG, F\eta G, \epsilon)$ is a coseparable comonad.
\end{invisible}
 \endproof

\begin{rmk}
 We have recalled that if the right adjoint $G$ of an adjunction $(F,G)$ is separable then the associated monad $(GF,G\epsilon F,\eta)$ is separable. It is known that the converse is not true as \cite[Example 3.7(2)]{Chen15} shows. Explicitly, let $\cc$ and $\cc'$ be two nontrivial additive categories, and consider the product category $\dd=\cc\times \cc'$. Let $F:\cc\to\dd$ be the canonical functor sending an object $C$ to $(C,0)$ and a morphism $f$ to $(f,0)$. Its right adjoint is the projection functor $G:\dd\to\cc$. Then, the associated monad $GF$ equals the identity monad on $\cc$, which is separable, but $G$ is not separable, as it is not faithful. Nevertheless, $G$ results to be semiseparable. Indeed, let $\epsilon :FG\to\id_\dd$ be the counit of the adjunction given for any $D=(C,C')$ in $\dd$ by $\epsilon_D=(\id_C,\varphi^I_{C'}):FGD\to D$, where $\varphi^I_{C'}$ is the unique map from the zero object $0$ of $\cc$ to $C'$. Consider the natural transformation $\gamma:\id_\dd\to FG$, given for any $D=(C,C')$ in $\dd$ by $\gamma_D=(\id_C,\varphi^T_{C'}):D\to FGD$, where $\varphi^T_{C'}$ is the unique map from $C'$ to $0$. Then, from $\gamma_D\circ\epsilon_D=(\id_C,\varphi^T_{C'})\circ (\id_C,\varphi^I_{C'})=(\id_C,\id_{0})=\id_{FGD}$ it follows that $G$ is naturally full by \cite[Theorem 2.6]{ACMM06}, hence in particular semiseparable.
\end{rmk}
\begin{rmk}As it happens for the separable case, the fact that the associated (co)monad is (co)separable does not imply that the right (left) adjoint is semiseparable, i.e. the converse of Lemma \ref{lem:monad} is not necessarily true. To see this, note that if $(F,G)$ is an adjunction with $G$ (resp. $F$) fully faithful, then the associated monad (resp. comonad) is always idempotent (this will be proved in Corollary  \ref{cor:ffmonad}, resp. Corollary \ref{cor:ffcomonad}) whence separable (resp. coseparable). However $F$ (resp. $G$) needs not to be semiseparable in this case. For instance, we consider the usual adjunction $(\varphi^*,\varphi_*)$ attached to a ring homomorphism $\varphi:R\to S$ (we will be back on it in Subsection \ref{es:extens}). In \cite[Example 3.3]{ACMM06} it is shown an example of a ring epimorphism $\varphi:R\to S$ (in this case $\varphi_*$ is fully faithful) such that the extension of scalars functor $\varphi^*$ is full, but not naturally full, thus $\varphi^*$ is not semiseparable by Proposition \ref{prop:sep}.
\end{rmk}

The following result characterizes the semiseparability of a right adjoint functor in terms of properties of the comparison functor and of the forgetful functor from the Eilenberg-Moore category of modules over the associated monad. We remark that by Proposition \ref{prop:sep} the separability of the forgetful functor coincides with its semiseparability as it is faithful.

\begin{invisible} This follows also from $(1)\Leftrightarrow (4)$ of Corollary \ref{cor:conserv}, as it is conservative (\cite[Proposition 4.1.4]{BorII94}).
\proof
Let $f:(C,\mu_C)\to (D,\mu_D)$ be a morphism in $\cc_{GF}$ such that $f$ is an isomorphism in $\cc$. Then, from $$f\circ\mu_C\circ GF(f^{-1})=\mu_D\circ GF(f)\circ GF(f^{-1})=\mu_D=f\circ f^{-1}\circ\mu_D,$$ we have $\mu_C\circ GF(f^{-1})=f^{-1}\circ\mu_D$, hence $f^{-1}:(D,\mu_D)\to (C,\mu_C)$ is a morphism of $GF$-modules, inverse to $f$.
\endproof
\end{invisible}

\begin{thm}\label{thm:ssepMonad}
Let $(F:\cc\to\dd, G:\dd\to\cc)$ be an adjunction. Then, $G$ is semiseparable if and only if the forgetful functor $U_{GF}:\cc_{GF}\to\cc$ is separable (equivalently, the monad $(GF,G\epsilon F,\eta)$ is separable) and the comparison functor $K_{GF}:\dd\to \cc_{GF}$ is naturally full.
\end{thm}

\proof Set $U:=U_{GF}$ and $K:=K_{GF}$. Let $\eta$ and $\epsilon$ be the unit and counit of $(F,G)$ respectively. Assume $G$ is semiseparable. By Theorem \ref{thm:rafael} and Lemma \ref{lem_B} (ii), there is a natural transformation $\gamma : \id _{\dd}\to FG$ such that $G\epsilon\circ G\gamma = \id _{G}$. By Lemma \ref{lem:monad} (i), $(GF,G\epsilon F,\eta)$ is a separable monad, and by \cite[2.9 (1)]{BBW09}, the separability of this monad is equivalent to the separability of $U$. We now prove that $K:\dd\to\cc_{GF}$ is naturally full. Let $h:KX\to KY$ be a morphism in $\cc_{GF}$ and set $h':=\epsilon_Y\circ FUh\circ\gamma_X$. Then, since $U\circ K =G$, which is semiseparable by assumption, we obtain
\begin{equation*}
UKh'= G(\epsilon_Y\circ FUh\circ\gamma_X)=(G\epsilon_Y\circ GFUh )\circ G\gamma_X = Uh\circ G\epsilon_X \circ G\gamma_X = Uh.
\end{equation*}
So $K$ is full, as $Kh'=h$. Moreover, since $U$ is faithful and $UK$ is semiseparable, by Lemma \ref{lem_A} (i) $K$ is semiseparable. By Proposition \ref{prop:sep} (ii) this means that $K$ is naturally full.
Conversely, if $U$ is separable and $K$ is naturally full, then by Corollary \ref{cor:fattoriz} $G=U\circ K$ is semiseparable.
\endproof

\begin{rmk}\label{rmk:compar-fattoriz}
	Let $F\dashv G:\dd\to\cc$ be an adjunction. If $G$ is semiseparable, let $e:\id_\dd\to\id_\dd$ be the associated idempotent natural transformation. Then, by Theorem \ref{thm:coidentifier} there is a unique functor $G_{e}:\dd%
	_{e}\rightarrow \cc$ (necessarily separable) such that $G=G_{e}\circ H$, where $H:\dd\rightarrow \dd_{e}$ is the quotient functor onto the coidentifier category $\dd_e$, which in turn is naturally full. By Theorem \ref{thm:ssepMonad} $G$ also factors as $U_{GF}\circ K_{GF}$ where $U_{GF}$ is separable and $K_{GF}$ is naturally full. These two factorizations of $G$ as a naturally full functor followed by a separable one are related, in view of Theorem \ref{thm:coidentifier}, by a unique functor $(K_{GF})_e:\dd_e\to \mathcal{C}_{GF}$ (necessarily fully faithful) such that $(K_{GF})_e\circ H=K_{GF}$ and $U_{GF}\circ (K_{GF})_e=G_e$.
The same result also establishes that the idempotent natural transformation associated to $K_{GF}$ is still $e$.
	\begin{gather*}
	\xymatrix{\dd\ar[r]^H\ar[d]_{K_{GF}}&\dd_e \ar@{.>}[dl]|{(K_{GF})_e}\ar[d]^{G_e}\\\mathcal{C}_{GF}\ar[r]_{U_{GF}}&\cc
	}
	\end{gather*}
\end{rmk}
As a consequence of Theorem \ref{thm:ssepMonad} we can now recover similar characterizations for separable, naturally full and fully faithful right adjoint. Let us start with the separable case.

\begin{cor}[cf. {\cite[proof of Proposition 3.5]{Chen15} and \cite[Proposition 2.16]{AGM15}}] \label{cor:sepmonad}
Let $(F:\cc\to\dd, G:\dd\to\cc)$ be an adjunction. Then, $G$ is separable if and only if the forgetful functor $U_{GF}:\cc_{GF}\to\cc$ is separable (equivalently, the monad $(GF,G\epsilon F,\eta)$ is separable) and the comparison functor $K_{GF}:\dd\to \cc_{GF}$ is fully faithful (i.e. $G$ is premonadic).

\end{cor}
\begin{proof}Set $U:=U_{GF}$ and $K:=K_{GF}$. By Proposition \ref{prop:sep} (i), $G$ is separable if and only if it is semiseparable and faithful. By Theorem \ref{thm:ssepMonad}, $G$ is semiseparable if and only if $U$ is separable and $K$ is naturally full. Since $G=U\circ K$ and $U$ is faithful, we get that $G$ is faithful if and only if $K$ is faithful. Putting all together we get that $G$ is separable if and only if $U$ is separable and $K$ is both naturally full and faithful. The latter means that $K$ is fully faithful, i.e. $G$ is premonadic.
\end{proof}

We now provide a new characterization of natural fullness of a right adjoint functor in terms of idempotence of its adjunction/monad and natural fullness of the comparison functor.

\begin{cor}\label{cor:natfulmonad}
The following are equivalent for an adjunction $(F:\cc\to\dd, G:\dd\to\cc)$.
\begin{enumerate}[(1)]
  \item $G$ is naturally full.
  \item The adjunction $(F,G)$ is idempotent and $G$ is semiseparable.
  \item The forgetful functor $U_{GF}:\cc_{GF}\to \cc$ is fully faithful (i.e. the monad $(GF,G\epsilon F,\eta)$ is idempotent) and the comparison functor $K_{GF}:\dd\to \cc_{GF}$ is naturally full.
\end{enumerate}
\end{cor}
\begin{proof}
Let $\eta :\id_{%
\cc}\rightarrow GF$ be the unit and let  $\epsilon :FG\rightarrow \mathrm{%
Id}_{\dd}$ be the counit of the adjunction $(F, G)$.

$(1)\Rightarrow(2)$. If $G$ is naturally full, by Rafael-type Theorem for naturally full functors \cite[Theorem 2.6 (2)]{ACMM06}, there is a natural transformation $\gamma :\id_{\dd}\to FG$ such that $\gamma\circ\epsilon =\id_{FG}.$ Thus $G\gamma \circ G\epsilon =\id_{GFG}.$ On the other hand we have the triangular identity $G\epsilon \circ \eta G =\id_{G}$ and hence $G\epsilon $ is invertible and $(F,G)$ is idempotent. Moreover $G$ is semiseparable by Proposition \ref{prop:sep} (ii).


$(2)\Rightarrow(3)$. It follows from the definition of an idempotent adjunction and from Theorem \ref{thm:ssepMonad}.

$(3)\Rightarrow(1)$.  Since $G=U_{GF}\circ K_{GF}$ we get that $G$ is naturally full as a composition of naturally full functors, see \cite[Proposition 2.3]{ACMM06}.
\end{proof}

Now, putting together the above corollaries, we recover the characterization for a fully faithful right adjoint. Although it is well-known that $(F:\cc\to\dd, G:\dd\to\cc)$ is an idempotent adjunction, that is,
 the forgetful functor $U_{GF}:\cc_{GF}\to \cc$ (resp. $U^{FG}:\dd^{FG}\to \dd$) is fully faithful, provided the functor $G$ (resp. $F$) is fully faithful, see e.g. \cite[Proposition 2.5]{AGM15}, the equivalence $(1)\Leftrightarrow(2)$ in Corollary \ref{cor:ffmonad} (resp. Corollary \ref{cor:ffcomonad}) is new as far as we know.

\begin{cor}\label{cor:ffmonad}
The following are equivalent for an adjunction $(F:\cc\to\dd, G:\dd\to\cc)$.
\begin{enumerate}[(1)]
  \item $G$ is fully faithful.
  \item The forgetful functor $U_{GF}:\cc_{GF}\to \cc$ is fully faithful (i.e. the monad $(GF,G\epsilon F,\eta)$ is idempotent) and the comparison functor $K_{GF}:\dd\to \cc_{GF}$ is fully faithful (i.e. $G$ is premonadic).
  \item The adjunction $(F,G)$ is idempotent and the comparison functor $K_{GF}:\dd\to \cc_{GF}$ is an equivalence (i.e. $G$ is monadic).
\end{enumerate}
\end{cor}

\begin{proof}
 $(1)\Leftrightarrow (2)$. Put together Corollary \ref{cor:sepmonad} and Corollary \ref{cor:natfulmonad}.

 $(1)\Leftrightarrow (3)$.  This follows by \cite[Proposition 2.5]{AGM15}.
\end{proof}

Let us consider the dual context of Theorem \ref{thm:ssepMonad}. We characterize the semiseparability of a left adjoint functor in terms of the natural fullness of the cocomparison functor and of the separability of the forgetful functor from the Eilenberg-Moore category of comodules over the associated comonad.

\begin{thm}\label{thm:ssep-comonad}
Let $(F:\cc\to\dd, G:\dd\to\cc)$ be an adjunction. Then, $F$ is semiseparable if and only if the forgetful functor $U^{FG}: \dd^{FG} \to \dd$ is separable (equivalently, the comonad $(FG, F\eta G, \epsilon )$ is coseparable) and the cocomparison functor $K^{FG}:\cc\to \dd^{FG}$ is naturally full.
\end{thm}

\begin{invisible}
Let $\eta$ and $\epsilon$ be the unit and counit of the adjunction $(F,G)$ respectively, and recall that $U^{FG}\circ K^{FG}=F$. Set $U:=U^{FG}$ and $K:=K^{FG}$.\par
Assume $F$ is semiseparable. From Theorem \ref{thm:rafael} and Lemma \ref{lem_B} (i) it follows that there exists a natural transformation $\nu :GF\rightarrow \id _{\cc}$ such that $F\nu\circ F\eta = \id _{F}$. By Lemma \ref{lem:monad} (ii) $(FG, F\eta G, \epsilon)$ is a coseparable comonad, and by \cite[2.9 (2)]{BBW09} the coseparability of this comonad is equivalent to the separability of $U$. We now prove that $K:\cc\to\dd^{FG}$ is naturally full. Let $h:KX\to KY$ be a morphism in $\dd^{FG}$ and set $h':=\nu_Y\circ GUh\circ\eta_X : X\to Y$. Then, since $U\circ K =F$, which is semiseparable by assumption, we obtain
\begin{equation*}
UKh'= F(\nu_Y\circ GUh\circ\eta_X)=F\nu_Y\circ (FGUh \circ F\eta_X) = F\nu_Y \circ F\eta_Y\circ Uh = Uh,
\end{equation*}
hence $K$ is full, as $Kh'=h$. Moreover, since $U$ is faithful and $UK$ is semiseparable, by Lemma \ref{lem_A} (i) $K$ is semiseparable. By Proposition \ref{prop:sep} (ii) this is equivalent to $K$ is naturally full.
\par
Conversely, if $U$ is separable and $K$ is naturally full, then by Corollary \ref{cor:fattoriz} $F=U\circ K$ is semiseparable.\par

\end{invisible}
Dually to Remark \ref{rmk:compar-fattoriz}, if $F\dashv G:\dd\to\cc$ is an adjunction with $F$ semiseparable and $e:\id_\cc\to\id_\cc$ is the idempotent natural transformation associated to $F$, then Theorem \ref{thm:coidentifier} and Theorem \ref{thm:ssep-comonad} yield two factorizations $F_{e}\circ H=F=U^{FG}\circ K^{FG}$ of $F$ as a naturally full functor followed by a separable one, and they are related by a unique functor $(K^{FG})_e:\cc_e\to \mathcal{D}^{FG}$ (necessarily fully faithful) such that $(K^{FG})_e\circ H=K^{FG}$ and $U^{FG}\circ (K^{FG})_e=F_e$.\medskip

For future reference, we now state the dual of Corollaries \ref{cor:sepmonad}, \ref{cor:natfulmonad} and \ref{cor:ffmonad}.

\begin{cor}\label{cor:sepcomonad} Let $(F:\cc\to\dd, G:\dd\to\cc)$ be an adjunction. Then, $F$ is separable if and only if the forgetful functor $U^{FG}: \dd^{FG} \to \dd$ is separable (equivalently, the comonad $(FG, F\eta G, \epsilon )$ is coseparable) and the cocomparison functor $K^{FG}:\cc\to \dd^{FG}$ is fully faithful (i.e. $F$ is precomonadic).
\end{cor}
\begin{invisible} By Proposition \ref{prop:sep} (i), $F$ is separable if and only if it is semiseparable and faithful. By Theorem \ref{thm:ssep-comonad}, $F$ is semiseparable if and only if $U^{FG}$ is separable and $K^{FG}$ is naturally full. Since $F=U^{FG} K^{FG}$ and $U^{FG}$ is faithful, we get that $F$ is faithful if and only if $K^{FG}$ is faithful, hence $F$ is separable if and only if $U^{FG}$ is separable and $K^{FG}$ is both naturally full and faithful. Since $K^{FG}$ is both naturally full and faithful if and only if it is fully faithful, the thesis follows.

\end{invisible}

%

\begin{cor}\label{cor:natfulcomonad}
The following are equivalent for an adjunction $(F:\cc\to\dd, G:\dd\to\cc)$.
\begin{enumerate}[(1)]
  \item $F$ is naturally full.
  \item The adjunction $(F,G)$ is idempotent and $F$ is semiseparable.
  \item The forgetful functor $U^{FG}:\dd^{FG}\to \dd$ is fully faithful (i.e. the comonad $(FG,F\eta G,\epsilon )$ is idempotent) and the cocomparison functor $K^{FG}:\cc\to \dd^{FG}$ is naturally full.
\end{enumerate}
\end{cor}
\begin{invisible}
Let $\eta :\id_{%
\cc}\rightarrow GF$ be the unit and let  $\epsilon :FG\rightarrow \mathrm{%
Id}_{\dd}$ be the counit of the adjunction $(F, G)$.

$(1)\Rightarrow(2)$. If $F$ is naturally full, by \cite[Theorem 2.6 (1)]{ACMM06}, there is a natural transformation $\nu :GF\to\id_{\cc}$ such that $\eta\circ\nu =\id_{GF}.$ Thus $F\eta\circ F\nu =\id_{FGF}.$ On the other hand we have the triangular identity $\epsilon F\circ F\eta =\id_{F}$ and hence $F\eta$ is invertible and $(F,G)$ is idempotent. Moreover $G$ is semiseparable by Proposition \ref{prop:sep} (ii).


$(2)\Rightarrow(3)$. It follows from the definition of an idempotent adjunction and from Theorem \ref{thm:ssep-comonad}.

$(3)\Rightarrow(1)$.  Since $U^{FG}\circ K^{FG}=F$ we get that $F$ is naturally full as a composition of naturally full functors, see \cite[Proposition 2.3]{ACMM06}.



\end{invisible}

\begin{cor}\label{cor:ffcomonad}
The following are equivalent for an adjunction $(F:\cc\to\dd, G:\dd\to\cc)$.
\begin{enumerate}[(1)]
  \item $F$ is fully faithful.
  \item The comonad $(FG,F\eta G,\epsilon )$ is idempotent and the cocomparison functor $K^{FG}:\cc\to \dd^{FG}$ is fully faithful (i.e. $F$ is precomonadic).
  \item The adjunction $(F,G)$  is idempotent and the cocomparison functor $K^{FG}:\cc\to \dd^{FG}$ is an equivalence (i.e. $F$ is comonadic).
\end{enumerate}
\end{cor}

\begin{invisible}
 $(1)\Leftrightarrow (2)$. It follows from Corollary \ref{cor:sepcomonad} and Corollary \ref{cor:natfulcomonad}.

 $(1)\Leftrightarrow (3)\Leftrightarrow (4)$  It is the dual of \cite[Proposition 2.5]{AGM15}.
\end{invisible}

We include here a consequence of Corollaries \ref{cor:natfulmonad} and \ref{cor:natfulcomonad} that will be used later on.

\begin{cor}\label{cor:idempotent}
  Let $(F,G) $ be an idempotent adjunction. Then, $F$ (resp. $G$) is semiseparable if and only if it is naturally full.
\end{cor}
\begin{invisible}
It follows by $(1)\Leftrightarrow (2)$ in Corollary \ref{cor:natfulcomonad} (resp. Corollary \ref{cor:natfulmonad}). 
\end{invisible}


\subsection{Adjoint triples and bireflections}\label{sub:adjtriple}
Let $\cc$ and $\dd$ be categories. Recall that an \emph{adjoint triple} $F\dashv G\dashv H:\cc\to\dd$ of functors is a triple of functors $F,H:\cc\to\dd$ and $G:\dd\to\cc$ such that $F\dashv G$ and  $G\dashv H$. The following result, which is new to the best of our knowledge, shows how semiseparable, separable and naturally full functors behave with respect to adjoint triples.

\begin{prop}\label{prop:adj-triples}
Let $F\dashv G\dashv H:\cc\to\dd$ be an adjoint triple. Then, $F$ is semiseparable (resp. separable, naturally full) if and only if so is $H$.
\end{prop}

\proof  We denote by $\eta^l$, $\epsilon^l$ and $\eta^r$, $\epsilon^r$ the unit and the counit of the adjunction $F\dashv G$ and of the adjunction $G\dashv H$, respectively. We just prove the ``only if'' part of the statement. For the other direction consider the adjoint triple $H^\op\dashv G^\op\dashv F^\op$ together with Remark \ref{rmk:opposemi}. To a natural transformation $\nu^l:GF\to \id _\cc$ we can attach the natural transformation $\gamma^r:= GH\nu^l\circ G\eta^r F\circ\eta^l:\id_\cc\to GH$ such that
\begin{equation}\label{eq:adj-triples}
\epsilon^r\circ\gamma^r=\epsilon^r\circ GH\nu^l\circ G\eta^r F\circ\eta^l=\nu^l\circ \epsilon^r GF\circ G\eta^r F\circ\eta^l=\nu^l\circ\eta^l.
\end{equation}
Assume $F$ is semiseparable. By Theorem \ref{thm:rafael} (i), there exists a natural transformation $\nu^l:GF\to \id _\cc$ such that $\eta^l\circ\nu^l\circ\eta^l=\eta^l$. Define $\gamma^r:\id_\cc\to GH$ that fulfils \eqref{eq:adj-triples} as above. We show that it is the required natural transformation of Theorem \ref{thm:rafael} (ii) such that $\epsilon^r\circ\gamma^r\circ\epsilon^r=\epsilon^r$. Indeed, by naturality of $\epsilon^r$, we have $\epsilon^r\circ\gamma^r\circ\epsilon^r=\nu^l\circ\eta^l\circ\epsilon^r=\epsilon^r\circ\nu^l GH\circ\eta^l GH=\epsilon^r$, where the last equality follows from $(1)\Leftrightarrow (3)$ of Lemma \ref{lem_B} (i).
    \begin{invisible}
    Conversely, assume $H$ is semiseparable. By Theorem \ref{thm:rafael} (ii), there exists a natural transformation $\gamma^r:\id_\cc\to GH$ such that $\epsilon^r\circ\gamma^r\circ\epsilon^r=\epsilon^r$. Define $\nu^l:= \epsilon^r\circ G\epsilon^l H\circ GF\gamma^r:GF\to\id_\cc$. Then, by naturality of $\eta^l$, $\epsilon^r$ and $\gamma^r$, we have $\eta^l\circ\nu^l\circ\eta^l=\eta^l\circ\epsilon^r\circ G\epsilon^l H\circ GF\gamma^r\circ\eta^l=\eta^l\circ\epsilon^r\circ G\epsilon^l H\circ \eta^l GH\circ\gamma^r=\eta^l\circ\epsilon^r\circ\gamma^r=\epsilon^r GF\circ GH\eta^l\circ\gamma^r=\epsilon^r GF\circ\gamma^r GF\circ\eta^l=\eta^l$, where the last equality follows from $(1)\Leftrightarrow (3)$ of Lemma \ref{lem_B} (ii). Thus, by Theorem \ref{thm:rafael} (i) $F$ is semiseparable.
    \end{invisible}

    If $F$ is separable, by Rafael Theorem, there exists a natural transformation $\nu^l:GF\to \id _\cc$ such that $\nu^l\circ\eta^l=\id$. Then, for $\gamma^r$ defined as above and \eqref{eq:adj-triples}, we have $\epsilon^r\circ\gamma^r=\nu^l\circ\eta^l=\id$ so that $H$ is separable again by Rafael Theorem.

    Assume $F$ is naturally full. By \cite[Theorem 2.6 (1)]{ACMM06}, there exists a natural transformation $\nu^l: GF\to \id _\cc$ such that $\eta^l\circ\nu^l=\id_{GF}$. Define $\gamma^r:\id_\cc\to GH$ as above. Observe that, from $\eta^lG\circ\nu^l G=\id_{GFG}$ and $G\epsilon^l\circ \eta^lG=\id_G$, it follows that $(\eta^l G)^{-1}=\nu^lG=G\epsilon^l$. Then, by naturality of $\gamma^r$ and $\eta^r$ we have $\gamma^r\circ\epsilon^r=GH\epsilon^r\circ\gamma^r GH=GH\epsilon^r\circ GH\nu^lGH\circ G\eta^rFGH\circ\eta^lGH=GH\epsilon^r\circ GHG\epsilon^lH\circ G\eta^rFGH\circ\eta^lGH=GH\epsilon^r\circ G(HG\epsilon^l\circ\eta^rFG)H\circ\eta^lGH=GH\epsilon^r\circ G\eta^rH\circ G\epsilon^lH\circ\eta^lGH=\id_{GH}\circ\id_{GH}=\id_{GH}$.
    \begin{invisible}
    Conversely, assume $H$ is naturally full. By \cite[Theorem 2.6 (2)]{ACMM06} there exists a natural transformation $\gamma^r:\id_\cc\to GH$ such that $\gamma^r\circ\epsilon^r=\id_{GH}$. Define $\nu^l:GF\to\id_\cc$ as in (i) and observe that $(\epsilon^r G)^{-1}=\gamma^r G=G\eta^r$. Then, by naturality of $\nu^l$ and $\epsilon^l$ we have $\eta^l\circ\nu^l=\nu^lGF\circ GF\eta^l=\epsilon^rGF\circ G\epsilon^lHGF\circ GF\gamma^rGF\circ GF\eta^l=\epsilon^r GF\circ G(\epsilon^l HG\circ F\gamma^rG)F\circ GF\eta^l=\epsilon^rGF\circ G\eta^r F\circ G\epsilon^lF\circ GF\eta^l=\id_{GF}\circ\id_{GF}=\id_{GF}$, hence $F$ is naturally full.
    \end{invisible}
\endproof

\begin{rmk}We already observed that a functor is fully faithful if and only if it is at the same time separable and naturally full. Thus, by Proposition \ref{prop:adj-triples}, we recover the well-known result that in an adjoint triple $F\dashv G\dashv H$, the functor $F$ is fully faithful if and only if so is $H$, see e.g. \cite[Proposition 3.4.2]{Bor94}. Adjoint triples $F\dashv G\dashv H$ where $F$ and $H$ are fully faithful are called \emph{fully faithful adjoint triples}.
\end{rmk}

We will apply Proposition \ref{prop:adj-triples} in Subsection \ref{es:extens} to an adjoint triple associated to a ring morphism and in Subsection \ref{sub:rightHopf} in the study of a particular adjoint triple attached to a bialgebra.\medskip


Now, recall that a functor $G:\dd\to\cc$ is called \emph{Frobenius} if there exists a functor $F:\cc\to\dd$ which is both a left and a right adjoint to $G$. Thus, a Frobenius functor $G:\dd\to\cc$ fits into an adjoint triple $F\dashv G\dashv F:\cc\to\dd$ where the left and right adjoint $F$ are equal. As a consequence of Theorem \ref{thm:rafael} and Lemma \ref{lem_B}, in the following result we obtain necessary and sufficient conditions for the semiseparability of a Frobenius functor. This is a semiseparable version of \cite[Proposition 49]{CMZ02} for separable Frobenius functors and of \cite[Proposition 2.7]{ACMM06} for naturally full Frobenius functors.
\begin{prop}\label{prop:frob-ssep}
Let $F:\cc\to\dd$ be a Frobenius functor, with left and right adjoint $G:\dd\to\cc$. Denote by $\eta^l$, $\epsilon^l$ and by $\eta^r$, $\epsilon^r$ the unit and the counit of the adjunctions $(F,G)$ and $(G,F)$, respectively. Then, the following assertions are equivalent:
\begin{itemize}
\item[(i)] $F$ is semiseparable.
\item[(ii)] There exists a natural transformation $\alpha: G\to G$ such that one of the following equivalent conditions holds:
    $$\eta^l\circ \epsilon^r\circ \alpha F\circ\eta^l=\eta^l;\qquad F\epsilon^r\circ F\alpha F\circ F\eta^l =\id_F;\qquad \epsilon^rG\circ \alpha FG\circ\eta^l G=\id_G.$$
\item[(iii)] There exists a natural transformation $\beta: F\to F$ such that one of the following equivalent conditions holds:
$$\eta^l\circ \epsilon^r\circ G\beta \circ\eta^l=\eta^l;\qquad F\epsilon^r\circ FG\beta\circ F\eta^l =\id_F;  \qquad \epsilon^rG\circ G\beta G\circ\eta ^l G=\id_G.$$
\item[(iv)]There exists a natural transformation $\alpha': G\to G$ such that one of the following equivalent conditions holds:
$$\epsilon^r\circ \alpha' F\circ\eta^l\circ\epsilon^r = \epsilon^r;\qquad F\epsilon^r\circ F\alpha' F\circ F\eta^l = \id _{F};\qquad \epsilon^r G\circ \alpha' FG\circ\eta^l G = \id _{G}.$$
\item[(v)]There exists a natural transformation $\beta': F\to F$ such that one of the following equivalent conditions holds:
$$\epsilon^r\circ G\beta' \circ\eta^l\circ\epsilon^r = \epsilon^r;\qquad F\epsilon^r\circ FG\beta' \circ F\eta^l = \id _{F};\qquad \epsilon^r G\circ G\beta' G \circ\eta^l G = \id _{G}.$$
\end{itemize}
\end{prop}
\proof
$(i)\Leftrightarrow (ii)\Leftrightarrow (iii)$ By \cite[Proposition 10]{CMZ02} applied to the adjunction $(G,F)$ we have the following bijective correspondences:
$$\mathrm{Nat}(GF,\id_\cc)\cong\mathrm{Nat}(G,G)\cong\mathrm{Nat}(F,F)\cong\mathrm{Nat}(\id_\dd,FG).$$ Explicitly, for any natural transformation $\nu:GF\to\id_\cc$ there are unique natural transformations $\alpha: G\to G$, $\beta: F\to F$ such that
\begin{equation}\label{eq:frob-nu}
\epsilon^r\circ\alpha F =\nu=\epsilon^r\circ G\beta .
\end{equation}
Apply Theorem \ref{thm:rafael} and Lemma \ref{lem_B} to the adjunction $(F,G)$ and then \eqref{eq:frob-nu} to the induced natural transformation $\nu^l:GF\to\id_\cc$ such that $\eta^l\circ\nu^l\circ\eta^l =\eta^l$.\\
$(i)\Leftrightarrow (iv)\Leftrightarrow (v)$ By \cite[Proposition 10]{CMZ02} applied to the adjunction $(F,G)$, for any natural transformation $\gamma:\id_\cc\to GF$ there are unique natural transformations $\alpha':G\to G$, $\beta':F\to F$ such that
\begin{equation}\label{eq:frob-gamma}
\alpha'F\circ\eta^l =\gamma= G\beta'\circ\eta^l.
\end{equation}
Consider the adjunction $(G,F)$ and apply Theorem \ref{thm:rafael} and Lemma \ref{lem_B}. Then, apply \eqref{eq:frob-gamma} to the induced natural transformation $\gamma^r:\id_\cc\to GF$ such that $\epsilon^r\circ\gamma^r\circ\epsilon^r = \epsilon^r$.
\endproof

\begin{rmk}\label{rmk:frob}
	A semiseparable functor is not necessarily Frobenius. Indeed, from \cite[Example 18, item 6, page 323]{CMZ02} let $G$ be a finite group and consider the group algebra $A=\Bbbk G$ over a field $\Bbbk$.
Then $A$ is a Frobenius $\Bbbk$-algebra and hence the restriction of scalars functor $\varphi_*:A\text{-}\mathrm{Mod}\to\Bbbk\text{-}\mathrm{Mod}$ is Frobenius, cf. \cite[Theorem 28, item 3]{CMZ02}. However if $\mathrm{char}(\Bbbk)$ divides $\vert G\vert$, the extension $A/\Bbbk $ is not separable so that $\varphi_*$ is not separable and therefore not even semiseparable as it is faithful. See Subsection \ref{es:extens} for further results on $\varphi_*$.
\end{rmk}

Next aim is to study semiseparable (co)reflections. Recall that
\begin{itemize}
  \item a functor admitting a fully faithful left adjoint is called a \emph{coreflection}, see \cite{Ber07};
  \item a functor with a fully faithful right adjoint is called a \emph{reflection};
  \item a functor $G:\dd\to\cc$ is called a \emph{bireflection} if it has a left and right adjoint equal, say $F:\cc\to \dd$, which is fully faithful and satisfies the coherent condition $\gamma\circ \epsilon=\id$ where $\epsilon:FG\to\id$ is the counit of $F\dashv G$ while $\gamma:\id\to  FG$ is the unit of $G\dashv F$, cf. \cite[Definition 8]{FOPTST99}.
\end{itemize} Being a coreflection (respectively a reflection) is equivalent to the fact that the unit (respectively counit) of the corresponding adjunction is an isomorphism, see \cite[Proposition 3.4.1]{Bor94}. The adjoint of the inclusion of a (co)reflective subcategory is a typical example of (co)reflection. In Theorem \ref{thm:frobenius} we will see how semiseparable (co)reflections $G:\dd\to\cc$ naturally give rise to fully faithful adjoint triples. Bireflective subcategories of a given category $\cc$ provide examples of bireflections. Recall that an idempotent $f:X\to X$ in a category $\cc$ is \emph{split} if there exist $g:X\to Y$ and $h:Y\to X$ such that $h\circ g=f$ and $g\circ h=\id_Y$. The splitting is unique up to isomorphism. In \cite{FOPTST99} an endo-natural transformation whose components are all split idempotents is called a \emph{split-idempotent natural transformation}. It is known that bireflective subcategories correspond bijectively to split-idempotent natural transformations $e:\id_\cc\to\id_\cc$ with specified splitting, \cite[Theorem 13]{FOPTST99}.  This fact is connected to the functor $H:\cc\to \cc_e$ of Theorem \ref{thm:coidentifier} which comes out to be a bireflection in case $e$ splits naturally, as we will see in Proposition \ref{prop:Hcorefl}. 

Let us see how the notions of (co)reflection and bireflection behave in connection to semiseparability. In particular, in the following Proposition we observe their behaviour with respect to a semiseparable composition of functors, cf. \cite[Proposition 2.4]{ACMM06} for the naturally full case.
\begin{prop}\label{prop:HGsscoref}
Let $G:\dd\to\cc$, $H:\cc\to\e$ be functors, and assume that $G$ is a (co)reflection. If $H\circ G:\dd\to\e$ is semiseparable, then $H$ is semiseparable.
\end{prop}
\proof
Assume that $G$ is a coreflection with a fully faithful left adjoint $F$.
If $HG$ is semiseparable, since $F$ is fully faithful (whence naturally full), then $HGF$ is semiseparable by Lemma \ref{lem:comp}.
The unit $\eta:\id_{\cc}\to GF$ of the adjunction $(F,G)$ is an isomorphism, so that $H\eta:H\to HGF$ is an isomorphism. By Corollary \ref{cor:isomorphic}, $H$ is semiseparable.
If $G$ is a reflection, the proof is similar.
\endproof

The next result provides a characterization of semiseparable (co)reflections. Surprisingly it involves the notions of naturally full and Frobenius functor as well as the one of bireflection.

\begin{thm}
\label{thm:frobenius}The following are equivalent for a functor $G:\dd\rightarrow \cc$.

\begin{enumerate}
\item[$\left( 1\right) $] $G$ is a naturally full coreflection.

\item[$\left( 2\right) $] $G$ is a semiseparable coreflection.

\item[$\left( 3\right) $] $G$ is a bireflection.

\item[$\left( 4\right) $] $G$ is a Frobenius coreflection.

\item[$\left( 5\right) $] $G$ is a naturally full reflection.

\item[$\left( 6\right) $] $G$ is a semiseparable reflection.

\item[$\left( 7\right) $] $G$ is a Frobenius reflection.

\end{enumerate}
\end{thm}

\begin{proof} We prove the equivalence between  $(1),(2),(3)$ and $(4)$. Assume that $G$ is a coreflection. Denote by $F$ the left adjoint of $G$, by $\eta :\id_{%
\cc}\rightarrow GF$ the unit and by $\epsilon :FG\rightarrow \mathrm{%
Id}_{\dd}$ the counit of the adjunction $(F, G)$. Since $F$ is fully faithful, we get that $\eta $ is invertible. Therefore, from $\epsilon
F\circ F\eta =\id_{F}$ and $G\epsilon \circ \eta G=\id_{G},$
we get $\left( F\eta \right) ^{-1}=\epsilon F$ and $\left( \eta G\right)
^{-1}=G\epsilon .$

$\left( 1\right) \Leftrightarrow \left( 2\right) $. Since $\eta$ is invertible, the adjunction $(F,G)$ is idempotent and Corollary \ref{cor:idempotent} applies.

$\left( 2\right) \Rightarrow \left( 3\right) $. If $G$ is
semiseparable, by Theorem \ref{thm:rafael} (ii) there is a natural transformation $\gamma :%
\id_{\dd}\rightarrow FG$ such that $\epsilon \circ \gamma
\circ \epsilon =\epsilon $. By Lemma \ref{lem_B}, we have $\epsilon F\circ \gamma F=%
\id_{F}$ and $G\epsilon \circ G\gamma =\id_{G}$ so that,
\begin{eqnarray*}
F\left( \eta ^{-1}\right) \circ \gamma F &=&\left( F\eta \right) ^{-1}\circ
\gamma F=\epsilon F\circ \gamma F=\id_{F}, \\
\eta ^{-1}G\circ G\gamma  &=&\left( \eta G\right) ^{-1}\circ G\gamma
=G\epsilon \circ G\gamma =\id_{G}.
\end{eqnarray*}%
This means that $\left( G,F\right) $ is an adjunction with unit $\gamma :%
\id_{\dd}\rightarrow FG$ and counit $\eta
^{-1}:GF\rightarrow \id_{\cc}$. The equality $\eta^{-1}G=G\epsilon$ implies the coherent condition $\gamma\circ \epsilon=\id$ by \cite[Proposition 10]{FOPTST99}.

$\left( 3\right) \Rightarrow \left( 4\right) $. It is trivial.

$\left( 4\right) \Rightarrow \left( 1\right) $. If $G$ is
Frobenius, then there are a unit $\eta ^{\prime }:\id_{\dd%
}\rightarrow FG$ and a counit $\epsilon ^{\prime }:GF\rightarrow \id_{%
\cc}$ of the adjunction $(G, F)$. Set $\sigma :=\epsilon ^{\prime }G\circ \eta G:G\to G$
and note that $\sigma \circ G\epsilon =\epsilon ^{\prime }G\circ \eta G\circ
G\epsilon =\epsilon ^{\prime }G\circ \eta G\circ \left( \eta G\right)
^{-1}=\epsilon ^{\prime }G$. If we set $\gamma :=F\sigma \circ \eta ^{\prime
}$, we obtain
\begin{equation*}
\gamma \circ \epsilon =F\sigma \circ \eta ^{\prime }\circ \epsilon \overset{\text{nat.}\eta'}{=}F\sigma
\circ FG\epsilon \circ \eta ^{\prime }FG=F\left( \sigma \circ G\epsilon
\right) \circ \eta ^{\prime }FG=F\epsilon ^{\prime }G\circ \eta ^{\prime }FG=%
\id_{FG}.
\end{equation*}%
By \cite[Theorem 2.6]{ACMM06}, we conclude that $G$ is naturally full.

The implications $(5)\Leftrightarrow (6)\Rightarrow (3)\Rightarrow(7)\Rightarrow(5)$ follow dually.
\begin{invisible}
In order to simplify the application of the results we need in this proof, we use $F$ in proving the implications above instead of $G$. On the other hand $G$ will denote a right adjoint of $F$ as above. If $G$ is fully faithful, then $\epsilon $ is an isomorphism. Therefore, from the triangular identities $\epsilon F\circ F\eta =\id_{F}$ and $G\epsilon \circ \eta G=\id_{G},$
we get $\left(\epsilon F  \right) ^{-1}=F\eta$ and $\left( G\epsilon\right)
^{-1}=\eta G .$

$\left( 5\right) \Leftrightarrow \left( 6\right) $. Since $\epsilon$ is invertible, the adjunction $(F,G)$ is idempotent. Thus we can apply Corollary \ref{cor:idempotent}.

$\left( 6\right) \Rightarrow \left( 3\right) $. Since $F$ is
semiseparable, by Theorem \ref{thm:rafael} (i) there is a natural transformation $\nu :%
GF\to\id_{\cc}$ such that $\eta \circ \nu\circ \eta =\eta $. By Lemma \ref{lem_B}, we have $F\nu\circ F\eta = \id _{F}$ and $\nu G\circ \eta G= \id _{G}$, hence it follows
\begin{eqnarray*}
\nu G\circ G\left( \epsilon ^{-1}\right) &=&\nu G\circ \left( G\epsilon \right) ^{-1}=\nu G\circ \eta G=\id_{G}, \\
F\nu\circ \left(\epsilon ^{-1}\right)F&=&F\nu\circ\left( \epsilon F\right) ^{-1}
=F\nu \circ F\eta =\id_{F}.
\end{eqnarray*}%
This means that $\left( G,F\right) $ is an adjunction with unit $\epsilon^{-1} :%
\id_{\dd}\rightarrow FG$ and counit $\nu :GF\rightarrow \id_{\cc}$, so $F$ is Frobenius. The equality $\epsilon^{-1}F=F\eta$ implies the coherent condition $\eta\circ \nu=\id$. In fact $\eta\circ\nu=\nu GF\circ GF\eta=\nu GF\circ G\epsilon^{-1}F=(\nu G\circ G\epsilon^{-1})F=\id_GF=\id_{GF}.$

$\left( 3\right) \Rightarrow \left( 7\right) $. Assume that $F$ is
Frobenius. Since the right adjoint $G$ is also a left adjoint of $F$, there are a unit $\eta ^{\prime }:\id_{\dd%
}\rightarrow FG$ and a counit $\epsilon ^{\prime }:GF\rightarrow \id_{%
\cc}$ of the adjunction $(G, F)$. Set $\tau :=\epsilon F\circ \eta^{\prime }F :F\to F$
and note that $F\eta \circ \tau =F\eta\circ\epsilon F\circ \eta^{\prime } F =\left(\epsilon F\right)^{-1}\circ \epsilon F\circ \eta^{\prime } F =\eta ^{\prime }F$. Consider $\nu :=\epsilon^{\prime }\circ G\tau$. It holds
\begin{equation*}
\eta \circ \nu =\eta\circ\epsilon ^{\prime }\circ G\tau\overset{\text{nat.}\epsilon^{\prime }}{=}\epsilon^{\prime } GF \circ GF \eta \circ G\tau = \epsilon^{\prime }GF\circ G\left( F\eta \circ \tau\right)=\epsilon ^{\prime }GF\circ G\eta ^{\prime }F=%
\id_{GF}.
\end{equation*}%
Thus, by \cite[Theorem 2.6 (1)]{ACMM06}, $F$ is naturally full.
\end{invisible}
\end{proof}

\begin{rmk}\label{rmk:sepcoref}
It is known that a conservative (co)reflection is always an equivalence (see e.g. \cite[Remark 1.4]{Ber07}).
\begin{invisible}
  Let $(F, G,\eta,\varepsilon)$ be an adjunction. If $G$ is a coreflection, then $F$ is fully faithful and hence $\eta$ is invertible. From the triangular identity $G\varepsilon\circ \eta G=\id_G$ we get that $G\varepsilon$ is invertible. Thus, if $G$ is also conservative, we obtain that $\varepsilon$ is invertible whence $G$ is an equivalence.
\end{invisible}
Since separable functors are conservative (cf. Remark \ref{rmk:Maschke}), one recovers the fact that a separable (co)reflection is, actually, an equivalence, see e.g. in \cite[Proposition 2.4]{Sar21}. Thus Theorem \ref{thm:frobenius} can be seen as a semi-analogue of this result.
\end{rmk}

The following result will be useful in Subsection \ref{sub:rightHopf}.

\begin{prop}\label{prop:sigma}
Let $F\dashv G\dashv H:\cc\to\dd$ be an adjoint triple with $G$ fully faithful. Denote by $\eta^l$, $\epsilon^l$ and $\eta^r$, $\epsilon^r$ the unit and the counit of the adjunction $F\dashv G$ and of the adjunction $G\dashv H$, respectively. Consider the natural transformation $\sigma :H\to F$ defined by $\sigma:=F\epsilon^r\circ (\epsilon^lH)^{-1}:H\to F$. Then, $H$ is semiseparable if and only if $\sigma$ is split-mono if and only if $\sigma$ is invertible.
\end{prop}

\begin{proof}Since $G$ is fully faithful, then $H$ is a coreflection so that, by Theorem \ref{thm:frobenius}, it is semiseparable if and only if it is naturally full if and only if it is Frobenius, i.e. $F\cong H$. By \cite[Proposition 2.2]{Sar21}, the condition $F\cong H$ is equivalent to the invertibility of $\sigma$. We now prove that $G$ is naturally full if and only if $\sigma$ is split-mono. We have a bijective correspondence $\mathrm{Nat}(F,H)\cong\mathrm{Nat}(\id_\cc,GH)$. Explicitly, for any natural transformation $\tau:F\to H$ there is a unique natural transformation $\gamma: \id_\cc\to GH$ given by $\gamma:=G\tau\circ \eta^l$. Then $\gamma\circ\epsilon^r=G\tau\circ \eta^l\circ\epsilon^r=G\tau\circ GF\epsilon^r\circ \eta^lGH=G\tau\circ G(\sigma\circ\epsilon^lH)\circ \eta^lGH=G(\tau\circ \sigma)\circ G\epsilon^lH\circ \eta^lGH=G(\tau\circ \sigma)$ so that $\gamma\circ\epsilon^r=G(\tau\circ \sigma)$. Thus, $\gamma\circ\epsilon^r=\id_{GH}$ if and only if $G(\tau\circ \sigma)=\id_{GH}$ if and only if $\tau\circ \sigma=\id_{H}$, as $G$ is faithful. By Rafael-type Theorem for naturally full functors, the condition $\gamma\circ\epsilon^r=\id_{GH}$ means that $H$ is naturally full.
\end{proof}

Let us see how the quotient functor $H:%
\cc\rightarrow \cc_{e}$ results to be a bireflection in meaningful cases.

\begin{prop}\label{prop:Hcorefl}
Let $\cc$ be a category and let $e:\id_{\cc%
}\rightarrow \id_{\cc}$ be an idempotent natural
transformation. Then, the quotient functor $H:\cc\to\cc_e$ is a bireflection if and only if  $e$ splits (e.g. when $\cc$ is idempotent complete).
\end{prop}

\begin{proof}
 Assume that $H:\cc\to\cc_e$ is a bireflection. Then, $H$ has left and right adjoint functors equal, say $L:\cc_e\to\cc$, which is fully faithful, and such that the coherence condition $\eta\circ\epsilon'=\id_{HL}$ is satisfied, where $\eta:\id\to HL$ is the unit of adjunction $L\dashv H$, and $\epsilon': HL\to\id$ is the counit of $H\dashv L$. Denote by $\eta':\id\to LH$ and $\epsilon: LH\to\id$ the unit and counit of the adjunctions $H\dashv L$ and $L\dashv H$, respectively. Since $L$ is fully faithful, $\eta$ is an isomorphism and hence, from the coherence condition and the triangular identity $L\epsilon'\circ \eta'L=\id_L$, we get that $L\eta=(L\epsilon')^{-1}=\eta'L$. Therefore, by naturality of $\eta'$, we have $\eta'\circ\epsilon =LH\epsilon\circ\eta'LH=LH\epsilon\circ L\eta H=L\id_H=\id_{LH}$. Similarly,  from the latter condition and the triangular identity  $H\epsilon \circ \eta H=\id_H$, it follows that $H\eta'=(H\epsilon)^{-1}=\eta H$. Then, $H(\epsilon\circ\eta')=H\epsilon\circ H\eta'=\id_H=H\id$. Thus, for all $X\in\cc$, we have $e_X=e_X\circ\epsilon_X\circ\eta'_X$. Now, recall $He=\id_H$ as $e$ is the idempotent natural transformation associated to $H$. Then, $e_X\circ\epsilon_X=\epsilon_X\circ LHe_X=\epsilon_X\circ LH\id_X=\epsilon_X$ so that the equality $e_X=e_X\circ\epsilon_X\circ\eta'_X$ simplifies as $e_X=\epsilon_X\circ\eta'_X$ and hence $e$ splits.
\par
The converse essentially follows from the dual of \cite[Proof of Theorem 13]{FOPTST99}. We give a slightly different proof here. Assume that $e$ splits. Since we know that $H$
is naturally full (see Subsection \ref{sub:coidentifier}), it is in particular semiseparable. Thus, in order to
conclude, by Theorem \ref{thm:frobenius}, it is enough to check that $H$ is a coreflection. Choose a splitting $\id_\cc \overset{\pi
}{\twoheadrightarrow }P\overset{\epsilon}{\hookrightarrow }\id_\cc$ of
the idempotent $e$ such that $\pi \circ \epsilon =\id_P$. Note that $Pe=Pe\circ \pi\circ \epsilon \overset{\text{nat.}\pi}{=}\pi \circ e\circ \epsilon=\pi \circ\epsilon\circ \pi\circ \epsilon=%
\id_{P}$ and hence  $%
Pe=\id_{P}$. Thus, by Lemma \ref{lem:coidentifier}, there is a unique functor $P_{e}:\cc%
_{e}\rightarrow \cc$ such that $P_{e}\circ H=P$. It is now straightforward to check that $P_{e}\dashv H$ with counit $\epsilon $ and invertible unit $\eta:\id_{\cc%
_{e}}\rightarrow HP_{e}$ defined by the equality $\eta H=H\pi $ i.e. by setting $%
\eta _{X}:=\left( \overline{\pi _{X}}\right) _{X\in \cc}.$ Thus $H$ is a coreflection.
\begin{invisible}Let us check that $%
P_{e}\dashv H.$  Since $P=P_{e}\circ H$ we can take $\epsilon $ as a
candidate counit. By Lemma \ref{lem:coidentifier}, since $\id_{\cc}\circ H=H$ and $(HP_{e})\circ H=HP$ we have $H_e=\id_{\cc}$, $(HP)_e=HP_{e}$ and hence we can define the unit $\eta:=(H\pi)_e :\id_{\cc%
_{e}}\rightarrow HP_{e}$ by the equality $\eta H=H\pi $ i.e. by setting $%
\eta _{X}:=\left( \overline{\pi _{X}}\right) _{X\in \cc}.$ Then $%
H\epsilon \circ \eta H=H\epsilon \circ H\pi =He=\id_{H}$. Moreover $\left(
\epsilon P_{e}\circ P_{e}\eta \right) H\circ \pi =\epsilon P_{e}H\circ
P_{e}\eta H\circ \pi =\epsilon P_{e}H\circ P_{e}H\pi \circ \pi =\epsilon
P\circ P\pi \circ \pi \overset{\text{nat.}\epsilon }{=}\pi \circ \epsilon\circ
\pi =\pi $ so that $\left(
\epsilon P_{e}\circ P_{e}\eta \right) H=\id_{P_{e}\circ H}$ and
hence $\epsilon P_{e}\circ P_{e}\eta =\id_{P_{e}}$. This proves that
$P_{e}\dashv H$. Since $\pi \circ \epsilon =\id_P$, we also get $\eta
H\circ H\epsilon =H\pi \circ H\epsilon =\id_{H}$. This equality,
together with $H\epsilon \circ \eta H=\id_{H}$, imply that $\eta H$
is invertible. Since $H$ is the identity on objects, we deduce that $\eta $
is invertible and hence $P_{e}$ is fully faithful by the dual of \cite[Proposition
3.4.1]{Bor94}.

[La dimostrazione di questa parte sembra nota implicitamente ma non l'ho trovata esplicitamente, quindi la scrivo qui con l'idea di occultarla.] Assume that $\cc$ is idempotent complete and let us prove that an idempotent natural transformation $e:\id_{\cc%
}\rightarrow \id_{\cc}$ necessarily splits.  Explicitly, choose a splitting $X\overset{\pi
_{X}}{\twoheadrightarrow }PX\overset{\epsilon _{X}}{\hookrightarrow }X$ of
the idempotent $e_{X}$ for any object $X$ in $\cc$. Then, by \cite[Section 1]{Bar72},  we have the coequalizer%
\begin{equation*}
X\overset{e_{X}}{\underset{\id_{X}}{\rightrightarrows }}X\overset{%
\pi _{X}}{\twoheadrightarrow }PX.
\end{equation*}%
Since, for every morphism $f:X\rightarrow Y$ in $\cc$, we have $\pi
_{Y}\circ f\circ e_{X}=\pi _{Y}\circ e_{Y}\circ f=\pi _{Y}\circ f=\pi
_{Y}\circ f\circ \id_{X}$, there is a unique morphism $%
Pf:PX\rightarrow PY$ such that $Pf\circ \pi _{X}=\pi _{Y}\circ f$. This
defines an endofunctor $P:\cc\rightarrow \cc$.
The equality $Pf\circ \pi _{X}=\pi _{Y}\circ f$ means that $%
\pi :=\left( \pi _{X}\right) _{X\in \cc}:\id_{\cc%
}\rightarrow P$ is a natural transformation. Moreover $\epsilon _{Y}\circ
Pf\circ \pi _{X}=\epsilon _{Y}\circ \pi _{Y}\circ f=e_{Y}\circ f=f\circ
e_{X}=f\circ \epsilon _{Y}\circ \pi _{Y}$ forces $\epsilon _{Y}\circ
Pf=f\circ \epsilon _{Y}$ which means that $\epsilon :=\left( \epsilon
_{X}\right) _{X\in \cc}:P\rightarrow \id_{\cc}$ is a
natural transformation. Thus we obtained a splitting $\id_\cc \overset{\pi
}{\twoheadrightarrow }P\overset{\epsilon}{\hookrightarrow }\id_\cc$ of
the idempotent $e$.
\end{invisible}
\end{proof}

As a consequence of Theorem \ref{thm:coidentifier} and Proposition \ref{prop:Hcorefl}, we have the following corollary.

\begin{cor}\label{cor:fact-birefl}A functor $F:\cc\to \dd$ factors as a bireflection followed by a separable functor if and only if it is semiseparable and the associated natural transformation $e:\id_\cc\to\id_\cc$ splits. Moreover, any such a factorization is the same given by the coidentifier within Theorem \ref{thm:coidentifier}, up to a category equivalence.
\end{cor}
\proof
Assume that $F=S\circ N$ where $N:\cc\to\e$ is a bireflection and $S:\e\to\dd$ is a separable functor. Since $N$ is in particular naturally full, we get that $F=S\circ N$ is semiseparable and hence by Theorem \ref{thm:coidentifier}, there is a unique functor $N_e:\cc_e\to \e$ (necessarily fully faithful) such that $N_e\circ H=N$ and $S\circ N_e=F_e$. If we denote by $L:\e\to\cc$ the left adjoint of $N$, then it is fully faithful and hence, since the unit $\eta:\id\to NL$ is an isomorphism, we have that $\id\cong N\circ L=N_e\circ H\circ L$. Therefore $N_e$ is essentially surjective on objects. Since it is also fully faithful, we get $N_e$ is an equivalence of categories and, from $\id\cong N_e\circ H\circ L$, it has quasi-inverse $H\circ L$. Thus, $(H\circ L)\circ N_e\cong\id$ and hence $(H\circ L)\circ N=(H\circ L\circ N_e)\circ H\cong \id\circ H=H$, from which it follows that $H$ is a bireflection as $N$ is. Thus, by Proposition \ref{prop:Hcorefl}, the idempotent $e$ splits. Moreover the factorization $F=S\circ N$, up to the category equivalence $N_e$, is the same given by the coidentifier within Theorem \ref{thm:coidentifier}.

Conversely, if the natural transformation $e:\id_\cc\to\id_\cc$ attached to the semiseparable functor $F$ splits, then by Proposition \ref{prop:Hcorefl} the quotient functor $H:\cc\to\cc_e$ results to be a bireflection. Thus, since by Theorem \ref{thm:coidentifier} the semiseparable functor $F$ factors as $H:\cc\to \cc_e$ followed by a separable functor $F_e:\cc_e\to \dd$, we achieve the desired factorization of $F$ into a bireflection followed by a separable functor.
\endproof

\begin{rmk}\label{rmk:factsepbir}
By Theorem \ref{thm:coidentifier}, any semiseparable functor $F:\cc\to \dd$ factors as $H:\cc\to \cc_e$ followed by a separable functor $F_e:\cc_e\to \dd$, where $e$ is the associated idempotent natural transformation. Assume that $e$ splits. Then, by Proposition \ref{prop:Hcorefl},$H:\cc\to \cc_e$ is a bireflection. In particular $H$ is a coreflection and $F_e$ is conservative. This is what is called an \emph{image-factorization} of $F$ in \cite[Definition 1.1]{Ber07}. Image-factorizations are unique up to an equivalence of categories, see \cite[Lemma 1.2]{Ber07}. As a consequence, if we can write $F=S\circ N$ where $S:\e\to \dd$ is conservative (e.g. separable) and $N:\cc\to \e$ is a coreflection (e.g. a bireflection) then there is an equivalence $\e \cong \cc_e$.
\end{rmk}

\section{Applications and examples}\label{sect:applications}
In this section we test the notion of semiseparability on relevant functors attached to ring and coalgebra morphisms, corings, bimodules and Hopf modules. 

We start in Subsection \ref{es:extens} by considering the restriction of scalars functor $\varphi_*: S\text{-Mod}\rightarrow R\text{-Mod}$, the extension of scalars functor $\varphi^*= S\otimes_{R}(-):R\text{-Mod}\rightarrow S\text{-Mod}$ and the coinduction functor $\varphi^!= {}_R\Hom(S,-):R\text{-Mod}\rightarrow S\text{-Mod}$ associated to a ring morphism $\varphi:R\to S$. On the one hand, since $\varphi_*$ is faithful, its semiseparability falls back to its separability. On the other hand, the functors above form an adjoint triple $\varphi^*\dashv\varphi_*\dashv\varphi^!$ so that the semiseparability of $\varphi^!$ is equivalent to the one of $\varphi^*$. The latter is characterized in Proposition \ref{prop:inducfunc} in terms of the regularity of $\varphi$ as a morphism of $R$-bimodules and in Proposition \ref{prop:modextens} in terms of the existence of a suitable central idempotent element $z\in R$.
In a similar fashion in Subsection \ref{es:coinduct-coalgebras} we investigate the semiseparability of two adjoint functors attached to a coalgebra map $\psi:C\to D$. The main result here is Proposition \ref{prop:coinduc-coalg}.

In Subsection \ref{es:coring}, we turn our attention to the induction functor $(-)\otimes_R\cc : \text{Mod-}R\to\m^{\cc}$, attached to an $R$-coring $\cc$. Here we highlight Theorem \ref{thm:inducoring} where this functor is proved to be semiseparable if and only if the coring counit $\varepsilon_\cc:\cc\to R$ is a regular morphism of $R$-bimodule. 
In Subsection \ref{es:bimod}, given an $(R,S)$-bimodule $M$, we consider the coinduction functor $\sigma_*:=\Hom_S(M,-): \mathrm{Mod}\text{-}S\to \mathrm{Mod}\text{-}R$ together with its left adjoint $\sigma^*:=(-)\otimes_R M: \text{Mod-}R\to \text{Mod-}S$. In Theorem \ref{thm:bimod} we show that the semiseparability of  $\sigma_*$ can be completely rewritten both in terms of the regularity of the evaluation map plus a mild condition that is redundant when $M$ is projective as a right $S$-module, and in terms of a property of $M$ that will lead us to introduce the new notion of $M$-semiseparability over $R$ for the ring $S$, a right semiseparable version of the one given in \cite{Su71}.
In Corollary \ref{cor:sep-bimod} we prove that $S$ is $M$-separable over $R$ if and only if $S$ is $M$-semiseparable over $R$ and $M$ is a generator in $\mathrm{Mod}\text{-}S$. A different characterization of $M$-semiseparability of $S$ over $R$, that will allow us to exhibit an example where $S$ is $M$-semiseparable but not $M$-separable over $R$ (see Example \ref{es:ssVSs}), is obtained in Proposition \ref{prop:ssVSs}. Then, if we add the assumption that $M$ is finitely generated and projective as a right $S$-module, further characterizations of the semiseparability of $\sigma_*$ and $\sigma^*$ are provided in Proposition \ref{prop:right-equiv-bimod} and Proposition \ref{prop:left-equiv-bimod}, respectively.

It is worth noticing that the above functors $\varphi^*$, $(-)\otimes_R\cc$, and $\sigma_*$ have sources which are idempotent complete categories so that, by Corollary \ref{cor:fact-birefl}, they always admit a factorization as a bireflection followed by a separable functor, when they are semiseparable.
In Proposition \ref{prop:phi*}, Corollary \ref{cor:inducoring}, and Proposition \ref{prop:condcMfact}, we  explicitly provide such factorizations.

Theorem \ref{thm:rightHopf} concerns the semiseparability of the coinvariant functor $\left( -\right)^{\co B}:\mathfrak{M}_{B}^{B}\rightarrow \mathfrak{M}$, from the category of right Hopf modules over a $\Bbbk$-bialgebra $B$ to the category of $\Bbbk$-vector spaces over a field $\Bbbk$, proving that it is semiseparable if and only if $B$ is a right Hopf algebra with anti-multiplicative and anti-comultiplicative right antipode.

Finally, Subsection \ref{appendix:corefl} provides particular examples of (co)reflections that highlight the connections between the types of functors we have studied in this paper.

\subsection{Extension and restriction of scalars}\label{es:extens}
A morphism of rings $\varphi :R\to S$ induces
\begin{itemize}
  \item the restriction of scalars functor $\varphi_* : S\text{-Mod}\rightarrow R\text{-Mod}$;
  \item the extension of scalars (or induction) functor $\varphi^*:= S\otimes_{R}(-):R\text{-Mod}\rightarrow S\text{-Mod}$;
  \item the coinduction functor $\varphi^!:= {}_R\Hom(S,-):R\text{-Mod}\rightarrow S\text{-Mod}$.
\end{itemize}
All together these functors form an adjoint triple $\varphi^*\dashv\varphi_*\dashv\varphi^!.$

The unit $\eta$ and the counit $\epsilon$ of the adjunction $(\varphi^*,\varphi_*)$, are respectively defined by
$$\eta_M = \varphi\otimes_R M : M\to S\otimes_R M,\, m\mapsto 1_{S}\otimes_R m,\qquad\text{and}\qquad \epsilon_N: S\otimes_R N\to N,\, s\otimes_R n\mapsto sn,$$
while the unit $\eta^!$ and the counit $\epsilon^!$ of the adjunction $(\varphi_*,\varphi^!)$, are defined by
$$\eta^!_N:N\to {}_R\Hom(S,N),\, n\mapsto [s\mapsto sn],\qquad\text{and}\qquad\epsilon^!_M:{}_R\Hom(S,M)\to M,\, f\mapsto f(1_S),$$
for every $M\in R\text{-Mod}$ and  $N\in S\text{-Mod}$.
In the literature we can find results either on the separability or on the natural fullness of these functors. For instance we know that
\begin{itemize}
  \item $\varphi_*$ is separable if and only if $S/R$ is separable, see \cite[Proposition 1.3]{NVV89};
  \item $\varphi_*$ is naturally full if and only if it is full, see \cite[Proposition 3.1 (1)]{ACMM06}, if and only if it is fully faithful (in fact it is always faithful being a forgetful functor) if and only if $\varphi$ is an epimorphism in the category of rings, cf. \cite[Proposition XI.1.2]{Ste75};
  \item    $\varphi^*$ is separable if and only if $\varphi$ is split-mono as an $R$-bimodule map, i.e. if there is $E\in {}_{R}\Hom_{R}(S,R)$ such that $E\circ \varphi=\id$, see \cite[Proposition 1.3]{NVV89};
  \item $\varphi^*$ is naturally full if and only if  $\varphi$ is split-epi as an $R$-bimodule map, i.e. if there is $E\in {}_{R}\Hom_{R}(S,R)$ such that $ \varphi\circ E=\id$, see \cite[Proposition 3.1 (2)]{ACMM06};
  \item $\varphi^!$ is separable if and only if so is $\varphi^*$ \cite[Corollary 3.10]{CGN98}.
\end{itemize}
We now investigate the semiseparability of these three functors. Indeed, from the general characterization given in Proposition \ref{prop:adj-triples}, we know that $\varphi^!$ is semiseparable (resp. separable, naturally full) if and only if so is $\varphi^*$. For this reason we are only dealing with the functors $\varphi_*$ and $\varphi^*$.

Concerning $\varphi_*$, since it is always faithful, we have that $\varphi_*$ is semiseparable if and only if $\varphi_*$ is separable, that is, $S/R$ is separable. Thus, although we are tempted to name $S/R$ ``semiseparable'' whenever $\varphi_*$ is semiseparable, by the foregoing, this would bring us back to $S/R$ separable.

In the next results we investigate when the functor $\varphi^*$ is semiseparable.

\begin{prop}\label{prop:inducfunc}
Let $\varphi :R\to S$ be a morphism of rings. Then, the extension of scalars functor $\varphi^*= S\otimes_{R}(-):R\text{-}\mathrm{Mod}\rightarrow S\text{-}\mathrm{Mod}$ is semiseparable if and only if $\varphi$ is a regular morphism of $R$-bimodules, i.e. there is $E\in {}_{R}\Hom_{R}(S,R)$ such that $\varphi\circ E\circ\varphi =\varphi$, i.e., such that $\varphi E(1_S)=1_S$.
\end{prop}
\proof It is known that there is a bijective correspondence
$
\mathrm{Nat}(\varphi_*\varphi^*, \id_{R\text{-}\mathrm{Mod}})\cong {}_{R}\Hom_{R}(S,R)
$, see \cite[Theorem 27]{CMZ02}.
Now, by Theorem \ref{thm:rafael}, $\varphi^*$ is semiseparable if and only if there exists a natural transformation $\nu \in \mathrm{Nat}(\varphi_*\varphi^*,\id_{R\text{-}\mathrm{Mod}})$ such that $\eta\circ\nu\circ\eta = \eta$. So, given $\nu$ for $\varphi^*$, we consider the corresponding $E\in {}_{R}\Hom_{R}(S,R)$, $E(s):=\nu_R(s\otimes_R 1_R)$, for every $s\in S$. Then, for every $r\in R$,
we get $
(\varphi\circ E\circ\varphi )(r) = \varphi (E(\varphi(r)))=\varphi (\nu_R(\varphi(r)\otimes_R 1_R))=\varphi(\nu_R(\eta_R(r))) = r_S\eta_R(\nu_R(\eta_R(r)))
=r_S\eta_R(r)=\varphi(r)
$
where $r_S:S\otimes_R R\to R, s\otimes_R r\mapsto s\varphi(r),$ is the canonical isomorphism.
Conversely, given $E\in {}_{R}\Hom_{R}(S,R)$ such that $\varphi\circ E\circ\varphi =\varphi$, define $\nu_M:S\otimes_R M\to M$, $\nu_M(s\otimes_R m)=E(s)m$, for every $M\in R\text{-}\mathrm{Mod}$, $m\in M$ and $s\in S$. Then,
\begin{equation*}
\begin{split}
(\eta_M\circ\nu_M\circ\eta_M )(m)&= \eta_M(\nu_M(1_S\otimes_R m)) = \eta_M (\nu_M (\varphi(1_R)\otimes_R m)
=\eta_M (E(\varphi (1_R))m)\\ &= 1_S\otimes_R E(\varphi (1_R))m = 1_S E(\varphi (1_R))\otimes_R m = \varphi (E(\varphi (1_R)))\otimes_R m \\
&\overset{\varphi E\varphi = \varphi}{=}\varphi (1_R)\otimes_R m = 1_S\otimes_R m
=\eta_M (m).
\end{split}
\end{equation*}
Now, note that, since $E$ is a morphism of $R$-bimodules, we get $(\varphi\circ E\circ\varphi)(r)=\varphi( E(\varphi(r)))=\varphi( E(r1_S))=\varphi( rE(1_S))=\varphi( r)\varphi E(1_S) $. As a consequence, the condition $(\varphi\circ E\circ\varphi)( r)=\varphi( r)$ is equivalent to $\varphi E(1_S)=1_S$.
\endproof

We now give an example of a semiseparable functor which is neither separable nor naturally full.

\begin{es}\label{es:induction}
Let $\varphi:R\to S$ and $\psi:Q\to R$ be morphisms of rings whose induction functors  $\varphi^*$ and $\psi^*$ are separable and naturally full respectively. This means there is $E\in {}_{R}\Hom_{R}(S,R)$ such that $E\circ \varphi=\id$ (in particular $\varphi$ is injective) and there is $D\in {}_{Q}\Hom_{Q}(R,Q)$ such that $\psi \circ D=\id$ (in particular $\psi$ is surjective). By Corollary \ref{cor:fattoriz}, the composition $\varphi^*\circ \psi^*\cong (\varphi\circ \psi)^*$ is semiseparable. The map corresponding to $\varphi\circ \psi$ via Proposition \ref{prop:inducfunc} is  $D\circ E\in {}_{Q}\Hom_{Q}(S,Q)$. Note that, if $\varphi\circ \psi$ is neither injective nor surjective, we can conclude that $(\varphi\circ \psi)^*$ is neither separable nor naturally full. For instance, let $\varphi:\QQ\to \QQ[X]$ be the canonical injection of the field of rational numbers into the polynomial ring over it and let $\psi:\QQ\times\ZZ\to \QQ$ be given by $\psi((q,z))=q$. Then we can define $D$ by setting $D(q)=(q,0)$ and $E$ to be the evaluation at $0$ of the given polynomial. Then $(\varphi\circ \psi)^*$ is semiseparable but it is neither separable nor naturally full.
 \end{es}

 In a similar way as in Example \ref{es:induction}, the following example shows that semiseparable functors are not closed under composition.
\begin{es}\label{es:sscompos}
Let $\varphi:R\to S$ and $\psi:S\to Q$ be morphisms of rings whose induction functors  $\varphi^*$ and $\psi^*$ are separable and naturally full respectively (in particular both semiseparable by Proposition \ref{prop:sep}). This means there is $E\in {}_{R}\Hom_{R}(S,R)$ such that $E\circ \varphi=\id$ and there is $D\in {}_{S}\Hom_{S}(Q,S)$ such that $\psi \circ D=\id$. The results we have proved so far do not allow us to conclude that the composition $\psi^*\circ \varphi^*\cong (\psi\circ \varphi)^*$ is semiseparable. Indeed we can provide a specific example where this is not true. Let $\varphi:\ZZ\to \QQ\times\ZZ,z\mapsto(z,z)$, and let $\psi:\QQ\times\ZZ\to \QQ$ be given by $\psi((q,z))=q$. Then we can define $D$ by setting $D(q)=(q,0)$ and $E$ by setting $E((q,z))=z$. In this way we get that $\varphi^*$ and $\psi^*$ are separable and naturally full respectively. Let us show that $(\psi\circ \varphi)^*$ is not semiseparable. Otherwise, by Proposition \ref{prop:inducfunc} there exists $E'\in {}_{\ZZ}\Hom_{\ZZ}(\QQ,\ZZ)$ such that $\psi \varphi(E'(1_\QQ))=1_\QQ.$ Since ${}_{\ZZ}\Hom_{\ZZ}(\QQ,\ZZ)=\{0\}$, this means $0_\ZZ=1_\ZZ$, a contradiction.
\begin{invisible}
  Given $f\in {}_{\ZZ}\Hom_{\ZZ}(\QQ,\ZZ)$, for every $q\in\QQ$ and $n\in\NN$, we have $f(q)=f(n\frac{q}{n})=nf(\frac{q}{n})$ so that $n$ divides $f(q)$ in $\ZZ$ for every $n\in\NN$. This is possible only if $f(q)=0$.
\end{invisible}
\end{es}

 Let us see that all morphisms of rings $\varphi :R\to S$ whose induction functor $\varphi^*:= S\otimes_{R}(-):R\text{-}\mathrm{Mod}\rightarrow S\text{-}\mathrm{Mod}$ is semiseparable are of the form given in Example \ref{es:induction}. More precisely, we will get that $\varphi^*$ factors as a bireflection followed by a separable functor. First we need the next remark.

 \begin{rmk}\label{rmk:epirng}
Let $\varphi:R\to S$ be an epimorphism in the category of rings. By \cite[Proposition 1.2]{Ste75} the faithful functor $\varphi_*$ is also full, and hence its left adjoint $\varphi^*=S\otimes_R(-)$ is a reflection, whereas its right adjoint $\varphi^!= {}_R\Hom(S,-)$ is a coreflection. Thus, Theorem \ref{thm:frobenius} applies in this case to get that $\varphi^*$ is naturally full if and only if it is semiseparable if and only if it is Frobenius, that is, in the same way $\varphi^!$ is naturally full if and only if it is semiseparable if and only if it is Frobenius. In particular, in this case $\varphi^*$ and $\varphi^!$ are isomorphic bireflections.
\end{rmk}

\begin{prop}\label{prop:phi*}
  Let $\varphi :R\to S$ be a morphism of rings. Write $\varphi=\iota\circ \overline{\varphi}$ where $\iota:\varphi(R)\to S$ is the canonical inclusion and $\overline{\varphi}:R\to \varphi(R)$ is the corestriction of $\varphi$ to its image $\varphi(R)$.

  Then, the induction functor $\varphi^*:= S\otimes_{R}(-):R\text{-}\mathrm{Mod}\rightarrow S\text{-}\mathrm{Mod}$ is semiseparable if and only if $\iota^*$ is separable and $\overline{\varphi}^*$ is a bireflection.
\end{prop}
\proof
If $\varphi^*$ is semiseparable, by Proposition \ref{prop:inducfunc}, there is $E\in {}_{R}\Hom_{R}(S,R)$ such that $\varphi\circ E\circ \varphi=\varphi$, i.e. $\iota\circ \overline{\varphi}\circ E\circ\iota\circ\overline{\varphi}=\iota\circ \overline{\varphi}$. Since $\iota$ is injective and $\overline{\varphi}$ is surjective, we get $ \overline{\varphi}\circ E\circ\iota=\id_{\varphi(R)}$ which implies that $\iota^*$ is separable. On the other hand, $\overline{\varphi}^*$ is a bireflection in view of Remark \ref{rmk:epirng} and surjectivity of $\overline{\varphi}$.
Conversely, if  $\iota^*$ is separable and $\overline{\varphi}^*$ is a bireflection, whence naturally full, then the composition $\iota^*\circ \overline{\varphi}^*\cong (\iota\circ \overline{\varphi})^*=\varphi^*$ is semiseparable by Corollary \ref{cor:fattoriz}.
\endproof

In the proof of Proposition \ref{prop:phi*} we obtained a factorization $\iota^*\circ \overline{\varphi}^*\cong \varphi^*$ in case $\varphi^*$ is semiseparable. In view of Corollary \ref{cor:fact-birefl}, this factorization is the same given by the coidentifier within Theorem \ref{thm:coidentifier}, up to a category equivalence.\medskip

Next proposition provides a further characterization of the semiseparability of $\varphi^*$. We point out that its current proof, more direct than the original one, was suggested by P. Saracco.

First recall that, given a central idempotent element $z$ in a ring $R$, then $zRz=Rz$ is a ring with addition
and multiplication those of $R$ restricted to $zR$ and with identities $%
0_{Rz}=0_{R}z=0_R$ and $1_{Rz}=1_{R}z=z,$ and there is a surjective ring
homomorphism $R\rightarrow Rz,r\mapsto rz$, see \cite[1.16]{AF92}.

\begin{prop}\label{prop:modextens}
Let $\varphi:R\to S$ be a ring homomorphism. Then, the induction functor $\varphi^*=S\otimes_R(-):R\text{-}\mathrm{Mod}\to S\text{-}\mathrm{Mod}$  is semiseparable if and only if there is a central idempotent $z\in R$ (necessarily unique) such that $\varphi(z)=1_S$ and the ring map $\tau:=\varphi_{\mid Rz} :Rz\rightarrow S$ is split-mono as an $Rz$-bimodule map.
\end{prop}

\begin{proof}
 By Proposition \ref{prop:inducfunc}, the functor $\varphi^*$ is semiseparable if and only if there exists $E\in{}_R\Hom_{R}(S,R)$ such that $\varphi E(1_S)=1_S.$ Assume that there is $E$ as above and set $z:=E(1_S)\in R$. Clearly $\varphi(z)=\varphi E(1_S)=1_S$, i.e.  $\varphi(z)=1_S$. For any $r\in R$ we have $rz=rE(1_S)=E(\varphi (r)1_S)=E(\varphi(r))$ and similarly $zr=E(\varphi(r))$ so that $rz=zr$, i.e. $z$ is central. Taking $z=r$ in the computation above, we get $zz=E(\varphi(z))=E(1_S)=z$, thus $z$ is an idempotent. Concerning the ring map $\tau:=\varphi_{\mid Rz} :Rz\rightarrow S$, consider the canonical projection $\psi : R\rightarrow Rz,r\mapsto rz$. Since $z$ is central, we get that $\psi$ is $R$-bilinear so that the map $\pi:=\psi\circ E:S\to Rz,s\mapsto E(s)z$ is $R$-bilinear as a composition of bilinear maps. In particular $\pi$ is $Rz$-bilinear. Then, for any $r\in R$ we have $\pi\tau (rz)=\pi \varphi(rz)=\pi (\varphi(r)\varphi(z))=\pi (r1_S)=r\pi(1_S)=rE(1_S)z=rzz=rz$, hence $\pi\circ \tau=\id_{Rz}$ i.e. $\tau$ is a split-monomorphism of $Rz$-bimodules.

Conversely, assume there is a central idempotent $z\in S$ such that $\varphi(z)=1_S$ and the ring map $\tau:=\varphi_{\mid Rz} :Rz\rightarrow S$ is split-mono through a $Rz$-bimodule map $\pi:S\to Rz$. Set $E:S\to R,s\mapsto \pi(s)$. Then $rs=\varphi(r)s=\varphi(r)\varphi(z)s=\varphi(rz)s=(rz)s$ so that $E(rs)=E((rz)s)=\pi((rz)s)=rz\pi(s)=r\pi(s)=rE(s)$ where the second-last equality follows from the fact that $\pi(s)\in Rz$ and $z$ is a central idempotent. Similarly one gets $E(sr)=E(s)r$ so that $E$ is $R$-bilinear. Finally, we have $E(1_S)=\pi(1_S)=\pi(\varphi(z))=\pi(\tau(z))=z$ and hence $\varphi E(1_S)=\varphi(z)=1_S.$
Assume there is another idempotent $z'\in R$ as in the statement. Then $zz'=E(1_S)z'=E(1_S z')=E(\varphi(z'))=E(1_S)=z$. By exchanging the roles of $z$ and $z'$ we get $z'z=z'$ and hence $z=z'$.
 \end{proof}

\begin{rmk}\label{rmk:modextens}
In the proof of Proposition \ref{prop:modextens} we considered the maps $\tau:=\varphi_{\mid Rz} :Rz\rightarrow S$ and $\psi : R\rightarrow Rz,r\mapsto rz$. Since $\varphi(z)=1_S$, we get the equality $\varphi=\tau\circ\psi$ which provides the factorization $\varphi^*\cong\tau^*\circ\psi^*$. On the other hand, in Proposition \ref{prop:phi*} we obtained the equality $\varphi=\iota\circ\overline{\varphi}$, where $\iota:\varphi(R)\to S,s\mapsto s,$ and $\overline{\varphi}:R\to \varphi(R),r\mapsto \varphi(r)$, which yields the factorization $\varphi^*\cong\iota^*\circ\overline{\varphi}^*$. Define the morphism $\lambda:Rz\to\varphi(R),rz\mapsto\varphi(r)$. Clearly the following diagrams commute.
  $$\xymatrixcolsep{1.6cm}\xymatrixrowsep{.4cm}
\xymatrix{R\ar[dd]_\varphi\ar@{->>}[dr]_\psi\ar@/^1pc/@{->>}[rrd]^{\overline{\varphi}}\\&Rz\ar@{^(..>>}[r]^\lambda\ar@{^(->}[dl]_\tau&\varphi(R)\ar@/^1pc/@{^(->}[lld]^{\iota}\\S
}
\qquad
\xymatrixcolsep{1.6cm}\xymatrixrowsep{.6cm}
\xymatrix{R\text{-}\mathrm{Mod}\ar[rd]_{\psi^*}\ar@/^1pc/[rrd]^{\overline{\varphi} ^{\ast }}
\ar[dd]_{\varphi^*}&& \\
&Rz\text{-}\mathrm{Mod}\ar[ld]_-{\tau^*}\ar[r]^{\lambda^{\ast }}&\varphi(R)\text{-}\mathrm{Mod}\ar@/^1pc/[dll]^{\iota^*}
\\
S\text{-}\mathrm{Mod}
} $$ The first diagram entails that $\lambda$ is both injective and surjective whence bijective. As a consequence the given factorizations are the same up to the equivalence $\lambda^*$.
\end{rmk}

\subsection{Coinduction and corestriction of coscalars functors}\label{es:coinduct-coalgebras}
Let $\Bbbk$ be a field. We simply denote the tensor product over $\Bbbk$ by the unadorned $\otimes$. 
A $\Bbbk$-\emph{coalgebra} $C$ is a vector space $C$ over $\Bbbk$ equipped with two $\Bbbk$-linear maps $\Delta_C:C\to C\otimes C$ and $\varepsilon_C:C\to \Bbbk$ such that $\Delta_C$ is coassociative and counital, i.e. the equalities\begin{equation*}
(\Delta_{C}\otimes C )\circ\Delta_{C} = (C\otimes \Delta_{C})\circ\Delta_{C} \quad\text{and}\quad (\varepsilon_{C}\otimes C )\circ\Delta_{C}=(C\otimes_{\Bbbk}\varepsilon_{C})\circ\Delta_{C}=C
\end{equation*} hold true.
A right $C$-comodule $M$ is a $\Bbbk$-vector space together with a $\Bbbk$-linear map $\rho_M:M\to M\otimes C$, called the \emph{coaction}, that is coassociative and right counital i.e.
\begin{equation*}
(\rho_M\otimes C )\circ\rho_M = (M\otimes\Delta_{C})\circ\rho_M \quad\text{and}\quad (M\otimes\varepsilon_{C})\circ\rho_M =M.
\end{equation*}
A coalgebra $C$ can be seen as a right $C$-comodule with $\rho_C=\Delta_C$. Both for $\Delta$ and $\rho_M$ we adopt the usual Sweedler notations $\Delta(c)=\sum c_1\otimes c_2$ and $\rho_M(m)=\sum m_0\otimes  m_1$ for every $c\in C,m\in M$.
A \emph{morphism of right $C$-comodules} (or a $C$-\emph{colinear} morphism) is a $\Bbbk$-linear map $f:M\to  N$ between right $C$-comodules such that $\rho_N\circ f = (f\otimes C )\circ\rho_M$. The category of right $C$-comodules and their morphisms is denoted by $\m^{C}$. Analogously, the category of left $C$-comodules and their morphisms is denoted by ${}^C\m$. Recall from \cite{Tak77} that, given a right $C$-comodule $M$ and a left $C$-comodule $N$, the \emph{cotensor product} $M\square_C N$ is the kernel of the $\Bbbk$-linear map $$\rho_M\otimes N-M\otimes\rho_N: M\otimes N\to M\otimes C\otimes N,$$
where $\rho_M$ and $\rho_N$ are the right and the left $C$-comodule structures of $M$ and $N$, respectively.
\par
Now, let $\psi:C\to D$ be a \emph{morphism of coalgebras}, i.e. a $\Bbbk$-linear map $\psi:C\to D$ such that $\Delta_D\circ\psi=(\psi\otimes\psi)\circ\Delta_C$ and $\varepsilon_D\circ \psi=\varepsilon_C$.
Since any right $C$-comodule $M$ with coaction $\rho_M:M\to M\otimes C$ can be viewed as a right $D$-comodule with coaction $(M\otimes\psi)\circ \rho_M: M\to M\otimes D$ and $C$ can be considered as a $(D,C)$-bicomodule, $\psi$ induces
\begin{itemize}
	\item the corestriction of coscalars functor $\psi_*:\m^C\to\m^D$,
	\item the coinduction functor $\psi^*:=(-)\square_D C:\m^D\to\m^C$,
\end{itemize}
which form an adjunction $\psi_*\dashv \psi^*:\m^D\to\m^C$, with unit $\eta:\id_{\m^C}\to \psi^*\psi_*$ and counit $\epsilon:\psi_*\psi^* \to\id_{\m^D}$, given by $$\eta_M : M\to M\square_D C,\, m\mapsto\sum m_0\square_D m_1,\quad \text{and}\quad \epsilon_N: N\square_D C\to N,\, n\square_D c\mapsto n\varepsilon_C(c),$$
for any $M\in\m^C$ and $N\in \m^D$, see \cite[11.10]{BW03}. Note that, in the definition of $\psi^*$, for any right $D$-comodule $N$, the cotensor product $N\square_D C$ is regarded as a right $C$-comodule via $\rho_{N\square_D C}:N\square_D C\to (N\square_D C)\otimes C,n\square_D c\mapsto \sum (n\square_D c_1)\otimes c_2.$ 
Furthermore, the coaction $\rho_M$ of $M$ as a right $C$-comodule induces a morphism of right $C$-comodules $\bar{\rho}_M=\eta_M:M\to M\square_D C$ such that $\rho_{M}=i\circ\bar{\rho}_M$, where $i:M\square_D C\to M\otimes C$ is the canonical inclusion. In particular, if $M=C$ then $\bar{\rho}_{C}=\bar{\Delta}_C=\eta_C:C\to C\square_D C$.
\begin{invisible}
We prove that $\psi_*\dashv \psi^*:\m^D\to\m^C$ is an adjunction as above. First, for any $M\in\m^C$, let us see that $\eta_M$ is well defined i.e. that its image is contained in $M\square_DC$. For $m\in M$, we have $((M\otimes\psi)\circ \rho_M)\otimes C)(\sum m_0\otimes m_1)=\sum (M\otimes\psi)\rho_M(m_{0})\otimes m_1=\sum m_{0_0}\otimes\psi(m_{0_1})\otimes m_1$ and $(M\otimes (\psi\otimes C) \Delta_C)(\sum m_0\otimes m_1)=\sum m_0\otimes (\psi\otimes C)\Delta_C(m_1)=\sum m_0\otimes\psi(m_{1_1})\otimes_{\Bbbk}m_{1_2}$.
Moreover $\eta_M$ is a morphism of right $C$-comodules, as $(\rho_{M\square_D C}\circ\eta_M)(m)=\rho_{M\square_D C}(\sum m_0\square_D m_1)=\sum (m_{0}\square_D m_{1_1})\otimes m_{1_2}$ and $(\eta_M\otimes C)\rho_M(m)=(\eta_M\otimes C)(\sum m_0\otimes m_1)=\sum (m_{0_0}\square_D m_{0_1})\otimes m_1$, for all $m\in M$.
For any $N\in\m^D$, $\epsilon_N$ is a morphism of right $D$-comodules since for all $n\square_D c\in N\square_DC$ we have $(\rho_N\circ\epsilon_N)(n\square_D c)=\rho_N(n\varepsilon_{C}(c))=\rho_N(n)\varepsilon_C(c)=\sum n_0\otimes n_1\varepsilon_C(c)\overset{(*)}{=}\sum n\otimes\psi(c_1)\varepsilon_{C}(c_2)=\sum n\otimes\psi(c_1\varepsilon_C(c_2))=n\otimes_{\Bbbk}\psi(c)$, where $(*)$ follows from the fact that $n\square_D c$ belongs to $N\square_D C$ so that $\sum n_0\otimes n_1\otimes_{\Bbbk} c=\sum n\otimes \psi(c_1)\otimes_{\Bbbk} c_2$, and $(\epsilon_N\otimes D)\rho_{N\square_D C}(n\square_D c)=(\epsilon_N\otimes D)(\sum (n\square_D c_1)\otimes_{\Bbbk} \psi(c_2))=\sum n\varepsilon_C(c_1)\otimes\psi (c_2)=\sum n\otimes\psi (\varepsilon_C(c_1)c_2)=n\otimes_{\Bbbk}\psi(c)$.
\\
Now, we show that the triangle identities are satisfied. Let $M$ be in $\m^C$. For all $m\in M$, we have $(\epsilon_{\psi_*M}\circ\psi_* \eta_M) (m) =\epsilon_{\psi_*M}(\sum m_0\square_D m_1)=\sum m_0\varepsilon_C (m_1)=m=\id_{\psi_*M}(m)$, where $\psi_*M$ is $M$ viewed as a right $D$-comodule. Now, let $N$ be in $\m^D$. For all $n\in N$, $c\in C$ we have $((\epsilon_N\square_D C)\circ \eta_{N\square_DC})(n\square_D c)=(\epsilon_N\square_D C)(\sum (n\square_Dc)_0\square_D (n\square_D c)_1)=\sum \epsilon_N((n\square_Dc)_0)\square_D(n\square_Dc)_1= \sum \epsilon_N(n\square_Dc_1)\square_Dc_2=\sum n\varepsilon_{C}(c_1)\square_Dc_2= n \square_D \sum \varepsilon_{C}(c_1)c_2=n\square_D c=Id_{N\square_DC}(n\square_Dc)$.
\end{invisible}
It is known that
\begin{itemize}
	\item $\psi_*$ is separable if and only if the canonical morphism $\bar{\Delta}_C:C\to C\square_D C$ is split-mono as a $C$-bicomodule map, see \cite[Theorem 2.4]{CGN97};
	\item $\psi_*$ is naturally full if and only if $\bar{\Delta}_C$ is split-epi  as a $C$-bicomodule map, see \cite[Examples 3.23 (1)]{ACMM06};
	\item  $\psi^*$ is separable if and only if $\psi$ is split-epi as a $D$-bicomodule map, see \cite[Theorem 2.7]{CGN97};
	\item $\psi^*$ is naturally full if and only if $\psi$ is split-mono as a $D$-bicomodule map, see \cite[Examples 3.23 (1)]{ACMM06}.
\end{itemize}
Since $\psi_*$ is faithful, we have that $\psi_*$ is semiseparable if and only if it is separable. The semiseparability of $\psi^*$ is investigated in the following result.
\begin{prop}\label{prop:coinduc-coalg}
	Let $\psi :C\to D$ be a morphism of coalgebras. Then, the coinduction functor $\psi^* =(-)\square_D C:\m^D\to\m^C$ is semiseparable if and only if $\psi$ is a regular morphism of $D$-bicomodules if and only if there is a $D$-bicomodule morphism $\chi:D\to C$ such that $\varepsilon_C \circ\chi\circ\psi=\varepsilon_C$.
\end{prop}
\proof
Assume that $\psi^*$ is semiseparable. By Theorem \ref{thm:rafael}, there exists a natural transformation $\gamma : \id_{\m^D} \rightarrow \psi_*\psi^*$ such that $\epsilon_N\circ\gamma_N\circ\epsilon_N = \epsilon_N$, for any $N\in\m^D$. Since $D$ is a right $D$-comodule, consider the right $D$-comodule map $\gamma_D : D\to D\square_D C$ and define the map $\chi:D\to C$ as $\chi:=l_C\circ \gamma_D$, where $l_C:D\square_D C\to C,\sum_i d_i\otimes c_i\mapsto \sum_i \varepsilon_D(d_i)c_i,$ is the canonical isomorphism. Note that $\psi$ is a morphism of $D$-bicomodules and $\psi=\epsilon_D\circ l_C^{-1}$. We show that $\chi$ is a morphism of $D$-bicomodules. For any $f\in D^*=\Hom_\Bbbk (D,\Bbbk)$, consider the morphism of right $D$-comodules $\hat{f}:N\to D$, $n\mapsto \sum f(n_0)n_1$. For any $n\in N$, denote $\gamma(n)$ by $\sum_in_i\otimes c_i$. Then, by naturality of $\gamma$, we have that $l_C\gamma_D\hat{f}(n)=\chi\hat{f}(n)=\sum f(n_0)\chi (n_1)$ is equal to $l_C(\hat{f}\square_D C)\gamma_N(n)=l_C(\hat{f}\square_D C)(\sum_in_i\otimes c_i)=l_C(\sum_i f(n_{i_0})n_{i_1}\otimes c_i)=\sum_i f(n_i)c_i$. Since $f$ is arbitrary, it follows that for any $N\in\m^D$ and for all $n\in N$, $\gamma_N(n)=\sum_in_i\otimes c_i=\sum n_0\otimes\chi(n_1)$. In particular, consider $\gamma_D:D\to D\square_D C$.
Since $\sum d_1\otimes\chi(d_2)=\gamma_D(d)\in D\square_D C$, we have $\sum d_1\otimes d_2\otimes\chi(d_3)=\sum d_1\otimes\psi(\chi(d_2)_1)\otimes\chi(d_2)_2$. If we apply on both sides $\varepsilon_D\otimes \id$, we get $\sum d_1\otimes \chi(d_2)=\sum \psi(\chi(d)_1)\otimes\chi(d)_2$ which means that $\chi$ is a morphism of left $D$-comodules whence of $D$-bicomodules.
Moreover, we have $\psi\circ\chi\circ\psi =(\epsilon_D\circ l^{-1}_C)\circ (l_C\circ \gamma_D)\circ (\epsilon_D\circ l^{-1}_C)=\epsilon_D\circ  \gamma_D\circ \epsilon_D\circ l^{-1}_C=\epsilon_D\circ l^{-1}_C = \psi ,$ hence $\chi$ is a regular morphism of $D$-bicomodules.
\par
Assume that  $\psi$ is a regular morphism of $D$-bicomodules, i.e. there is a $D$-bicomodule morphism $\chi:D\to C$ such that $\psi\circ\chi\circ\psi=\psi$. Then $\varepsilon_C \circ\chi\circ\psi=\varepsilon_D\circ\psi\circ \chi\circ\psi=\varepsilon_D\circ\psi=\varepsilon_C$.

Assume now there is a $D$-bicomodule morphism $\chi:D\to C$ such that $\varepsilon_C \circ\chi\circ\psi=\varepsilon_C$ and let us prove that $\psi^*$ is semiseparable.
For any $N\in\m^D$ define $\gamma_N: N\to N\square_D C$ as $\gamma_N (n)=\sum n_0\otimes \chi(n_1)$, for every $n\in N$. Using that $\chi$ is a left $D$-comodule morphism, one easily checks that the image of $\gamma_N$ is really contained in $N\square_DC$. Moreover $\gamma_N$ comes out to be a right $D$-comodule morphism, since $\chi$ is a morphism of right $D$-comodules, and natural in $N$.
\begin{invisible}
Its image is contained in $N\square_DC$ as $(\rho_N\otimes C)(\sum n_0\otimes \chi(n_1))=\sum n_{0_0}\otimes n_{0_1}\otimes \chi(n_1)$ and $(N\otimes\rho_C)(\sum n_0\otimes \chi(n_1))=\sum n_0\otimes (\psi\otimes C)\Delta_C(\chi(n_1))=\sum n_0\otimes (D\otimes\chi)\Delta_D (n_1)=\sum n_0\otimes n_{1_1}\otimes \chi(n_{1_2})$, since $\chi$ is a left $D$-comodule morphism. Moreover, $\gamma_N$ is a right $D$-comodule morphism, as for all $n\in N$ we have $(\gamma_N\otimes D)\rho_N(n)=(\gamma_N\otimes D) (\sum n_0\otimes n_1)=\sum n_0\otimes \chi(n_{1})\otimes n_{2}$ and $\rho_{N\square_DC}\gamma_N(n)=\rho_{N\square_DC}(\sum n_0\otimes \chi(n_1))=\sum n_{0}\otimes \chi(n_{1})_1\otimes \psi(\chi(n_{1})_2)$ and $\chi$ is a morphism of right $D$-comodules.
Furthermore, it is also natural in $N$. Indeed, for any morphism $f:N\to M$ in $\m^D$, $(\gamma_M\circ f)(n) = \sum f(n)_0\otimes \chi(f(n)_1)=\sum f(n_0)\otimes\chi (n_1)=(f\square_D C) (\sum n_0\otimes\chi (n_1))=((f\square_D C )\circ\gamma_N) (n)$.
\end{invisible}
For every $n\in N$, $c\in C$, we have $\gamma_N\epsilon_N (n\square_D c)=\gamma_N (n\varepsilon_C(c))=\gamma_N (n)\varepsilon_C(c)=\sum n_0\otimes \chi(n_1)\varepsilon_C(c)=\sum n\otimes \chi\psi(c_1)\varepsilon_C(c_2)=\sum n\otimes \chi\psi(c)$, where in the second-last equality we used that $n\square_D c$ belongs to $N\square_D C $. Thus
$\epsilon_N\gamma_N\epsilon_N (n\square_D c)=\epsilon_N(\sum n\otimes\chi\psi(c))=\sum n\varepsilon_C \chi\psi(c)=n\varepsilon_C(c)=\epsilon_N (n\square_D c).
$
Therefore, by Theorem \ref{thm:rafael}, $\psi^*$ is semiseparable.
\endproof

\begin{es}
It is known that the Axiom of Choice is equivalent to require that, for any function $f:A\to B$, there is a function $g:B\to A$ such that $f\circ g\circ f=f$.
\begin{invisible}
Consider the canonical factorization $f=i\circ \pi$ where $i:\mathrm{Im}(f)\to B,b\mapsto b,$ and $\pi:A\to \mathrm{Im}(f),a\mapsto f(a)$.
If the Axiom of Choice holds, then the surjection $\pi$ has a right inverse $\sigma.$ On the other hand $i$ always admits a left inverse $p$. Set $g:=\sigma\circ p$. Then $f\circ g\circ f=i\circ \pi\circ \sigma\circ p\circ i\circ \pi=i\circ \pi=f$. Conversely, let $f:A\to B$ be a surjective function. If there is $g:B\to A$ such that $f\circ g\circ f=f$ then, by cancelling $f$, we obtain $f\circ g=\id$ i.e. any surjective function has a right inverse which is equivalent to Axiom of Choice. [Irving Kaplansky, "Set Theory and Metric Spaces". page 64].
\end{invisible}
Consider the group-like coalgebras $\Bbbk A$ and $\Bbbk B$ and the coalgebra map $\psi:=\Bbbk f:\Bbbk A\to\Bbbk B$ defined by setting $\psi(a)=f(a)$, for every $a\in A$. Define also the linear map $\chi: \Bbbk B\to\Bbbk A$ by setting $\chi(b)=g(b)$ if $b\in\mathrm{Im}(f)$ and $\chi(b)=0$ otherwise, for all $b\in B$. It is easy to check that $\chi$ is a $\Bbbk B$-bicomodule morphism such that $\varepsilon_{\Bbbk A} \circ\chi\circ\psi=\varepsilon_{\Bbbk A}$. Thus, by Proposition \ref{prop:coinduc-coalg}, the functor $\psi^* =(-)\square_{\Bbbk B} \Bbbk A:\m^{\Bbbk B}\to\m^{\Bbbk A}$ is semiseparable.
\begin{invisible}
 Let us check that $\chi$ is left $\Bbbk B$-comodule map i.e. $\psi(\chi(b)_1)\otimes \chi(b)_2=b_1\otimes \chi(b_2)$ i.e. $\psi(\chi(b)_1)\otimes \chi(b)_2=b\otimes \chi(b)$ for all $b\in B$. This equality is trivially true if $b\notin\mathrm{Im}(f)$ as $\chi(b)=0$. Otherwise there is $a\in A$ such that $b=f(a)$ and, in this case, $\chi(b)=g(b)$ so that $\psi(\chi(b)_1)\otimes \chi(b)_2=\psi(g(b)_1)\otimes g(b)_2=\psi(g(b))\otimes g(b)=f(g(b))\otimes g(b)=f(g(f(a)))\otimes g(b)=f(a)\otimes g(b)=b\otimes \chi(b)$ as desired.  Let us check that $\chi$ is left $\Bbbk B$-comodule map i.e. $\chi(b)_1\otimes \psi(\chi(b)_2)=\chi(b_1)\otimes b_2$ i.e. $\chi(b)_1\otimes \psi(\chi(b)_2)=\chi(b)\otimes b$ for all $b\in B$. As above it suffices to test it for $b=f(a)$ and $\chi(b)=g(b)$: $\chi(b)_1\otimes \psi(\chi(b)_2)=g(b)_1\otimes \psi(g(b)_2)=g(b)\otimes \psi(g(b))=g(b)\otimes f(g(b))=g(b)\otimes f(g(f(a)))=g(b)\otimes f(a)=\chi(b)\otimes b$. Finally $\varepsilon_{\Bbbk A} \chi\psi(a)=\varepsilon_{\Bbbk A} \chi(f(a))=\varepsilon_{\Bbbk A} (g(f(a)))=1=\varepsilon_{\Bbbk A}(a)$.
\end{invisible}
However it is neither separable nor naturally full in general. Indeed, if $\psi^*$ is separable, then $\psi $ is split-epi whence surjective. In this case $f$ must be surjective too.
\begin{invisible}
  Let $b\in B$. If $\psi=\Bbbk f $ is surjective, there is $\sum_{a\in A}k_aa\in \Bbbk A$ such that $b=\psi(\sum_{a\in A}k_aa)=\sum_{a\in A}k_a\psi(a)=\sum_{a\in A}k_af(a)$. Since the group-like elements are linearly independent, we get there is $a\in A$ such that $b=f(a)$ and hence $f$ is surjective.
  \end{invisible}
  Similarly, if $\psi^*$ is naturally full, then  $\psi $ is split-mono whence injective. In this case $f$ must be injective too.
\end{es}

\subsection{Corings}\label{es:coring}
Let $R$ be a ring. Recall that an $R$-\emph{coring} \cite{Sw75} is an $R$-bimodule $\cc$ together with $R$-bimodule maps $\Delta_{\cc} : \cc\to\cc\otimes_R\cc$ and $\varepsilon_{\cc}:\cc\to R$, called the comultiplication and the counit, respectively, such that $\Delta_{\cc}$ is coassociative and counital similarly to the case of coalgebras.
Given an $R$-coring $\cc$, a \emph{right $\cc$-comodule} is a right $R$-module $M$ together with a right $R$-linear map $\rho_M :M\to M\otimes_R\cc $, called the coaction, that is coassociative and right counital.
A map between right $\cc$-comodules is defined in the expected way.
Let $\m^{\cc}$ denote the category of right $\cc$-comodules and consider the induction functor $$G:=(-)\otimes_R\cc : \text{Mod-}R\to\m^{\cc},\quad M\mapsto M\otimes_R \cc,\quad f\mapsto f\otimes_R \cc,$$ which is the right adjoint of the forgetful functor $F:\m^{\cc}\to\mathrm{Mod}\text{-}R$, see e.g. \cite[Lemma 3.1]{Brz02}. The right $\cc$-comodule structure of $M\otimes_R \cc$ is given by $M\otimes_R\Delta_{\cc}$. The unit and counit of the adjunction are given by $\eta_M = \rho_M :M\to M\otimes_R \cc $, for every $M\in \m^{\cc}$, and $\epsilon_N = N\otimes_R\varepsilon_{\cc}: N\otimes_R\cc\to N$, $\epsilon_N (n\otimes_R c)=n\varepsilon_{\cc}(c)$, for every $N\in \text{Mod-}R$, $n\in N$, $c\in\cc$, respectively.


Denote by $\cc^R = \{c\in\cc \mid rc=cr,\text{ } \forall r\in R\}$ the set of invariant elements in $\cc$.

Concerning the separability and natural fullness of $F$ and $G$ we know that
\begin{itemize}
  \item $F$ is separable if and only if  the coring $\cc$ is coseparable, see \cite[Corollary 3.6]{Brz02};
  \item $F$ is naturally full if and only if $\Delta_{\cc}$ is surjective, see \cite[Proposition 3.13]{ACMM06};
  \item $G$ is separable if and only if there exists an invariant element $z\in\cc^R$ such that $\varepsilon_{\cc}(z)=1_R$, see \cite[Theorem 3.3]{Brz02}; if $G$ is separable, the coring $\cc$ is said to be \emph{cosplit} \cite[26.12]{BW03};
  \item $G$ is naturally full if and only if there exists an invariant element $z\in\cc^R$ such that $c=\varepsilon_{\cc}(c)z$, for every $c\in \cc$, see \cite[Proposition 3.13]{ACMM06}.
\end{itemize}

Since $F:\m^{\cc}\to\mathrm{Mod}\text{-}R$ is faithful, by Proposition \ref{prop:sep} (i), it is semiseparable if and only if it is separable so, although we are led to name a coring ``semicoseparable'' whenever $F$ is semiseparable, this would bring us back to the notion of coseparable coring.
Let us study when the induction functor $G=(-)\otimes_R\cc : \text{Mod-}R\to\m^{\cc}$ is semiseparable. In this case, we say that the $R$-coring $\cc$ is \textbf{semicosplit}. Note that, since separable functors are in particular semiseparable, it is obvious that cosplit corings are in particular semicosplit.


\begin{thm}\label{thm:inducoring}
Let $\cc$ be an $R$-coring. Then, the following are equivalent.
\begin{enumerate}
  \item $\cc$ is semicosplit. 
  \item The coring counit $\varepsilon_{\cc}:\cc\to R$ is regular as a morphism of $R$-bimodules.
  \item There exists an invariant element $z\in\cc^R$ such that $\varepsilon_{\cc}(z)\varepsilon_{\cc}(c)=\varepsilon_{\cc}(c)$ (equivalently such that $\varepsilon_{\cc}(z)c=c$), for every $c\in \cc$.
\end{enumerate}
\end{thm}

\proof
$(1)\Rightarrow (2)$. Assume that $\cc$ is semicosplit, i.e. the induction functor $G=(-)\otimes_R\cc$ is semiseparable. Then, by Theorem \ref{thm:rafael}, there exists a natural transformation $\gamma : \id _{\dd}\rightarrow FG$ such that $\epsilon\circ\gamma\circ\epsilon = \epsilon$. Consider the canonical isomorphism $l_\cc:R\otimes_R\cc \to \cc$. Since $R$ is a right $R$-module, consider the right $R$-linear map $\gamma_R : R\to R\otimes_R\cc $. Let us check it is also left $R$-linear. For any $r\in R$ define the morphism $f_r:R\to R$ by $f_r(r')=rr'$. Since $\gamma_R$ is natural, we have
\begin{equation*}
\gamma_R (rr')= (\gamma_R\circ f_r) (r')=((f_r\otimes_R\cc )\circ\gamma_R)(r')= r\gamma_R (r').
\end{equation*}
Thus $\gamma_R$ is a morphism of $R$-bimodules. Define the $R$-bimodule map $\alpha:=l_\cc\circ \gamma_R:R\to \cc$. By noting that $\varepsilon_\cc=\epsilon_R\circ l_\cc^{-1}$, we get $$\varepsilon_\cc\circ\alpha\circ\varepsilon_\cc =(\epsilon_R\circ l_\cc^{-1})\circ (l_\cc\circ \gamma_R)\circ (\epsilon_R\circ l_\cc^{-1})=\epsilon_R\circ  \gamma_R\circ \epsilon_R\circ l_\cc^{-1}=\epsilon_R\circ l_\cc^{-1} = \varepsilon_\cc$$ so that $\varepsilon_\cc$ is a regular morphism of $R$-bimodules.

$(2)\Rightarrow (3)$. Assuming the regularity of $\varepsilon_\cc$, i.e. the existence of an $R$-bimodule map $\alpha$ such that $\varepsilon_\cc\circ\alpha\circ\varepsilon_\cc=\varepsilon_\cc$, we can set $z=\alpha (1_R)\in\cc$. For $r\in R$, we have $rz=r\alpha (1_R)=\alpha (r)=\alpha (1_R)r=zr$ so that $z$  is in $\cc^R$. Moreover, from $\varepsilon_{\cc}(c)=\varepsilon_\cc\alpha\varepsilon_\cc(c)=\varepsilon_\cc\alpha (1_R\varepsilon_\cc(c)) = \varepsilon_\cc\alpha(1_R)\varepsilon_\cc(c)=\varepsilon_{\cc}(z)\varepsilon_{\cc}(c)$ it follows that $\varepsilon_{\cc}(c)=\varepsilon_{\cc}(z)\varepsilon_{\cc}(c)$, for every $c\in \cc$.

$(3)\Rightarrow (1)$. Suppose there exists $z\in \cc^R$ such that $\varepsilon_{\cc}(c)=\varepsilon_{\cc}(z)\varepsilon_{\cc}(c)$, for every $c\in \cc$. For any $N\in\text{Mod-}R$ define $\gamma_N: N\to N\otimes_R \cc$, $\gamma_N (n)=n\otimes_R z$, for every $n\in N$. Since $z\in \cc^R$, for every $n\in N$, $r\in R$, we have $\gamma_N (nr)=nr\otimes_R z=n\otimes_R rz=n\otimes_R zr=\gamma_N (n)r$, so $\gamma_N$ is a right $R$-module morphism, and it is also natural in $N$: indeed, for any morphism $f:N\to M$ in $\text{Mod-}R$,
$(\gamma_M\circ f)(n) = f(n)\otimes_R z= ((f\otimes_R\cc )\circ\gamma_N) (n)$.
Moreover, for every $n\in N$, $c\in\cc$, we have
\begin{equation*}
(\epsilon_N\circ\gamma_N\circ\epsilon_N )(n\otimes_R c)=\epsilon_N\gamma_N (n\varepsilon_{\cc}(c))=\epsilon_N (n\otimes_R\varepsilon_{\cc}(c)z)=n\varepsilon_{\cc}(z)\varepsilon_{\cc}(c)\overset{(*)}{=}n\varepsilon_{\cc}(c)=\epsilon_N (n\otimes_R c),
\end{equation*}
where $(*)$ follows from the assumption $\varepsilon_{\cc}(c)=\varepsilon_{\cc}(z)\varepsilon_{\cc}(c)$. Therefore, by Theorem \ref{thm:rafael} $G$ is semiseparable and $\cc$ is semicosplit.

Finally, assume that $\varepsilon_{\cc}(z)\varepsilon_{\cc}(c)=\varepsilon_{\cc}(c)$, for every $c\in\cc$. Then $\varepsilon_{\cc}(z)c=\varepsilon_{\cc}(z)\varepsilon_{\cc}(c_{(1)})c_{(2)}=\varepsilon_{\cc}(c_{(1)})c_{(2)}=c$ and hence $\varepsilon_{\cc}(z)c=c$. Conversely, if $\varepsilon_{\cc}(z)c=c$, for every $c\in\cc$, then $\varepsilon_{\cc}(z)\varepsilon_{\cc}(c)=\varepsilon_{\cc}(\varepsilon_{\cc}(z)c)=\varepsilon_{\cc}(c)$.
\endproof

\begin{rmk}
At the beginning of Subsection \ref{es:coring} we mentioned that the functor $G:=(-)\otimes_R\cc : \text{Mod-}R\to\m^{\cc}$ is naturally full if and only if there exists $z\in\cc^R$ such that $c=\varepsilon_{\cc}(c)z$, for every $c\in\cc$. We expect this characterization to be a particular case of Theorem \ref{thm:inducoring}, as a naturally full functor is semiseparable by Proposition \ref{prop:sep} (ii). Indeed, if there exists $z\in\cc^R$ such that $c=z\varepsilon_{\cc}(c)$ for every $c\in\cc$, then $\varepsilon_{\cc}(c)=\varepsilon_{\cc}(z\varepsilon_{\cc}(c))=\varepsilon_{\cc}(z)\varepsilon_{\cc}(c)$, for every $c\in\cc$, and equivalently, $\varepsilon_\cc(z)c=\varepsilon_\cc(z)z\varepsilon_\cc(c)=\varepsilon_\cc(z\varepsilon_\cc(c))z=\varepsilon_\cc(c)z=c$. Analogously, we recalled that $G$ is separable if and only if there exists an invariant element $z\in\cc^R$ such that $\varepsilon_{\cc}(z)=1_R$ and hence the equality $\varepsilon_{\cc}(c)=\varepsilon_{\cc}(z)\varepsilon_{\cc}(c)$ trivially holds true in this case.
\begin{invisible}
Note that, in this case, since $r=r1_R=r\varepsilon_{\cc}(z)=\varepsilon_{\cc}(r z)$, for every $r\in R$, $\varepsilon_{\cc}$ is a surjective morphism, as it happens for instance when we consider the Sweedler coring $\cc=S\otimes_R S$ associated with the extension of rings $\varphi : R\to S$, where the counit $\varepsilon_{\cc}: S\otimes_R S\to S$ is given by $\varepsilon_{\cc}(a\otimes_R a')=aa'$ and it is indeed surjective.
\end{invisible}
\end{rmk}

Next aim is to show that, when $G$ is semiseparable, then we can provide an explicit factorization of it as a bireflection followed by a separable functor. By Corollary \ref{cor:fact-birefl}, this factorization amounts to the one given by the coidentifier, up to a category equivalence.

\begin{cor}\label{cor:inducoring}
Let $\cc$ be an $R$-coring. Then, $\cc$ is semicosplit if and only if the induction functor $G=(-)\otimes_R\cc : \mathrm{Mod}\text{-}R\to\m^{\cc}$ factors up to isomorphism as $\psi_*\circ G'$ where $\psi_*=(-)\square_I\cc:\m^{I}\to \m^{C}$ is separable and $G'=(-)\otimes_RI : \mathrm{Mod}\text{-}R\to\m^{I}$ is a bireflection for some morphism of corings $\psi:\cc\to I$. 
\end{cor}

\begin{proof}
Assume $\cc$ is semicosplit, i.e $G=(-)\otimes_R\cc : \text{Mod-}R\to\m^{\cc}$ is semiseparable. Then, by Theorem \ref{thm:inducoring}, there exists an invariant element $z_\cc\in\cc^R$ such that $\varepsilon_{\cc}(c)=\varepsilon_{\cc}(z_\cc)\varepsilon_{\cc}(c)$, for every $c\in \cc$. We observe that, since $\varepsilon_\cc$ is a morphism of bimodules, $I:=\mathrm{Im}(\varepsilon_\cc)$ is an ideal of $R$ with multiplicative identity $z:=\varepsilon_\cc(z_\cc)$. Indeed, for any $r\in I$ there is $c\in\cc$ such that $r=\varepsilon_\cc(c)$ and hence $rz=\varepsilon_\cc(c)\varepsilon_\cc(z_\cc)=\varepsilon_\cc(c)=r$. Therefore the morphism $\varphi:R\to I,r\mapsto r z$, is a ring epimorphism (in fact it is surjective) and hence the map $m_I:I\otimes_R I\to I$ is bijective, see \cite[Proposition XI.1.2]{Ste75}. Thus we can consider $\Delta_I = m_I^{-1}: I\to I\otimes_{R}I$, $\Delta_I(i)=i\otimes_R z=z\otimes_R i$, so that $(I,\Delta_I,\varepsilon_I)$ is an $R$-coring, where the counit $\varepsilon_I:I\hookrightarrow R$ is the canonical inclusion. Note that $\psi:\cc\to I,c\mapsto\varepsilon_\cc(c)$, is a morphism of corings and consider the corresponding coinduction functor $\psi_*=(-)\square_I\cc:\m^{I}\to \m^{\cc}$. We recall from \cite{GT02} that, given $M$ a $(\cc',\cc )$-bicomodule and $N$ a $(\cc,\cc'')$-bicomodule, where $\cc',\cc,\cc''$ are corings over the rings $R',R,R''$, respectively, then $M\square_\cc N$ is the kernel of the $(\cc',\cc'')$-bicomodule map
\begin{displaymath}
\xymatrixcolsep{8pc}\xymatrix{M\otimes_R N\ar[r]^-{\omega_{M,N}:=\rho_M\otimes_R N-M\otimes_R\lambda_N}& M\otimes_R\cc\otimes_R N,}
\end{displaymath}
where $\rho_M$ and $\lambda_N$ are the right and the left $\cc$-comodule structures of $M$ and $N$, respectively.\par
Consider also the induction functor $G':=(-)\otimes_RI : \text{Mod-}R\to\m^{I}$. Our aim is to prove that $G$ factors as $G\cong\psi_*\circ G'$, that $\psi_*$ is separable and $G'$ is a bireflection.

First let us check that $G\cong\psi_*\circ G'$. In fact, for every $T\in\text{Mod-}R$,
\begin{displaymath}
(\psi_*\circ G')(T)=\psi_*(T\otimes_R I)=(T\otimes_R I)\square_I\cc\overset{(*)}\cong T\otimes_R(I\square_I\cc)\cong T\otimes_R\cc = G(T),
\end{displaymath}
where we note that the above isomorphism $(*)$ follows e.g. from \cite[10.6, page 95]{BW03} once observed that
$\lambda_\cc(c)=\psi(c_{(1)})\otimes_R c_{(2)}=z\otimes_R \varepsilon_\cc(c_{(1)})c_{(2)}=z \otimes_R c$ and hence $\omega_{I,\cc}(i\otimes_R c)=\rho_I(i)\otimes_R c-i\otimes_R\lambda_\cc(c)=i\otimes_R z\otimes_R c-i\otimes_Rz\otimes_R c=0 $ so that $\omega_{I,\cc}$ is the zero map whence trivially \emph{$T$-pure} \cite[40.13]{BW03}.
Let us check that $G'$ is a bireflection. To this aim, first note that, since $i=\varepsilon_I(i)z$, for every $ i\in I$, then $G'$ is naturally full by the characterization of natural fullness of the induction functors we recalled at the beginning of the present subsection.

The functor $G'$ is right adjoint of the forgetful functor $F'$ and the unit is $\eta'_M=\rho_M:M\to M\otimes_RI$ for every $(M,\rho_M)$ in $\m^I$. Since $I=Rz$, for every $m\in M$ there is $m'\in M$ such that $\rho_M(m)=m'\otimes_Rz$ and hence $m=\sum m_0\varepsilon_I(m_1)=\sum m_0m_1=m'z.$ As a consequence $\rho_M(m)=m'\otimes_Rz=m'\otimes_Rzz=m'z\otimes_Rz=m\otimes_Rz$ for every $m\in M$. Now, given $w\in M\otimes_RI$, there is $m\in M$ such that $w=m\otimes_R z=\rho_M(m)$ and hence $\rho_M$ is surjective. Since it is also split-mono via $r_M\circ(M\otimes_R\varepsilon_I)$, where $r_M:M\otimes_R R\to M$ is the canonical isomorphism, we get that $\eta'_M=\rho_M$ is invertible and hence $F'$ is fully faithful. Hence $G'$ is a naturally full coreflection thus a bireflection by Theorem \ref{thm:frobenius}.

It remains to check that $\psi_*$ is separable. If we see $\cc$ as an $I$-bicomodule with left structure $\lambda_\cc: \cc\to I\otimes_R\cc , c\mapsto z\otimes c$, and right structure $\rho_\cc : \cc\to\cc\otimes_R I , c\mapsto c\otimes z$, the map $\nu:I\to \cc,i\mapsto i z_\cc = z_\cc i$, is an $I$-bicomodule morphism which satisfies $\psi\circ\nu =\id _I$. Indeed, $\psi(\nu(i))=\psi(i z_\cc )=\varepsilon_\cc(i z_\cc )=i\varepsilon_\cc( z_\cc )
=iz=i$. The existence of $\nu$ implies that $\psi_*:(-)\square_I\cc:\m^{I}\to \m^{C}$ is separable by \cite[Theorem 5.8]{GT02} in case $A=B=R$ and $\dd=I$ once we have checked its hypothesis, namely that both $_RR$ and $_R\cc$ preserve the equalizer of $(\rho_M\otimes_R\cc,M\otimes_R\lambda_\cc)$ for every $(M,\rho_M)$ in $\m^I$. By the foregoing, for such an $(M,\rho_M)$, one has $\rho_M(m)=m\otimes z$ so that $\omega_{M,\cc}(m\otimes_R c)=\rho_M(m)\otimes_R c-m\otimes_R\lambda_\cc(c)=m\otimes_R z\otimes_R c-m\otimes_Rz\otimes_R c=0 $ so that $\omega_{M,\cc}=\rho_M\otimes_R\cc-M\otimes_R\lambda_\cc$ is the zero map. Thus both $_RR$ and $_R\cc$ trivially preserve the equalizer of $(\rho_M\otimes_R\cc,M\otimes_R\lambda_\cc)$ for every $(M,\rho_M)$ in $\m^I$ as desired.
\begin{invisible}
OLD: This is a particular instance of \cite[Theorem 5.8]{GT02}, where in this case, for every right $I$-comodule $M$, the morphism of right $I$-comodules
\begin{displaymath}
\xymatrixcolsep{3pc}\xymatrix{M\ar[r]^-{\rho_M}&M\otimes_R I\ar[r]^-{M\otimes_R\nu}& M\otimes_R\cc}
\end{displaymath}
factors through $M\square_I\cc$, as
\begin{displaymath}
\begin{split}
(M\otimes_R \lambda_\cc)(M\otimes_R\nu)\rho_M &=(M\otimes_R I\otimes_R\nu)(M\otimes_R\Delta_I)\rho_M\\
&=(M\otimes_R I\otimes_R\nu)(\rho_M\otimes_R I)\rho_M=(\rho_M\otimes_R\nu)\rho_M \\
&=(\rho_M\otimes_R\cc)(M\otimes_R\nu)\rho_M .
\end{split}
\end{displaymath}
Moreover, the induced natural transformation $\gamma_M:M\to M\square_I\cc$ splits $\chi_M:=r_M(M\otimes_R\varepsilon_\cc)$, where $r_M:M\otimes_R R\to M$ is the canonical isomorphism. By \cite[Section 5.6]{GT02} the restriction of the natural transformation $\chi$ to $(-)\square_I\cc$ gives the counit of the adjunction $U:=(-)\otimes_R R\dashv (-)\square_I\cc :\m^I\to\m^\cc$. So it follows that $\psi_*$ is separable. 
\end{invisible}
\end{proof}

\begin{rmk}Consider the $R$-coring $I$ of Corollary \ref{cor:inducoring}. By construction it is also a ring with unit $z$. Since the comultiplication $\Delta_I$ of $I$ is invertible, then $I$ is a coseparable $R$-coring. Thus, by \cite[Proposition 2.17]{BV09} there is a category isomorphism between the category $\m^I$ of right comodules over the coring $I$ and the category $\mathrm{Mod}\text{-}I$ of right modules over the ring $I$.
\end{rmk}

We already mentioned that a cosplit coring is always semicosplit. We now give a concrete example of a semicosplit coring $\cc$ which is not cosplit.

\begin{es}\label{phi_coring}
1) Let $\varphi:R\to S$ be a morphism of rings such that the induction functor $\varphi^*=S\otimes_R(-)$ is naturally full. As recalled in Subsetion \ref{es:extens}, there exists $\varepsilon\in {}_{R}\Hom_{R}(S,R)$ such that $\varphi \circ \varepsilon=\id_S$. Since, in particular,  $\varphi:R\to S$ is an epimorphism in the category of rings, by \cite[Proposition XI.1.2]{Ste75}, the multiplication $m:S\otimes_RS\to S$ is bijective and hence we can set $\Delta:=m^{-1}$ so that $\Delta(s)=s\otimes_R 1_S=1_S\otimes_R s$. We compute $$(\varepsilon\otimes_R S)\Delta(s)=(\varepsilon\otimes_R S)(1_S\otimes_R s)=\varepsilon(1_S)\otimes_R s=1_R\otimes_R \varepsilon(1_S)s=1_R\otimes_R \varphi\varepsilon(1_S)s =1_R\otimes_R s$$ and similarly $(S\otimes_R \varepsilon)\Delta(s)=(S\otimes_R \varepsilon)(s\otimes_R 1_S)=s\otimes_R1_S$. As a consequence $(S,\Delta,\varepsilon)$ is an $R$-coring. Now $\varepsilon(1_S)s=\varphi\varepsilon(1_S)s=1_Ss=s$ so that $z:=1_S\in S^R$ fulfills the conditions of Theorem \ref{thm:inducoring} guaranteeing that the functor $G:=(-)\otimes_{R}S : \text{Mod-}R\to\m^{S}$ is semiseparable and hence $S$ is a semicosplit $R$-coring. Nevertheless $S$ is not cosplit in general. In fact if $G$ is separable, as observed at the beginning of the present subsection, there exists $w\in S^R$ such that $1_R=\varepsilon(w)$ and hence, for every $r\in R$, we have $r=r1_R=r\varepsilon(w)=\varepsilon(r w)$ so that $\varepsilon$ is surjective which, together with the condition $\varphi \circ \varepsilon=\id_S$, implies that $\varphi$ and $ \varepsilon$ are mutual inverses.

2) To get an example of 1) with $\varphi$ not invertible, consider $S$ and $T$ rings, set $R:=S\times T$, take  $\varphi:R\to S,(s,t)\mapsto s$ and $\varepsilon:S\to R,s\mapsto (s,0)$. Then $S$ is a semicosplit but not cosplit $R$-coring.
%
\end{es}

\begin{es}\label{es:pure}
Let $R$ be a commutative ring and consider an idempotent ideal $I$ of $R$, assumed to be a pure right $R$-submodule. We recall that a submodule $N$ of an $R$-module $M$ is said to be \emph{pure} \cite[40.13]{BW03} if the inclusion $N\hookrightarrow M$ remains injective after tensoring by any right $R$-module.
Since $I$ is pure, we get that the multiplication $m_I:I\otimes_R I\to I$, $m_I(a\otimes_R a')= aa'$, is injective as it is obtained as the composition $I\otimes_R I\overset{I\otimes_R \varepsilon_I}{\to} I\otimes_R R\overset{\cong}{\to} I$, where $\varepsilon_I:I\to R$ is the canonical inclusion. Since $I$ is idempotent, i.e. $I^2=I$, we get that $m_I$ is also surjective whence bijective. Thus we can consider $\Delta_I = m_I^{-1}: I\to I\otimes_{R}I$ and write $\Delta_I(a)=\sum a_1\otimes_R a_2$ by means of Sweedler's notation. Then $\sum \varepsilon_I(a_1) a_2=\sum a_1 a_2=m_I(\sum a_1\otimes_R a_2)=m_I(\Delta_I(a))=a$ and similarly $\sum a_1 \varepsilon_I(a_2)=a$ so that $(I,\Delta_I,\varepsilon_I)$ is an $R$-coring. Now, the condition in Theorem \ref{thm:inducoring} for this coring is the existence of an element $z\in I^R$ such that $c=\varepsilon_{I}(z)c$ i.e., by definition of $\varepsilon_I$, the existence of $z\in I$ such that $c=zc=cz$ for every $c\in I$. This means that $z$ is the multiplicative identity in $I$. This goes back to a particular case of the ideal $I$ constructed in Corollary \ref{cor:inducoring} by taking $\cc=I$ and noting that $\mathrm{Im}(\varepsilon_I)=I$.
%
Moreover, in Example \ref{phi_coring} 2) we can identify $S$ with the idempotent ideal $I=S\times \{0\}$ of the ring $R=S\times S$, through the isomorphism $S\overset{\cong}{\to} I :s\to (s,0)$. In this case, we can take $z=(1,0)$ (note that $z\neq (1,1)=1_R$) and $\Delta_I(x):=x\otimes_R z= z\otimes_R x$.

\end{es}

\subsection{Bimodules}\label{es:bimod}

Let $R$ and $S$ be rings, and let ${}_{R}\m_{S}$ denote the category of $(R,S)$-bimodules. For an $(R,S)$-bimodule $M$ we often write ${}_R M$, $M_S$, ${}_R M_S$ to indicate the left $R$-module, the right $S$-module, the $(R,S)$-bimodule structure used, respectively, and morphisms in the corresponding categories are denoted by ${}_R\Hom(-,-)$, $\Hom_S(-,-)$ and ${}_R\Hom_S(-,-)$. \par
We recall from \cite{ABM08, ACMM06} that every $M\in {}_{R}\m_{S}$ defines an adjunction $\sigma^*\dashv \sigma_*$ formed by
\begin{itemize}
  \item the induction functor $\sigma^*:=(-)\otimes_R M: \text{Mod-}R\to \text{Mod-}S$,
  \item the coinduction functor $\sigma_*:=\Hom_S(M,-): \text{Mod-}S\to \text{Mod-}R$.
\end{itemize}
The unit $\eta$ and the counit $\epsilon$ of this adjunction are given for all $X\in\mathrm{Mod}\text{-}R$ and $Y\in \mathrm{Mod}\text{-}S$ by
\begin{equation*}
\eta_X : X\to \Hom_S(M,X\otimes_R M),\, x\mapsto [m\mapsto x\otimes_R m],\quad\epsilon_Y : \Hom_S(M ,Y)\otimes_R M\to Y,\, f\otimes_R m\mapsto f(m).
\end{equation*}
Given bimodules ${}_R M_S$ and ${}_{R'} N_{S}$, where $R'$ is a ring, the abelian group $ \Hom_S(M,N)$ is an $(R', R)$-bimodule via the multiplication defined by
\begin{equation*}
(r'fr)(m):=r'f(rm),\quad \text{for every } f\in \Hom_S(M,N), r\in R, m\in M, r'\in R'.
\end{equation*}
In particular, the endomorphism ring $\e:= \End_S(M)$ belongs to ${}_R\m_{R}$. We denote by
\begin{equation*}
{}^*M={}_R\Hom (M,R) \quad\text{and}\quad M^{*}=\Hom_S(M,S)
\end{equation*}
the left dual and the right dual of $M$, respectively, which both belong to ${}_S\m_R$. 

Given an $(R,S)$-bimodule $M$, in \cite{Su71} $R$ is said to be \emph{$M$-separable over $S$} if the evaluation map
\begin{equation}\label{eval2-sugano}
\ev_M:M\otimes_S {}^*M\to R, \quad \ev_M(m\otimes_S f)=f(m),
\end{equation}
is a split epimorphism of $R$-bimodules. By \cite[Theorem 34]{CMZ02}, this is equivalent to say that the functor ${}_R\Hom (M,-):R\text{-Mod}\to S\text{-Mod}$ is separable. Hereafter, we consider the right version of this definition. Explicitly, we say that $S$ is $M$\emph{-separable over} $R$, if the evaluation map
\begin{equation}\label{eval}
\ev_M:M^*\otimes_R M\to S, \quad \ev_M(f\otimes_R m)=f(m),
\end{equation}
is a split epimorphism of $S$-bimodules. This means  there is a central element $\sum_i f_i\otimes_R m_i$ $\in (M^*\otimes_R M)^S$ such that $\sum_i f_i(m_i)=1_S$.
\begin{rmk}\label{rmk:choice}
As we will see in Claim \ref{claim:bimod}, when $M_S$ is finitely generated and projective, the Eilenberg-Moore category $(\mathrm{Mod}\text{-}S)^{\sigma^*\sigma_*}$ results to be equivalent to the category $\m^\cc$ of right comodules over the comatrix $S$-coring $\cc$ which was defined in \cite{ElKGT03} as $\cc:=M^*\otimes_R M$. This explains our choice to use the right version of \eqref{eval2-sugano}, since otherwise we would have achieved the less usual coring $M\otimes_S {}^*M$. Moreover, our choice is also motivated by the fact that the map
\begin{equation}\label{def:varphiE}
 \varphi:R\to \e,\quad r\mapsto [m\mapsto rm],
\end{equation}
results to be a ring homomorphism.
\begin{invisible}($\varphi (rs)(m)=rsm=r(sm)=\varphi (r)(sm)=\varphi(r)\varphi(s)(m)$).\end{invisible}
If we had taken \eqref{eval2-sugano}, we would have been forced to choose $\e={}_R\mathrm{End}(M)={}_R\Hom (M,M)\in {}_S\m_S$ and to consider the ring homomorphism $\varphi: S\to\e^{\mathrm{op}}$, $s\mapsto [m\mapsto ms]$ into the opposite ring.
\begin{invisible} In $\e^{\op}$, $\varphi (s's)(m)=m\cdot_{\e^\op} s's= s'sm=s'(m\cdot_{\e^{\op}}s)=m\cdot_{\e^\op}s\cdot_{\e^\op}s'=\varphi (s')(m\cdot_{\e^\op}s)=\varphi (s')\varphi(s)(m)$.\end{invisible}
\end{rmk}
\par
Given an $(R,S)$-bimodule $M$, concerning the separability and natural fullness of $\sigma_*=\Hom_S(M,-)$ and $\sigma^*=(-)\otimes_R M$, we know that
\begin{itemize}
  \item $\sigma_*$ is separable if and only if $S$ is $M$-separable over $R$ (right version of \cite[Theorem 34]{CMZ02});\begin{invisible}
  It follows by the right version of the third part of \cite[Lemma 11]{CMZ02} and Rafael Theorem.
  \end{invisible}
  \item $\sigma_*$ is naturally full if and only if there is $\sum_i f_i\otimes_R m_i\in (M^*\otimes_R M)^S$ satisfying $\id_M\otimes_R m=\sum_i mf_i(-)\otimes_R m_i$ for all $m\in M$ (right version of \cite[Theorem 3.8 (1)]{ACMM06});
  \item if  $\sigma^*$ is separable, then (right version of \cite[Proposition 2.5]{Raf90}) there is $E\in {_R\Hom_R}(\e,R)$ such that $E\circ\varphi=\id_R$, i.e. $\varphi^*$ is separable, where $\varphi$ is the map in \eqref{def:varphiE}.

\item Assume $M$ is finitely generated and projective as a right $S$-module. Then, $\sigma^*$ is naturally full if and only if there is $E\in {_R\Hom_R}(\e,R)$ such that $\varphi\circ E=\id_\e$ (right version of \cite[Theorem 3.8 (2)]{ACMM06}). By what we recalled at the beginning of Subsection \ref{es:extens}, this is equivalent to say that $\varphi^*$ is naturally full.
\end{itemize}

Now, we investigate the semiseparability of $\sigma_*$ and, in the finitely generated and projective case, the one of $\sigma^*$. To this aim we first introduce the following definition, which will be mainly used in its first part by the same reasons discussed in Remark \ref{rmk:choice}.

\begin{defn}\label{bimod:semisep}
Let $R,S$ be rings and $M$ an $(R,S)$-bimodule. We say that $S$ is $M$\textbf{-semiseparable over} $R$ if there exists an element $\sum_i f_i\otimes_R m_i\in (M^*\otimes_R M)^S$ such that $\sum_i mf_i(m_i)=m$ for every $m\in M$. In a similar way, it is possible to define $R$ is $M$\textbf{-semiseparable over} $S$.
\end{defn}

Given an $(R,S)$-bimodule $M$, the equivalence between $(1)$ and $(3)$ in the following result is the semiseparable counterpart of \cite[Theorem 34]{CMZ02}.

\begin{thm}\label{thm:bimod}
Let $R,S$ be rings and $M$ an $(R,S)$-bimodule. Then, the following are equivalent.
\begin{enumerate}
  \item The functor $\sigma_*=\Hom_{S}(M,-): \mathrm{Mod}\text{-}S\to \mathrm{Mod}\text{-}R$ is semiseparable.
  \item $\ev_M:M^*\otimes_R M\to S$ is regular as a morphism of $S$-bimodules and $M\otimes_S\ev_M$ is surjective.
  \item $S$ is $M$-semiseparable over $R$.
\end{enumerate}
\end{thm}

\proof
It is known (e.g. the right version of \cite[Lemma 11]{CMZ02}) that there is a bijective correspondence
\begin{equation}\label{bimod:nat}
\Nat (\id _{{}_{\mathrm{Mod}\text{-}S}}, \sigma^*\sigma_*)\cong (M^*\otimes_R M)^S .
\end{equation}
Explicitly, a natural transformation $\gamma : \id _{\mathrm{Mod}\text{-}S}\rightarrow \sigma^*\sigma_*$ is mapped to $\gamma_S(1)\in (M^*\otimes_R M)^S$ while an element $\sum_i f_i\otimes_R m_i \in (M^*\otimes_R M)^S$ is mapped, for every $Y\in \mathrm{Mod}\text{-}S$, to
\begin{equation}\label{eq:bimod:nat}
\gamma_Y : Y\rightarrow  \Hom_S(M,Y)\otimes_R M, \quad \gamma_Y(y)=\sum_i yf_i(-)\otimes_R m_i .
\end{equation}
$(1)\Rightarrow(2).$ If the functor $\sigma_*$ is semiseparable, then by Theorem \ref{thm:rafael} there exists a natural transformation $\gamma : \id _{\mathrm{Mod}\text{-}S}\rightarrow \sigma^*\sigma_*$ such that $\epsilon\circ\gamma\circ\epsilon = \epsilon $. Consider the right $S$-module map $\gamma_S:S\to M^*\otimes_R M$. For every $s\in S$ set $f_s:S\to S,s'\mapsto ss'$. Since $\gamma_S$ is natural in $S$, we have
$\gamma_S(ss')=(\gamma_S\circ f_s)(s')=((\Hom_S(M ,f_s)\otimes_R M)\circ\gamma_S)(s')=s\gamma_S(s')$ so that $\gamma_S$ is $S$-bilinear.
Since $\epsilon_S=\ev_M$, from $\epsilon\circ\gamma\circ\epsilon = \epsilon $ we get $\ev_M\circ\gamma_S\circ\ev_M = \ev_M $ and hence $\ev_M$ is regular as a morphism of $S$-bimodules. Note that any $m\in M$ is of the form $m=\id(m)=\epsilon_M(\id\otimes_R m)$ so that $\epsilon_M$ is surjective. Thus, from $\epsilon_M\circ\gamma_M\circ\epsilon_M = \epsilon_M $, we get $\epsilon_M\circ\gamma_M = \id $. From \eqref{bimod:nat} we have that $\gamma_M$ is defined by \eqref{eq:bimod:nat} for $Y=M$, where $\sum_i f_i\otimes_R m_i=\gamma_S(1_S)\in (M^*\otimes_R M)^S$. Thus $$m=\id(m)=(\epsilon_M\circ\gamma_M)(m)= \sum_i mf_i(m_i)=r_M(M\otimes_S\ev_M)(\sum_im\otimes_Sf_i\otimes_Rm_i) $$ where $r_M:M\otimes_SS\to M$ is the canonical isomorphism. Thus, $r_M\circ(M\otimes_S\ev_M)$ is surjective and hence also $M\otimes_S\ev_M$ is surjective.

$(2)\Rightarrow(3).$ Assume that $\ev_M$ is regular as a morphism of $S$-bimodules, i.e. that there is an $S$-bimodule map $\gamma_S:S\to M^*\otimes_R M$ such that $\ev_M\circ\gamma_S\circ\ev_M = \ev_M $. Thus $(M\otimes_S\ev_M)\circ(M\otimes_S\gamma_S)\circ(M\otimes_S\ev_M) = (M\otimes_S\ev_M) $. If $M\otimes_S\ev_M$ is surjective, we get $(M\otimes_S\ev_M)\circ(M\otimes_S\gamma_S) = \id_{M\otimes_SS} $. Now set $\sum_i f_i\otimes_R m_i=\gamma_S(1_S)\in (M^*\otimes_R M)^S$. Thus, $S$ is $M$-semiseparable over $R$ as
$$m=r_M\id_{M\otimes_SS}(m\otimes_S1)=r_M(M\otimes_S\ev_M)(M\otimes_S\gamma_S)(m\otimes_S1)=\sum_i mf_i(m_i).$$
$(3)\Rightarrow(1).$  Assume $S$ is $M$-semiseparable over $R$. By definition, there exists an element $\sum_i f_i\otimes_R m_i \in (M^*\otimes_R M)^S$ such that $\sum_i f_i(m_i)m=m$ for every $m\in M$ and the corresponding natural transformation $\gamma : \id _{\mathrm{Mod}\text{-}S}\rightarrow \sigma^*\sigma_*$ from \eqref{bimod:nat} is given, for every $Y\in \mathrm{Mod}\text{-}S$, by \eqref{eq:bimod:nat}.
Moreover, for every $Y\in \mathrm{Mod}\text{-}S$, $m\in M$, $f\in \Hom_S(M,Y)$, we have
$\epsilon_Y\gamma_Y\epsilon_Y (f\otimes_R m)=\epsilon_Y\gamma_Y (f(m))
=\epsilon_Y(\sum_i f(m)f_i(-)\otimes_R m_i)
=\sum_i f(m)f_i(m_i)=f(\sum_i mf_i(m_i))=f(m)
=\epsilon_Y (f\otimes_R m)$. Thus $\epsilon$ is regular and by Theorem \ref{thm:rafael} (ii) $\sigma_*$ is semiseparable.
\endproof

\begin{rmk}
  In the setting of Theorem \ref{thm:bimod}, assume further that $M_S$ is projective. Then, the requirement that $M\otimes_S\ev_M$ is surjective is superfluous. Indeed, there is a dual basis formed by elements $m_i\in M,f_i\in M^*$, with $i\in I$, such that, for every $m\in M$, we have $m=\sum_{i\in I}m_if_i(m)$. By definition, $f_i(m)=0$ for almost all $i$. Thus there is a finite subset $I(m)$ of $I$ such that    $m=\sum_{i\in I(m)}m_if_i(m)= r_M(M\otimes_S\ev_M)(\sum_{i\in I(m)} m_i\otimes_Sf_i\otimes_Rm)$, whence $M\otimes_S\ev_M$ is surjective.
\end{rmk}

%

As a consequence of Theorem \ref{thm:bimod}, we have the following characterization of $M$-separability, for an $(R,S)$-bimodule $M$, which extends some known results, see e.g. \cite[Theorem 1]{Su71}, \cite[Corollary 2.4]{Raf90} and \cite[Proposition 4.3]{ABM08}.

\begin{cor}\label{cor:sep-bimod}
Let $R,S$ be rings and $M$ an $(R,S)$-bimodule. Then, $S$ is $M$-separable over $R$ if and only if $S$ is $M$-semiseparable over $R$ and $M_S$ is a generator.
\end{cor}
\proof
By what we recalled at the beginning of this subsection, $S$ is $M$-separable over $R$ if and only if $\sigma_*=\Hom_S(M,-):\mathrm{Mod}\text{-}S\to \mathrm{Mod}\text{-}R$ is a separable functor. By Proposition \ref{prop:sep} (i), this is equivalent to require that $\sigma_*$ is semiseparable and faithful. The semiseparability of $\sigma_*$ is equivalent to $S$ being $M$-semiseparable over $R$, by Theorem \ref{thm:bimod}.
Since the forgetful functor $U:\mathrm{Mod}\text{-}R\to\Set$ is faithful, the faithfulness of $\sigma_*$ is equivalent to the faithfulness of the composition $U\circ\sigma_*=\Hom_S (M,-):\mathrm{Mod}\text{-}S\to \Set$, i.e. to $M_S$ being a generator, see e.g. \cite[Section 6]{Ste75}.
\endproof

\begin{rmk}
Let $R,S$ be rings and $M$ an $(R,S)$-bimodule. If $M_S$ is a generator and $\varphi:R\to\e=\mathrm{End}_S(M)$ is a ring epimorphism, then by the right version of \cite[Proposition 3.11]{ACMM06}, the functor $\sigma_*$ is fully faithful, i.e. $\sigma^*$ is a reflection. Thus, by Theorem \ref{thm:frobenius}, $\sigma^*$ results to be semiseparable if and only if it is naturally full  if and only if it is Frobenius in this case.
\end{rmk}

We now obtain a different characterization of $M$-semiseparability of $S$ over $R$, for an $(R,S)$-bimodule $M$, that will allow us to exhibit an example where $S$ is $M$-semiseparable but not $M$-separable over $R$, see Example \ref{es:ssVSs}.

\begin{prop}\label{prop:ssVSs}
Let $R,S$ be rings and let $M$ be an $(R,S)$-bimodule. Then $S$ is $M$-semiseparable over $R$ if and only if there is a central idempotent $z\in S$ (necessarily unique) such that $M$ is obtained by restriction of scalars from an $(R, Sz)$-bimodule $N$ and $Sz$ is $N$-separable over $R$, via $\varphi :S\rightarrow Sz,s\mapsto sz$. Furthermore, $S$ is $M$-separable over $R$ if and only if $z=1_S$.
\end{prop}

\begin{proof}
Assume that $S$ is $M$-semiseparable over $R$, i.e. that there is a central element $\sum_i f_i\otimes_R m_i$ $\in (M^*\otimes_R M)^S$ such that $\sum_i mf_i(m_i)=m$, for every $m\in M$. Set $z:=\sum_i f_i(m_i)\in S$ so that $mz=m$, for every $m\in M$. Since $\ev_M:M^*\otimes_R M\to S$ is a morphism of $S$-bimodules, it induces a morphism of $\ev_M^S:(M^*\otimes_R M)^S\to S^S$ so that $z=\ev_M(\sum_i f_i\otimes_R m_i)\in S^S$, i.e. $z$ is central. Moreover $zz=\sum_i f_i(m_i)z=\sum_i f_i(m_iz)=\sum_i f_i(m_i)=z$, so that $z$ is idempotent. Since for every $m\in M$ one has $mz=m$, then $M$ becomes a right $Sz$-module, via $\mu_M:M\times Sz\to M, (m,sz)\mapsto ms$. Let us write $N$ for $M$ regarded as an $(R,Sz)$-bimodule so that $M=\varphi _{\ast }N$ where $\varphi:S\rightarrow Sz,s\mapsto sz$. Set $N^*:=\Hom_{Sz}(N,Sz)$. Then $\sum_i \varphi f_i\otimes_R m_i\in (N^*\otimes_R N)^{Sz}$ and $\sum_i \varphi f_i(m_i)=\varphi(z)=zz=z=1_{Sz}$ so that $Sz$ is $N$-separable over $R$.
\begin{invisible}
By definition $f_i\in M^*=\Hom_{S}(M,S)$ so that $f_i$ is right $S$-linear. Since $z$ is central also $\varphi:S\rightarrow Sz,s\mapsto sz$, is right $S$-linear so that $\varphi f_i$ is right $S$-linear whence a fortiori, right $Sz$-linear as $Sz\subseteq S$. Thus $\sum_i \varphi f_i\otimes_R m_i\in N^*\otimes_R N$. Let us check it is $Sz$-central. For all $s\in S$, se have $\sum_i s(\varphi\circ f_i)\otimes_R m_i=\sum_i\varphi\circ (sf_i)\otimes_R m_i=\sum_i\varphi\circ  f_i\otimes_R m_is$ so that $\sum_i \varphi f_i\otimes_R m_i\in (N^*\otimes_R N)^{S}\subseteq(N^*\otimes_R N)^{Sz}.$
\end{invisible}
Conversely, assume there is a central idempotent $z\in S$ such that such that $M$ is obtained by restriction of scalars from an $(R, Sz)$-bimodule $N$ and $Sz$ is $N$-separable over $R$, via $\varphi :S\rightarrow Sz,s\mapsto sz$. This implies $mz=m$ for every $m\in M$.
\begin{invisible}
Indeed $(M,\mu_M:M\times S\to M)$ is obtained by restriction of scalars from an $(R, Sz)$-bimodule $(N,\mu_N:N\times Sz\to N)$ and $Sz$ is $N$-separable over $R$, via $\varphi :S\rightarrow Sz,s\mapsto sz$. Thus $M=N$ and $\mu_M=\mu_N\circ (N\times \varphi  )$ so that $\mu_M(m, z)=\mu_N(m, zz)=\mu_N(m,1_Sz)=\mu_M(m, 1_S)=m$ i.e. $mz=m$ for every $m\in M$.
\end{invisible}
Since $Sz$ is $N$-separable over $R$, there is $\sum_i g_i\otimes_R m_i$ $\in (N^*\otimes_R N)^S$ such that $\sum_i g_i(m_i)=1_{Sz}=z$. Let $j:Sz\to S$ be the canonical injection. Then $f_i:=j\circ g_i\in M^*$ and $\sum_i f_i\otimes_R m_i$ $\in (M^*\otimes_R M)^S$. Moreover $\sum_i mf_i(m_i)=\sum_i mjg_i(m_i)=mj(z)=mz=m$ so that $S$ is $M$-semiseparable over $R$.
\begin{invisible}
$jg_i(ms)=g_i(ms)=g_i((mz)s)=g_i(m(sz))=g_i(m)sz=g_i(m)zs=g_i(m)s=jg_i(m)s$ where the second-last equality follows from the fact that $g_i(m)\in Sz$ and hence $g_i(m)z=g_i(m)$. Thus $jg_i\in M^*$ and hence $\sum_i jg_i\otimes_R m_i\in M^*\otimes_R M$. Let us check it is $S$-central: $\sum_i s(j\circ g_i)\otimes_R m_i=\sum_i j\circ (sg_i)\otimes_R m_i=\sum_i j\circ (szg_i)\otimes_R m_i=\sum_i j\circ (g_i)\otimes_R m_isz=\sum_i j\circ (g_i)\otimes_R m_izs=\sum_i j\circ (g_i)\otimes_R m_is$.
\end{invisible}
Assume there is another central idempotent $z'\in S$ such that $M=\varphi _{\ast }N'$ for some $(R,Sz')$-bimodule $N'$ and $Sz$ is $N'$-separable over $R$ via the ring homomorphism $\varphi' :S\rightarrow Sz',s\mapsto sz'$. Then $zz'=\sum_i f_i(m_i)z'=\sum_i f_i(m_iz')=\sum_i f_i(m_i)=z$. Exchanging the roles of $z$ and $z'$, we also get $z'z=z'$ so that $z=z'$.
Let $z\in S$ be a central idempotent such that $M=\varphi_*N $ where $Sz$ is $N$-separable over $R$ via $\varphi :S\rightarrow Sz,s\mapsto sz$. If $z=1_S$, then $S$ is $N$-separable over $R$ and $\varphi=\id_S$ so that $S$ is $M=\varphi_*N$-separable over $R$ as well. Conversely, if $S$ is $M$-separable over $R$, then $z=1_S$ is an idempotent as in the statement, whence the unique one.
%
 \end{proof}

The following is an instance of an $(R,S)$-bimodule $M$ such that $S$ is $M$-semiseparable but not $M$-separable over $R$.

\begin{es}\label{es:ssVSs}
Let $\varphi:S\to T$ be a ring homomorphism and assume that there is $E\in{}_S\Hom_S(T,S)$ such that $\varphi\circ E=\id_T$. If we set $z:=E(1_T)\in S$, then $z$ is a central idempotent in $ S$, the map $\varphi_{\mid Sz}:Sz\to T$ is a ring isomorphism and $\varphi:S\to T\cong Sz$ is the projection $s\mapsto sz$, see \cite[Proposition 3.1]{ACMM06}. By Proposition \ref{prop:ssVSs}, if $N$ is a $(R,T)$-bimodule such that $T$ is $N$-separable over $R$, then $M:=\varphi_*N$ is an $(R,S)$-bimodule such that $S$ is $M$-semiseparable over $R$. Moreover, if $S$ is also $M$-separable over $R$, then $z=1_S$, whence $\varphi$ is bijective. As a consequence, $S$ will not be $M$-separable over $R$ unless $\varphi:S\to T$ is bijective. As an example, let $\psi:\QQ\times \ZZ\to \QQ,(q,z)\mapsto q$ and $D:\QQ\to \QQ\times \ZZ,q\mapsto (q,0)$ be as in Example \ref{es:induction}. Then, if $N$ is a $(R,\QQ)$-bimodule such that $\QQ$ is $N$-separable over $R$, then the $(R,\QQ\times \ZZ)$-bimodule $M:=\psi_*N$ is such that $\QQ\times \ZZ$ is $M$-semiseparable but not $M$-separable over $R$. For instance  consider the $\QQ$-vector space $N:=\QQ^n$, with $n>1$, and take $R:=\QQ$. Let us check that $N$ is a $(\QQ,\QQ)$-bimodule such that $\QQ$ is $N$-separable over $\QQ$, with $nq=qn$ for all $n\in N,q\in\QQ$. Since $N$ is a free left $\QQ$-module, then it is a generator.
\begin{invisible}
See  \cite[page 94]{Ste75} just before Section 2.
\end{invisible}
Moreover $\e=\mathrm{End}_{\QQ}( N)=\mathrm{End}_{\QQ}( \QQ^n)\cong\mathrm{M}_n( \mathrm{End}_{\QQ}(\QQ))\cong\mathrm{M}_n(\QQ)$ is a separable $\QQ$-algebra.
\begin{invisible}
The matrix ring $\mathrm{M}_n(\QQ)$ is a separable $\QQ$-algebra with separability idempotent given by $\sum _{i=1}^{n}e_{ij}\otimes e_{ji}$. The isomorphism $\mathrm{End}_{\QQ}( \QQ^n)\cong\mathrm{M}_n( \mathrm{End}_{\QQ}(\QQ))$ is consequence of Proposition 13.2 in Anderson-Fuller's book.
\end{invisible} Therefore by the right version of \cite[Theorem 1(1)]{Su71} (see also Proposition \ref{prop:right-equiv-bimod} below), $\QQ$ is $N$-separable over $\QQ$. Thus the $(\QQ,\QQ\times \ZZ)$-bimodule $M:=\psi_*N=\QQ^n$ is such that $\QQ\times \ZZ$ is $M$-semiseparable but not $M$-separable over $\QQ$. For a direct computation by means of Definition \ref{bimod:semisep}, set $m:=(1,0,\ldots,0)$ and define $f\in M^*=\Hom_{\QQ\times\ZZ}(\QQ^n,\QQ\times \ZZ)$ by $f(q_1,\ldots,q_n):=(q_1,0)$. Then $f\otimes_ {\QQ} m \in (M^*\otimes_{\QQ}M)^{\QQ\times \ZZ}$ and for every $m'\in M$ one has $m'f(m)=m'(1,0)=m'\psi(1,0)=m'.$
\begin{invisible}
  Let us check that $f$ is right ${\QQ\times \ZZ}$-linear: $f((q_1,\ldots,q_n)(q,z))=f((q_1,\ldots,q_n)\psi(q,z))=f((q_1,\ldots,q_n)q)=f(q_1q,\ldots,q_nq)=(q_1q,0)=(q_1,0)(q,z)=f(q_1,\ldots,q_n)(q,z).$
Let us check that $f\otimes_ {\QQ} m \in (M^*\otimes_{\QQ}M)^{\QQ\times \ZZ}$: $(f\otimes_ {\QQ} m)\cdot (q,z)=f\otimes_ {\QQ} m(q,z)=f\otimes_ {\QQ} m\psi(q,z)=f\otimes_ {\QQ} mq=qf\otimes_\QQ m$ while $(q,z)\cdot(f\otimes_ {\QQ} m)=(q,z)\cdot f\otimes_ {\QQ} m$; moreover $((q,z)\cdot f)(q_1,\ldots,q_n)=(q,z)\cdot f(q_1,\ldots,q_n)=(q,z)(q_1,0)=(qq_1,0)=q\cdot (q_1,0)=q\cdot f(q_1,\ldots,q_n)=(qf)(q_1,\ldots,q_n)$ so that $(q,z)\cdot f=qf$ and hence $(f\otimes_ {\QQ} m)(q,z)=(q,z)(f\otimes_ {\QQ} m)$.

\end{invisible}

\end{es}

Next result provides an explicit factorization as a bireflection followed by a separable functor for the coinduction functor $\sigma_*$ attached to an $(R,S)$-bimodule $M$ in case it is semiseparable. By Corollary \ref{cor:fact-birefl}, this factorization amounts to the one given by the coidentifier.

\begin{prop}\label{prop:condcMfact}
Let $M$ be an $(R,S)$-bimodule. The coinduction functor $\sigma_*=\Hom_S(M,-):\mathrm{Mod}\text{-}S\to \mathrm{Mod}\text{-}R$ is semiseparable if and only if there is an $S$-coring $I$ with a grouplike element $z\in I^S$ such that $\sigma_*\cong\tilde{\sigma}_*\circ G_I$ where $\tilde{\sigma}_*:=\Hom^I(M,-):\m^I\to \mathrm{Mod\text{-}}R$ is separable and the induction functor $G_I:=(-)\otimes_S I: \mathrm{Mod}\text{-}S\to \m^I$ is a bireflection. Here $M$ is in $\m^I$ via $\rho_M(m)=m\otimes_S z$.
\end{prop}

\begin{proof}
Assume that $\sigma_*$ is semiseparable. Then, by Theorem \ref{thm:bimod} $S$ is $M$-separable over $R$ through some $c:=\sum_i f_i\otimes_R m_i\in (M^*\otimes_R M)^S$.
Since $\ev_M:M^*\otimes_R M\to S$ is a morphism of $S$-bimodules, then $I:=\mathrm{Im}(\ev_M)$ is an ideal of $S$ with multiplicative identity $z:=\ev_M(c)=\sum_i f_i(m_i)$. Indeed, for all $s\in S$, $z s=\ev_M (c)s=\ev_M (cs)=\ev_M(sc)=s\ev_M(c)=sz $ and hence $z \in I^S$. For all $m\in M,f\in M^*$, we have $z f(m)=\sum_i f(m)f_i(m_i)=f(\sum_i mf_i(m_i))=f(m)$ and hence $z i=i$ for every $i\in I$. Moreover, since the morphism $\varphi:S\to I,s\mapsto sz $, is a ring epimorphism, the map $m_I:I\otimes_S I\to I$ is bijective. Thus we can consider $\Delta_I = m_I^{-1}: I\to I\otimes_{S}I$, $\Delta_I(i)=i\otimes_S z =z \otimes_S i$, so that $(I,\Delta_I,\varepsilon_I)$ becomes an $S$-coring, where $\varepsilon_I:I\hookrightarrow S$ is the canonical inclusion. By the foregoing $z \in I^S$ and, for every $i\in I$, we have $\varepsilon_I(i)z =iz =i.$ By what we recalled at the beginning of Subsection \ref{es:coring}, the induction functor $G_I:=(-)\otimes_S I: \mathrm{Mod}\text{-}S\to \m^I$ is naturally full. Consider its left adjoint, the forgetful functor $F_I:\m^I\to \mathrm{Mod}\text{-}S$, and the corresponding unit $\eta$ defined on each $N$ in $\m^I$ by setting $\eta_N:=\rho_N:N\to N\otimes_S I$. Given $n\in N$ write $\rho_N(n)=\sum_tn_t\otimes_Si_t$. By applying $N\otimes_S \varepsilon_I$ we get $n=\sum_tn_ti_t$. Thus $\rho_N(n)=\sum_tn_t\otimes_Si_t=\sum_tn_t\otimes_Si_tz =\sum_tn_ti_t\otimes_Sz =n\otimes_S z $. We have so proved that $\rho_N(n)=n\otimes_S z $, for every $n\in N$. By applying $N\otimes_S \varepsilon_I$ to this equality we get $n=nz$, for every $n\in N$. Therefore $\rho_N$ is invertible with inverse given by $n\otimes_{S} i\mapsto ni$, and then the unit $\eta$ of $F_I\dashv G_I$ is invertible, i.e. $G_I$ is a coreflection. By Theorem \ref{thm:frobenius}, $G_I$ is a bireflection.
As in \cite[Example 4.3]{GT02}, once noticed that $M\in{^R\m^I}$ (this just means that $\rho_M$ is left $R$-linear), we can consider the functor $\tilde{\sigma}^*:=(-)\otimes_R M:\mathrm{Mod}\text{-} R=\m^R\to \m^I$. By \cite[18.10.2]{BW03} we have that  $\tilde{\sigma}^*\dashv\tilde{\sigma}_*=\Hom^I(M,-)$ with unit and counit given by
\begin{equation*}
\tilde{\eta}_X : X\to \Hom^I(M,X\otimes_R M),\, x\mapsto [m\mapsto x\otimes_R m],\quad\tilde{\epsilon}_Y : \Hom^I(M ,Y)\otimes_R M\to Y,\, f\otimes_R m\mapsto f(m).
\end{equation*}
Thus, by Rafael Theorem, $\tilde{\sigma}_*$ is separable if and only if there is a natural transformation  $\tilde{\gamma}:\id\to \tilde{\sigma}^*\tilde{\sigma}_*$ such that $\tilde{\epsilon}\circ \tilde{\gamma}=\id$. For $Y$ in $\m^I$, define
$
\tilde{\gamma}_Y : Y\to \Hom^I(M ,Y)\otimes_R M,\, y\mapsto \sum_i yf_i(-)\otimes_R m_i.
$
It is easy to check it defines a natural transformation $\tilde{\gamma}:\id\to \tilde{\sigma}^*\tilde{\sigma}_*$. Moreover $\tilde{\epsilon}_Y\tilde{\gamma}_Y(y)=\sum_i yf_i(m_i)=yz $ but we already proved that $yz =y$, hence $\tilde{\epsilon}\circ\tilde{\gamma}=\id$ and $\tilde{\sigma}_*$ is separable.
Let us check that $G\cong\tilde{\sigma}_*\circ G_I$.
Note that $\varphi\circ \varepsilon=\id_I$ and both $\varphi$ and $\varepsilon$ are both left $S$-linear. As a consequence $I$ is projective, whence flat, as a left $S$-module. Thus, by \cite[22.12]{BW03} applied in case $\dd$ is the $S$-coring $S$, for every $N$ in $\mathrm{Mod\text{-} R}$ we have a functorial isomorphism of abelian groups $$\tilde{\sigma}_*G_I(N)=\Hom^I(M,N\otimes_S I)\to \Hom_S(M,N)=\sigma_*(N),\quad f\mapsto (N\otimes_S\varepsilon_I) \circ f.$$ This isomorphism is easily checked to be right $R$-linear. Thus it yields $\tilde{\sigma}_*\circ G_I\cong \sigma_*$ as desired.
Conversely, if $ \sigma_*\cong\tilde{\sigma}_*\circ G_I$, where $G_I$ is a bireflection, whence naturally full by Theorem \ref{thm:frobenius}, and $\tilde{\sigma}_*$ is separable, then $\sigma_*$ is semiseparable in view of Lemma \ref{lem:comp}(ii).
\end{proof}

\begin{claim}\label{claim:bimod}
As already mentioned, given an $(R,S)$-bimodule $M$, in order to characterize the semiseparability of the induction functor $\sigma^*=(-)\otimes_R M: \mathrm{Mod}\text{-}R\to \mathrm{Mod}\text{-}S$ we need, as in the separable case, the further assumption that $M_S$ is finitely generated and projective.
It is well-known that this hypothesis implies that the map $N\otimes_S M^*\to \Hom_S (M,N), n\otimes f\mapsto [m\mapsto nf(m)],$ is an isomorphism natural in $N$, for every right $S$-module $N$, where $M^*=\Hom_S(M,S)$.
\begin{invisible}
Since this is well-known I would not make visible a reference. Anyway a standard one is Anderson-Fullers book, Proposition 20.10.
\end{invisible} As a consequence, the right adjoint of $\sigma^*$ can be chosen to be $\sigma_*=(-)\otimes_S M^*: \mathrm{Mod}\text{-}S\to \mathrm{Mod}\text{-}R$.
If we let $\{e^*_i,e_i\}\subseteq M^*\times M $ be a finite dual basis, the unit $\eta$ and the counit $\epsilon$ of this adjunction are given for all $X\in \mathrm{Mod}\text{-}R$, $Y\in \mathrm{Mod}\text{-}S$ by
\begin{equation*}
\eta_X : X\to X\otimes_R M\otimes_S M^*,\, x\mapsto \sum_i x\otimes_Re_i\otimes_Se^*_i,\quad\epsilon_Y : Y\otimes_S M^*\otimes_R M\to Y,\, y\otimes_R f\otimes_S m\mapsto yf(m).
\end{equation*}
It turns out that, see e.g. \cite[page 30]{Mes06}, the Eilenberg-Moore category $(\mathrm{Mod}\text{-}R)_{\sigma_*\sigma^*}$ is equivalent to the category $\mathrm{Mod}\text{-}\e$, where $\e:=\mathrm{End}_S(M)\cong M\otimes_S M^*$ is the endomorphism ring with canonical morphism $\varphi:R\to\e$, $\varphi(r)(m)=rm$, for all $r\in R$ and $m\in M$. Dually the Eilenberg-Moore category $(\mathrm{Mod}\text{-}S)^{\sigma^*\sigma_*}$ is equivalent to the category $\m^\cc$ of right comodules over the comatrix $S$-coring $\cc:=M^*\otimes_R M$, see e.g. \cite[page 36]{Mes06}. Comatrix corings have been introduced in \cite{ElKGT03} and they generalize the Sweedler's canonical coring. See also \cite{CDGV07} for further investigations. The diagram \eqref{diag:eil-moore} becomes:
\begin{gather*}
\vcenter{%
\xymatrixcolsep{1.8cm}\xymatrixrowsep{1.5cm}\xymatrix{
(\mathrm{Mod}\text{-}S)^{\sigma^*\sigma_*}\cong \m^\cc \ar@{}[r]|-\perp \ar@<1ex>[r]^-{F}& \mathrm{Mod}\text{-}S\ar@/^1pc/[rd]^{K_{\sigma_*\sigma^*}} \ar@<1ex>[l]^-{G=(-)\otimes_R\cc}
\ar@<1ex>[d]^*-<0.1cm>{^{\sigma_*=(-)\otimes_S M^*}}\\
&\mathrm{Mod}\text{-}R \ar@/^1pc/[lu]^{K^{\sigma^*\sigma_*}}  \ar@<1ex>[u]^*-<0.1cm>{^{\sigma^*=(-)\otimes_R M}}\ar@{}[u]|{\dashv}\ar@{}[r]|-\perp \ar@<1ex>[r]^-{\varphi^{*}=(-)\otimes_R\e}&\mathrm{Mod}\text{-}\e \cong (\mathrm{Mod}\text{-}R)_{\sigma_*\sigma^*}\ar@<1ex>[l]^-{\varphi_*}
} }
\end{gather*}
where $F\dashv G$ is the adjunction as in Subsection \ref{es:coring}, given by the forgetful functor $F:\m^\cc\to \mathrm{Mod}\text{-}S$ and the induction functor $G:=(-)\otimes_S \cc :\mathrm{Mod}\text{-}S\to \m^\cc$; the functors $K_{\sigma_*\sigma^*}$ and $K^{\sigma^*\sigma_*}$ are the comparison and the cocomparison functor, respectively. From the diagram, we have $$\varphi_*\circ K_{\sigma_*\sigma^*}=\sigma_* \qquad L\circ K^{\sigma^*\sigma_*}=\sigma^*\qquad K_{\sigma_*\sigma^*}\circ\sigma^*=\varphi^* \qquad K^{\sigma^*\sigma_*}\circ \sigma_*=R.$$
\end{claim}

The above refinement of the functors involved allows us to obtain a different characterization also for the semiseparability of $\sigma_*$

\begin{prop}\label{prop:right-equiv-bimod}
In the setting of Claim \ref{claim:bimod}, the following assertions are equivalent.
\begin{itemize}
\item[(i)] $S$ is $M$-semiseparable over $R$;
\item[(ii)] $\sigma_*=(-)   \otimes_SM^*:\mathrm{Mod}\text{-}S\to\mathrm{Mod}\text{-}R$ is semiseparable;
\item[(iii)] the comatrix $S$-coring $\cc$ is semicosplit;
\item[(iv)] there exists an invariant element $z\in\cc^S$ such that for every $c\in \cc$, $c=\varepsilon_{\cc}(z)c$, where $\varepsilon_{\cc}$ is the counit of the comatrix $S$-coring $\cc$;
\item[(v)] $\varphi_*:\mathrm{Mod}\text{-}\e\to \mathrm{Mod}\text{-}R$ is separable (that is, $\e /R$ is separable) and $K_{\sigma_*\sigma^*}$ is naturally full.
\end{itemize}
\end{prop}
\proof
(i)$\Leftrightarrow$ (ii). It is Theorem \ref{thm:bimod}.\\
(ii)$\Leftrightarrow$ (iii). By Remark \ref{rmk:monad} 3), $\sigma_*$  is semiseparable if and only if so is $V^{\sigma^*\sigma_*}=G$.\\
(iii)$\Leftrightarrow$ (iv). It follows by Theorem \ref{thm:inducoring}.\\
(ii)$\Leftrightarrow$ (v). It follows by Theorem \ref{thm:ssepMonad} applied to the adjunction $(\sigma^*,\sigma_*)$. 
\endproof

We now obtain the announced characterization of the semiseparability of $\sigma^*$.

\begin{prop}\label{prop:left-equiv-bimod}
In the setting of Claim \ref{claim:bimod}, the following assertions are equivalent.
\begin{itemize}
\item[(i)] $\sigma^*=(-)   \otimes_RM:\mathrm{Mod}\text{-}R\to\mathrm{Mod}\text{-}S$ is semiseparable;
\item[(ii)] $\varphi^*=(-)   \otimes_R\e:\mathrm{Mod}\text{-}R\to\mathrm{Mod}\text{-}\e$ is semiseparable;
\item[(iii)] there exists an $E\in {}_{R}\Hom_{R}(\e,R)$ such that $\varphi E(1_\e)=1_\e$;
\item[(iv)] $F:\m^\cc\to \mathrm{Mod}\text{-}S$ is separable (i.e. $\cc$ is coseparable) and $K^{\sigma^*\sigma_*}$ is naturally full.
\end{itemize}
\end{prop}
\proof
(i)$\Leftrightarrow$ (ii). By Remark \ref{rmk:monad} 4), $\sigma^*$ is semiseparable if and only if so is $V_{\sigma_*\sigma^*}=\varphi^*$.\\
(ii)$\Leftrightarrow$ (iii). It follows by Proposition \ref{prop:inducfunc}.\\
(i)$\Leftrightarrow$ (iv). It follows by Theorem \ref{thm:ssep-comonad} applied to the adjunction $(\sigma^*,\sigma_*)$. 
\endproof

\begin{rmk}
At the beginning of this subsection we recalled that the separability of $\sigma^*$ implies the one of $\varphi^*$. The other implication is also true if $M_S$ is finitely generated and projective. Indeed, from $K_{\sigma_*\sigma^*}\circ\sigma^*=\varphi^*$, by Remark \ref{rmk:monad} 4), we get that $\sigma^*$ is separable if and only if $\varphi^*$ is separable.
\end{rmk}

Now, as a particular case of Claim \ref{claim:bimod}, given a morphism of rings $\varphi:R\to S$ consider the $(R,S)$-bimodule $M:={}_R S_S$, with left action induced by $\varphi$, which is trivially finitely generated and projective as a right $S$-module. In this case $\sigma^*=(-)\otimes_RS=\varphi^*: \mathrm{Mod}\text{-}S\to \mathrm{Mod}\text{-}R$ is the induction functor of Section \ref{es:extens} As a consequence, the right adjoint $\sigma_*$ of $\sigma^*$ is isomorphic to the restriction of scalars functor $\varphi_*: \mathrm{Mod}\text{-}S\to \mathrm{Mod}\text{-}R$ and since it is faithful, it follows that $S$ is $S$-semiseparable over $R$ if and only if $S$ is $S$-separable over $R$.

In this case, the comatrix $S$-coring $\cc$ is the Sweedler coring $S\otimes_R S$, we have $\e = \mathrm{End}_S(M)\cong M\otimes_S M^*\cong S$,  $K_{\varphi_*\varphi^*} =\id_{\mathrm{Mod}\text{-}S}$, i.e. $\varphi_*$ is \emph{strictly monadic}, and  $K^{\varphi^*\varphi_*}=(-)\otimes_R S:\mathrm{Mod}\text{-}R\to \mathrm{Mod}\text{-}S$. Consider the induction functor $G=(-)\otimes_S\cc:\mathrm{Mod}\text{-}S\to \m^\cc$ and the forgetful functor $F: \m^\cc\to \mathrm{Mod}\text{-}S$. In this setting, as a consequence of Proposition \ref{prop:right-equiv-bimod} and Proposition \ref{prop:left-equiv-bimod} we have the following corollaries, relating the functors $\varphi_*$, $\varphi^*$, $F$, $G$ and the Sweedler coring $\cc$. We just point out that, since the coring counit $\varepsilon_{\cc}$ is the multiplication $S\otimes_R S\to S$ and we can choose $c=1_S\otimes_R 1_S\in\cc$, the existence of $z\in\cc^S$ such that $c=\varepsilon_{\cc}(z)c$, for every $c\in \cc$, is equivalent to the existence of $z\in\cc^S$ such that $1_S=\varepsilon_{\cc}(z)$ i.e. of a separability idempotent of $S/R$.

\begin{cor}\label{cor:sweed1}
In the above setting, the following assertions are equivalent.
\begin{itemize}
\item[(i)] $S$ is $S$-separable over $R$;
\item[(ii)] $\varphi_*: \mathrm{Mod}\text{-}S\to \mathrm{Mod}\text{-}R$ is separable, i.e. $S/R$ is separable;
\item[(iii)] the Sweedler $S$-coring $S\otimes_R S$ is semicosplit;
\item[(iv)] $S/R$ has a separability idempotent.
\end{itemize}
\end{cor}

\begin{cor}\label{cor:sweed2}
In the above setting, the following assertions are equivalent.
\begin{itemize}
\item[(i)] $\varphi^*$ is semiseparable;
\item[(ii)] there exists an $E\in {}_{R}\Hom_{R}(S,R)$ such that $\varphi E(1_S)=1_S$;
\item[(iii)] $F$ is separable (i.e. the Sweedler $S$-coring $S\otimes_R S$ is coseparable) and $K^{\varphi^*\varphi_*}$ is a bireflection.
\end{itemize}
\end{cor}

Clearly the equivalence between (i),(ii) and (iv) of Corollary \ref{cor:sweed1} above is well-known while the equivalence between (i) and (ii) of Corollary \ref{cor:sweed2} is just Proposition \ref{prop:inducfunc}.

\subsection{Right Hopf algebras}\label{sub:rightHopf}

Let $B$ be a bialgebra over a field $\Bbbk$, let $\mathfrak{M}$ denote the category of vector spaces over $\Bbbk$ and let $\mathfrak{M}_{B}^{B} $ denote the category of right Hopf modules over $B$. Consider the coinvariant functor $\left( -\right)
^{\co B}:\mathfrak{M}_{B}^{B}\rightarrow \mathfrak{M}$ which, for every object $M$ in $\mathfrak{M}%
_{B}^{B}$, is defined by setting $M^{\co B}:=\{m\in M\mid\rho_M(m)=m\otimes 1_B\}$. It is known that it fits into an adjoint triple $\overline{\left( -\right) }^{B}\dashv \left( -\right) \otimes B\dashv \left( -\right)
^{\co B}$, see e.g. \cite[Section 3]{Sar21}, where $\overline{M}^{B}=\frac{M}{MB^+}$ and $B^+=\ker(\varepsilon_B)$. The unit and counit are given by
\begin{gather*}
\eta _{M}:M\rightarrow \overline{M}^{B}\otimes
B,\; m\mapsto \sum \overline{m_{0}}\otimes m_{1},\qquad \epsilon _{V}:\overline{\left( V\otimes B\right) }^{B}\overset{%
\cong }{\rightarrow }V,\;\overline{v\otimes b}\mapsto v\varepsilon _{B}\left(
b\right)
\\
  \nu _{V}:V\overset{\cong }{\rightarrow }\left( V\otimes B\right) ^{\co %
B},\;v\mapsto v\otimes 1_{B}, \qquad \theta
_{M}:M^{\co B}\otimes B\rightarrow M,\;m\otimes b\mapsto mb.
\end{gather*}
By Proposition \ref{prop:adj-triples}, the functor $\left( -\right)
^{\co B}$ is semiseparable (resp. separable, naturally full) if and only if so is $\overline{%
\left( -\right) }^{B}$. Moreover, by \cite[Proposition 3.4.1]{Bor94}, the functor $%
\left( -\right) \otimes B$ is fully faithful so that $\left( -\right) ^{\co B}$ is a coreflection. Thus, by Theorem \ref{thm:frobenius} it follows that $\left( -\right) ^{\co B}$ is semiseparable, if and only if it is naturally full, if and only if it is Frobenius. Our aim here is to characterize the semiseparability of $\left( -\right)
^{\co B}$. Note that there is a natural transformation $\sigma:(-)^{\co B}\to \overline{(-)}^{B}$ defined on components by $\sigma_M:M^{\co B}\to \overline{M}^{B},m\mapsto \overline{m}:=m+MB^+$, see \cite[Section 3]{Sar21}.

$\left( 1\right) \Leftrightarrow \left( 2\right) $ in the
following result is a semi-analogue of $\left( 1\right) \Leftrightarrow
\left( 6\right) $ in \cite[Theorem 3.13]{Sar21}.

\begin{thm}\label{thm:rightHopf}
Let $B$ be a bialgebra over a field $\Bbbk$ and consider the coinvariant functor $\left( -\right)
^{\co B}:\mathfrak{M}_{B}^{B}\rightarrow \mathfrak{M}$. The following assertions
are equivalent.

\begin{enumerate}

\item[$\left( 1\right) $] $\left( -\right) ^{\co B}$ is
semiseparable.


\item[$\left( 2\right) $] $B$ is a right Hopf algebra with
anti-multiplicative and anti-comultiplicative right antipode.

\item[$\left( 3\right) $] The canonical natural transformation $\sigma:(-)^{\co B}\to \overline{(-)}^{B}$ is invertible.

\item[$\left( 4\right) $] The canonical natural transformation $\sigma:(-)^{\co B}\to \overline{(-)}^{B}$ is split-mono.
\end{enumerate}
\end{thm}

\begin{proof}


$\left( 1\right) \Leftrightarrow \left( 2\right)$. We already noticed that $\left( -\right) ^{\co B}$ is
semiseparable if and only if it is Frobenius. Moreover $\left( -\right)^{\co B}$ is  Frobenius if and only if the natural transformation $\sigma$ is invertible, c.f. \cite[Lemma 2.3]{Sar21} applied to the adjoint triple $\overline{\left( -\right) }%
^{B}\dashv \left( -\right) \otimes B\dashv \left( -\right) ^{\co B}$.

$\left( 2\right) \Leftrightarrow \left( 3\right)$. The equivalence follows by \cite[Theorem 3.7]{Sar21}.

$\left( 1\right) \Leftrightarrow \left( 3\right) \Leftrightarrow \left( 4\right) $. It follows from Proposition \ref{prop:sigma}.
\end{proof}

\begin{rmk}
As mentioned, the functor $\left( -\right) ^{\co B}:\mathfrak{M}%
_{B}^{B}\rightarrow \mathfrak{M}$ fits into an
adjoint triple $\overline{\left( -\right) }^{B}\dashv \left( -\right)
\otimes B\dashv \left( -\right) ^{\co B}$. Thus, $\left( -\right) ^{%
\co B}$ is Frobenius if and only if $\left( -\right) ^{\co %
B}\dashv \left( -\right) \otimes B$, if and only if $\overline{\left(
-\right) }^{B}\cong \left( -\right) ^{\co B}$. Note that there are
bialgebras $B$ which are not right Hopf algebras and hence $\left( -\right)
^{\co B}$ needs not to be a Frobenius functor in general. For instance let $G$
be a monoid and consider the monoid algebra $B=\Bbbk G$ over a field $\Bbbk$. If $B$ is a
right Hopf algebra, then it has a right antipode $S_{B}:B\rightarrow B$ and
hence, for every $x\in G$, one has $xS_{B}\left( x\right) =\sum
x_{(1)}S_{B}\left( x_{(2)}\right) =\varepsilon _{B}\left( x\right) 1_{B}=1_{G}$.
In particular each element in $G$ is right invertible and hence $G$ must be
a group, which is not always the case.
\begin{invisible}
If for every $x$ there is $x^{\prime }$ such that $xx^{\prime }=1,$
then $x^{\prime }\left( x^{\prime }\right) ^{\prime }=1$ and hence $%
x^{\prime }$ has both a left and a right inverse i.e. it is invertible. Thus, from $xx^{\prime }=1,$ we deduce that $x=\left( x^{\prime }\right) ^{-1}.$
\end{invisible}
Moreover, see \cite[Example 3.9]{Sar21}, there are bialgebras $B$ satisfying the equivalent conditions of Theorem \ref{thm:rightHopf} that are not Hopf algebras, i.e. such that the coreflection $\left( -\right) ^{\co B}$ is semiseparable but not separable. Indeed, $B$ is a Hopf algebra if and only if $\left( -\right) ^{\co B}$ is an equivalence if and only if it is separable, cf. Remark \ref{rmk:sepcoref}.
\end{rmk}

\subsection{Examples of (co)reflections}\label{appendix:corefl}
The connections between some type of functors we have considered in this paper are summarized in the following diagrams.
\begin{equation}\begin{split}
{ \includegraphics[scale=.6]{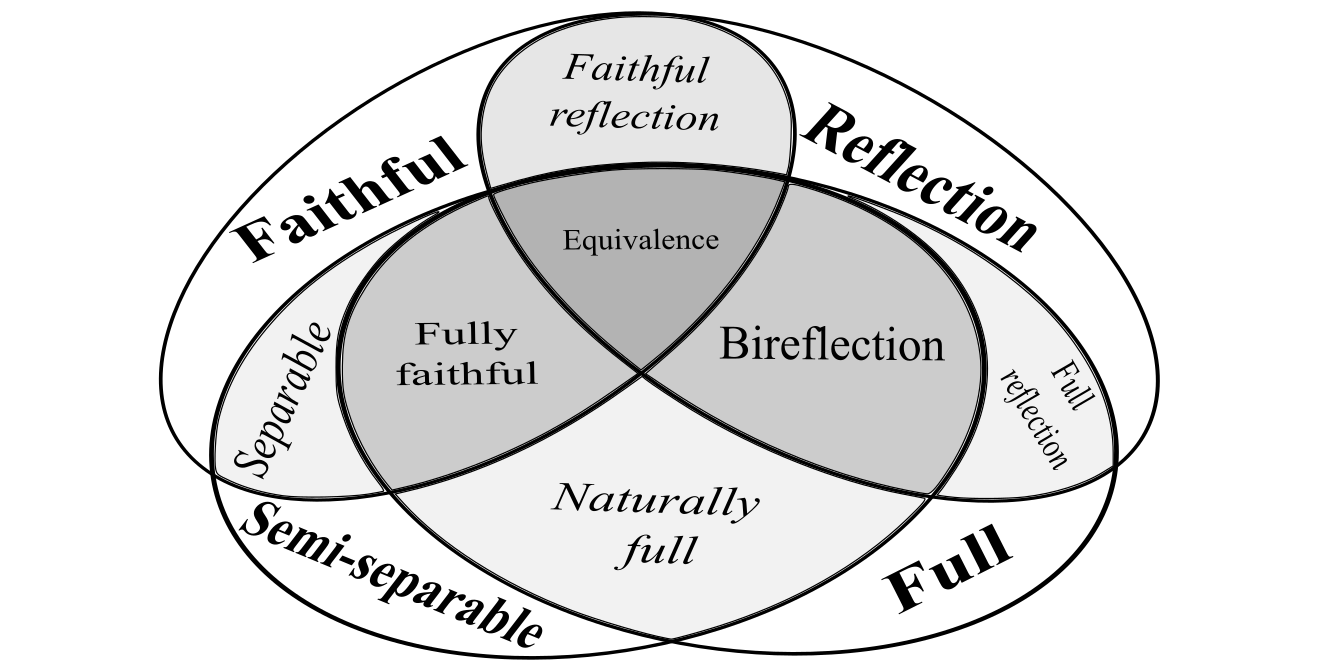}  \includegraphics[scale=.6]{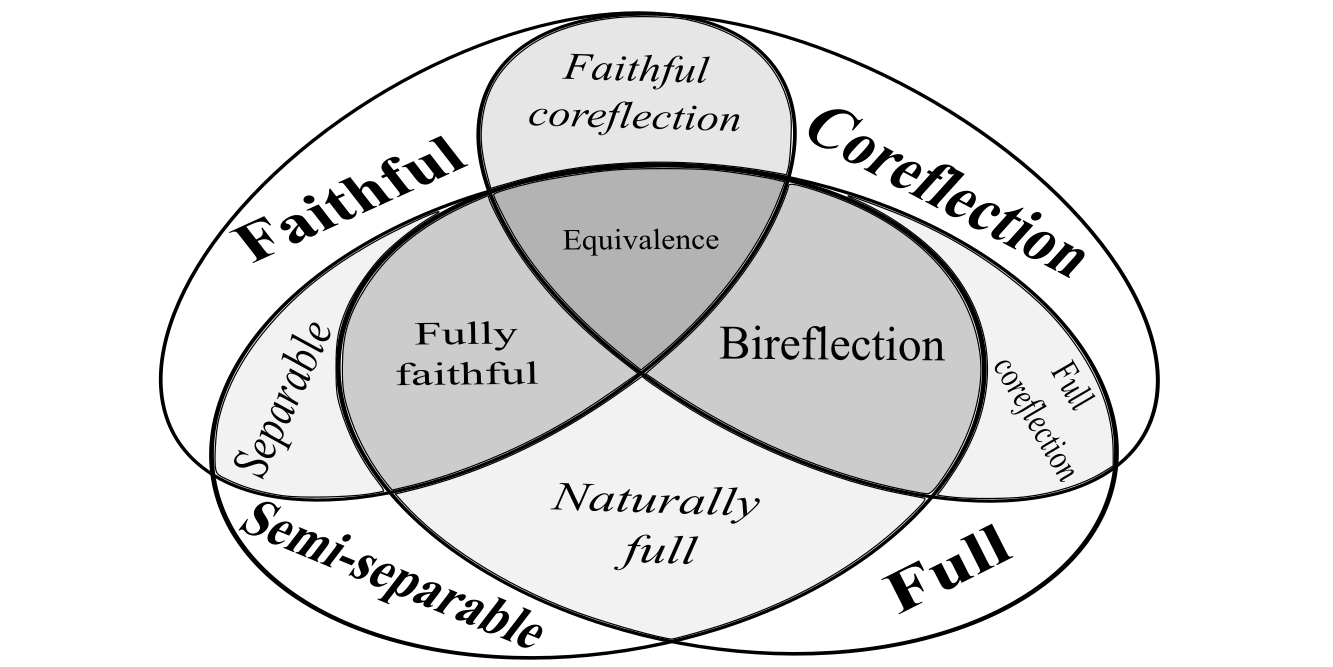}\label{diag:eulerovenn2} }
 \end{split}
  \end{equation}
  Indeed we know that separable functors are the ``intersection'' of semiseparable and faithful functors and that naturally full functors are the ``intersection'' of semiseparable and full functors, see Proposition \ref{prop:sep}. At the very beginning, we observed that a functor is fully faithful if and only if it is at the same time separable and naturally full. We know that a (co)reflection is semiseparable if and only if it is naturally full if and only if it is a bireflection, see Theorem \ref{thm:frobenius}. We also observed that a separable (co)reflection is necessarily an equivalence, see Remark \ref{rmk:sepcoref}.

In order to completely justify the coherence of   diagrams \eqref{diag:eulerovenn2} we need examples of a
\begin{itemize}
  \item (co)reflection which is neither full nor faithful;
  \item faithful (co)reflection which is neither semiseparable nor full;
  \item full (co)reflection which is neither semiseparable nor faithful.
\end{itemize}
Now, we have observed in Remark \ref{rmk:opposemi}, that a functor $F$ is semiseparable (resp. separable, naturally full, full, faithful, fully faithful) if and only if so is $F^\op$. On the other hand, since the opposite switches the functors of an adjunction, one has that $F$ is a reflection (resp. coreflection) if and only if $F^\op$ is a coreflection (resp. reflection). As a consequence we can focus on coreflections as the corresponding examples for reflections can be obtained by duality.

Moreover, a fully faithful coreflection is an equivalence, whence in particular semiseparable. Thus we can omit the last option in the second and in the third item above. Moreover, by Theorem \ref{thm:frobenius}, we know that a coreflection is
semiseparable if and only if it is naturally full if and only if it is a bireflection if and only if it is
Frobenius. Thus a faithful coreflection, which is also semiseparable, must be full whence an equivalence. Summing up, our problem reduces in finding  examples of a
\begin{itemize}
  \item coreflection which is neither full nor faithful;
  \item faithful coreflection which is not an equivalence;
  \item full coreflection which is not a bireflection.
\end{itemize}
\medskip

We start by including an example of coreflection which is neither full nor faithful.

\begin{es}
Let $\Bbbk $ be a field, let $\Coalg$ be the category of coalgebras
over $\Bbbk $ and let $\Set$ be the category of sets. The functor $G:%
\Coalg\rightarrow \Set$ that associates to a coalgebra $C$
the set $G\left( C\right) $ of grouplike elements in $C$, is a coreflection.
In fact it has a fully faithful left adjoint $F$ that takes a set $S$ to the
group-like coalgebra $\Bbbk S.$ The unit and counit components are the
canonical bijection $\eta _{S}:S\rightarrow GFS=G\left( \Bbbk S\right) \ $%
and the canonical injection $\epsilon _{C}:FGC=\Bbbk G\left( C\right) \hookrightarrow C$, respectively.
Let us check that $G$ is not full. Let $D$ be the matrix coalgebra $%
M_{2}^{c}\left( \Bbbk \right) .$ Note that $GD=G\left( M_{2}^{c}\left( \Bbbk
\right) \right) =\emptyset $ which is the initial object in $\Set$.

\begin{invisible}
For $n>1$ we have $G\left( M_{n}^{c}\left( \Bbbk \right) \right) \cong
G\left( M_{n}\left( \Bbbk \right) ^{\ast }\right) =\Alg\left(
M_{n}\left( \Bbbk \right) ,\Bbbk \right) =\left\{ 0\right\} .$ In fact, if $%
f:M_{n}\left( \Bbbk \right) \rightarrow \Bbbk $ is an algebra map. then $%
\mathrm{Ker}\left( f\right) $ is an ideal in $M_{n}\left( \Bbbk \right) $
and hence there is an ideal $I$ of $\Bbbk $ such that $\mathrm{Ker}\left(
f\right) =M_{n}\left( I\right) $. Since $\Bbbk $ is a field, then $I$ is
either $0$ or $\Bbbk $. If $I=0$ then $\mathrm{Ker}\left( f\right)
=M_{n}\left( 0\right) =0$ and hence $f$ is injective. Thus $M_{n}\left(
\Bbbk \right) $ embeds in $\Bbbk $ and hence it is commutative but this is
impossible as $n>1,$ unless $\Bbbk $ is trivial. Thus $I=\Bbbk $ and hence $%
\mathrm{Ker}\left( f\right) =M_{n}\left( \Bbbk \right) $ so that $f$ is the
zero map.
\end{invisible}
Note also that, if we denote by $0$ the zero coalgebra, then we also have $%
G0=\emptyset $ so that $GD=G0.$ If $G$ is full then there is a coalgebra map
$f:D\rightarrow 0$. In particular we have $\varepsilon _{D}=\varepsilon
_{0}\circ f=0,$ a contradiction as $\varepsilon _{D}$ is the map that
assigns to a matrix its trace. Thus $G$ is not full. Let us check it is not even faithful. Otherwise, $\epsilon_C$ would be an epimorphism in $\Coalg$, for every coalgebra $C$, but, by \cite[Theorem 3.1]{MT94}, an epimorphism in $\Coalg$ is necessarily surjective whence $\epsilon_C$  would be invertible and hence every coalgebra $C$ would be isomorphic to $\Bbbk G\left( C\right)$, a contradiction.
\end{es}

\begin{rmk}
  We already observed that a conservative (co)reflection is always an equivalence. It is known (see e.g. \cite[A1.2]{Gr13}) that a faithful functor from a balanced category (i.e. a category where every monomorphism which is an epimorphism is necessarily an isomorphism) is always conservative. As a consequence a faithful (co)reflection from a balanced category is always an equivalence. Since the category $R$-Mod of left modules over a ring $R$ is abelian, it is in particular balanced and hence any faithful (co)reflection from $R$-Mod is always an equivalence.
\end{rmk}

An instance of a faithful coreflection which is not an equivalence is obtained by duality from the following example. Other examples arise as full epi-coreflective subcategories, see \cite{He71}.

\begin{es}
Let $\Dom$ be the category of integral domains and injective ring homomorphisms. Let $\Field$ be the category of fields. The forgetful functor $G:\Field\to \Dom$ has a left adjoint $F$ that takes every integral domain $D$ to its quotient field $Q(D).$ Given an integral  domain $D$, denote by $j_D :D\to Q(D),d\mapsto\frac d 1$, the canonical injection. Then the unit on an integral domain $D$ is the injection $\eta_D=j_D:D\to GFD=Q(D)$, while the counit is the isomorphism $\epsilon_F=j_K^{-1}:FGK=Q(K)\to K$. Note that, in general, $\eta$ is just a monomorphism but not an isomorphism on components. Thus, $F$ is a faithful reflection which is not an equivalence.
\end{es}

We finally provide an example of a full coreflection which is not a bireflection.
	
\begin{es}\label{es:fullcoref}
Let $\Bbbk $ be an arbitrary field, let $\Coalg_\bullet$ be the full subcategory of $\Coalg$ whose objects are pointed coalgebras
over $\Bbbk $ and let $\Set$ be the category of sets. The functor $G:%
\Coalg_\bullet\rightarrow \Set,$ that associates to a coalgebra $C$
the set $G\left( C\right) $ of grouplike elements in $C$, is a coreflection.
In fact, it has a fully faithful left adjoint $F$ that takes a set $S$ to the
group-like coalgebra $\Bbbk S.$ The unit and counit components are the
canonical bijection $\eta _{S}:S\rightarrow GFS=G\left( \Bbbk S\right) \ $%
and the canonical injection $\epsilon _{C}:FGC=\Bbbk G\left( C\right) \hookrightarrow C$, respectively.
By the dual Wedderburn-Malcev Theorem \cite[Theorem 5.4.2]{Mo93}, since $C$ is pointed, there exists a coalgebra projection $\pi:C\to C_0=\Bbbk G\left( C\right)$ such that $\pi\circ \epsilon _{C}=\id$. Thus $\epsilon _{C}$ is a split monomorphism for each pointed coalgebra $C$ and hence $G$ is full. Therefore $G$ is a full coreflection. Let us check it is not a bireflection in general. Otherwise $G$  would be Frobenius and hence from $F\dashv G$ we should deduce $G\dashv F$. Consider the Sweedler's $4$-dimensional Hopf algebra $H=\Bbbk\langle g,x\mid g^2=1,x^2=0,gx+xg=0\rangle $ with coalgebra structure given by $\Delta(g)=g\otimes g$ and $\Delta(x)=x\otimes 1+g\otimes x$ and set $S:=G(H)=\{1,g\}$. We have
 $$\Hom_{\Set}(S,S)=\Hom_{\Set}(G(H),S)\cong\Hom_{\Coalg_\bullet}(H,FS)
 =\Hom_{\Coalg}(H,\Bbbk S).$$ Since $\Hom_{\Set}(S,S)$ has cardinality $4$, we get that $\Hom_{\Coalg}(H,\Bbbk S)$ must contain exactly $4$ elements. For every $k\in\Bbbk$ define $f_k:H\to \Bbbk S$ by setting $f_k(1)=1,f_k(g)=g,f_k(x)=k(1-g)=f_k(xg)$. Then $f_k$ is a coalgebra map.
\begin{invisible}Let us check that $f_k$ is comultiplicative:
\begin{align*}
  (f_k\otimes f_k)\Delta(1) & = f_k(1)\otimes f_k(1)=1\otimes 1=\Delta (1)=\Delta(f_k(1)),\\
(f_k\otimes f_k)\Delta(g) & =f_k(g)\otimes f_k(g)=g\otimes g=\Delta (g)=\Delta(f_k(g)), \\
(f_k\otimes f_k)\Delta(x) & = f_k(x)\otimes f_k(1)+f_k(g)\otimes f_k(x)= k(1-g)\otimes 1+g\otimes k(1-g)\\
&= k1\otimes 1-kg\otimes 1+kg\otimes 1-kg\otimes g= k1\otimes 1-kg\otimes g=\Delta (k1-kg)=\Delta(f_k(x))\\
(f_k\otimes f_k)\Delta(xg) & = f_k(xg)\otimes f_k(g)+f_k(1)\otimes f_k(xg)= k(1-g)\otimes g+1\otimes k(1-g),\\
&= k1\otimes g-kg\otimes g+k1\otimes 1-k1\otimes g= k1\otimes 1-kg\otimes g=\Delta (k1-kg)=\Delta(f_k(xg)).\\
\end{align*}
Let us check that $f_k$ is counitary:
\begin{align*}
  \varepsilon_{\Bbbk S}(f_k(1)) & = \varepsilon_{\Bbbk S}(1)=1= \varepsilon_{H}(1),\\
\varepsilon_{\Bbbk S}(f_k(g)) & =  \varepsilon_{\Bbbk S}(g)=1= \varepsilon_{H}(g),\\
\varepsilon_{\Bbbk S}(f_k(x)) & =  \varepsilon_{\Bbbk S}(k(1-g))=k \varepsilon_{\Bbbk S}(1)-k \varepsilon_{\Bbbk S}(g)=0= \varepsilon_{H}(x),\\
\varepsilon_{\Bbbk S}(f_k(xg)) & = \varepsilon_{\Bbbk S}(k(1-g))=k \varepsilon_{\Bbbk S}(1)-k \varepsilon_{\Bbbk S}(g)=0= \varepsilon_{H}(xg).\\
\end{align*}
\end{invisible}
By linear independence of grouplike elements, we have that $f_k\neq f_l$ for every $k,l\in\Bbbk$ such that $k\neq l$. Since $\Hom_{\Coalg}(H,\Bbbk S)$ contains $4$ elements we deduce that the field $\Bbbk$ has at most $4$ elements,  a contradiction.
\end{es}

\end{document}